\documentclass[10pt,leqno]{article}

\usepackage{amsmath,amssymb,amsthm,mathrsfs,dsfont}

\usepackage[margin=3cm]{geometry} 

\usepackage{titlesec,hyperref}

\usepackage{color}

\usepackage{fancyhdr}
\usepackage{subfigure}
\usepackage[graphicx]{realboxes}
\usepackage{float}
\usepackage{enumerate}
\usepackage{enumitem}
\setlist[enumerate]{label=(\arabic*), leftmargin=*, itemsep=2pt}
\usepackage{dsfont} 
\usepackage{extarrows}

\titleformat{\subsection}{\it}{\thesubsection.\enspace}{1pt}{}

\newtheorem{theo}{Theorem}[section]
\newtheorem{lemm}[theo]{Lemma}
\newtheorem{defi}[theo]{Definition}

\newtheorem{prop}[theo]{Proposition}
\newtheorem{rema}[theo]{Remark}

\allowdisplaybreaks 

\numberwithin{equation}{section}

\begin{document}
\title{Global Existence of Weak Martingale Solutions to the Camassa-Holm Equation with Linear Multiplicative Noise
	\hspace{-4mm}
}

\author{
	Wei $\mbox{Luo}^1$ \footnote{E-mail:  luow23@mail.sysu.edu.cn}, \quad
	Zhaoyang $\mbox{Yin}^{1,2}$\footnote{E-mail: mcsyzy@mail.sysu.edu.cn} \quad and
	Pei $\mbox{Zheng}^{1}$\footnote{E-mail: zhengp25@mail2.sysu.edu.cn}\\
	$^1\mbox{Department}$ of Mathematics,
	Sun Yat-sen University, Guangzhou 510275, China\\
	$^2\mbox{School}$ of Science,\\ Shenzhen Campus of Sun Yat-sen University, Shenzhen 518107, China}

\date{}
\maketitle
\hrule

\begin{abstract}
	In this paper, we consider the global existence and properties of $H^1$ martingale solution to the Camassa–Holm equation with linear multiplicative noise under periodic boundary conditions. The solution is obtained as limit of regular viscous approximate solutions to parabolic SPDEs, which are constructed using the Galerkin approximations ans the stochastic compactness method. The proof of convergence to a solution argues via tightness of the laws of the viscous approximations and Skorokhod-Jakubowski a.s. representations of random variables in quasi-Polish spaces. In particular, by means of the Girsanov-type transform for regular viscous approximations and the convergence of Skorokhod–Jakubowski representations, we are able to establish the one-sided supernorm estimate and space-time higher regularity of the first-order spatial derivative, and large-time behavior of the weak martingale solution in the stochastic framework.
	
	\vspace*{5pt}
					
	\vspace*{5pt}
	
%
%
\end{abstract}

\vspace*{10pt}

\tableofcontents

\section{Introduction}\label{1}
\subsection{Background and main result}
\quad In recent year, the Korteweg-de Vries (KdV) equation, which was introduced to describe the behavior of long waves on shallow water, is the most famous model in soliton theory because it is integrable and includes the phenomena of soliton interaction. But, the KdV equation cannot model the occurrence of breaking waves. In 1993, Camassa and Holm obtained the nonlinear partial differential equation \cite{CH}
$$
	m_{t}+um_{x}+2u_{x}m=0, m=u-u_{xx}+\kappa
$$
which is named the Camassa-Holm (CH) equation, has been well studied and a series of achievements have been made. A particular feature of the CH equation is that when $\kappa=0$ 
$$
	u_t-u_{xxt}+3u\,u_x=2u_x\,u_{xx}+u\,u_{xxx}
$$
it admits peaked soliton solutions which are also called peakons. It can be regarded as a shallow water wave equation with nonhydrostatic pressure \cite{CH,CH2,CH3}
$$
	\left\{
	\begin{aligned}
		&u_t+u\,u_x+P_x=0,\\
		&P-P_{xx}=u^2+\frac 1 2u_x^2
	\end{aligned}
	\right.
$$

The bi-Hamiltonian structure of CH equation was studied in \cite{CH6,CH7}, it ensures infinite conversation laws which was obtained in \cite {CH}, and the complete integrability was discussed in \cite{CH,CH4,CH5}. Moreover, the CH equation is such an equation that exhibits both phenomena of (peaked) soliton interaction and wave breaking (the solution remains bounded but its slope becomes unbounded in finite time; cf. \cite{CH12}.)

Constantin and Escher \cite{CH13,CH14} investigated the Cauchy problem for the periodic Camassa–Holm equation. The wave breaking of the Cauchy problem for Cauchy problem was studied in \cite{CH14,CH5,CH15}, and in \cite{CH16} Constantin claimed that wave breaking is the only way that singularities can occur in solutions. In particular, the CH equation admits exact orbitally stable peaked solitons of the form \cite{CH17}
$$u(t,x)=ce^{|x-ct|},\quad x\in\mathbb{R},\ c>0.$$ 

In recent year, the local well-posedness of Cauchy problem of the deterministic CH equation in Besov spaces and Sobolev spaces was proved in \cite{CH8,CH9,CH10,CH11}. According to the local well-posedness in Besov and Sobolev spaces, in Constantin and Molinet's work \cite{Constantin1999} in 1999, the existence and uniqueness of the global weak solution of energy conservation in whole space and periodic case is proved, and also the stability of solitons is given. Xin and Zhang also proved the existence and uniqueness of CH eqaution in \cite{Xin2000} by viscous approximation, and claimed the one-sided supernorm esimate and space-time integrability estimate in $L^p_{\rm loc}(\mathbb{R}^+\times\mathbb{R})$ for $p<3$ of the weak solutions, additionally, the large-time behavior of the weak solution is given. In \cite{Xin2002}, Xin and Zhang obtained a "weak=strong" theorem for the admissible weak solutions and the uniqueness of weak solution under the condition that the initial data $m_0=(1-\partial_x^2)u_0$ is a positive Radon measure.

However, due to the uncertainties in geophysical and climate dynamics \cite{Arnold2001,Holm2015}, we have to consider some influence of internal, external, or environmental noises. Besides, the whole background for the considered physical system may be difficult to describe deterministically. Thus, we consider the randomness of the background movement which is one of the prevailing hypotheses on the onset of turbulence in fluid models, and there is a lot of recent work done on PDEs with random perturbations \cite{SPDE1,SPDE2,Debussche2011,SPDE3,SPDE4,SPDE5,SPDE6,SPDE7}.

The stochastic CH equation was derived via the stochastic variational method in \cite{Holm2015,Holm2016}. Consequently, the well-posedness, uniqueness and blow up phenomena for the stochastic Camassa-Holm type equations with perturbation is currently a interesting topic in the field of physics and mathematics. Chen and Gao established the existence of stochastic CH equation with additive noise in $H^m$ with $m>3/2$ in \cite{Chen2012}, and the existence of a pathwise solution to a modified CH equation with deterministic initial data and linear multiplicative noise in \cite{Chen2016}. In \cite{Tang2018}, Tang proved the local existence and pathwise uniqueness of pathwise solution in Sobolev space $H^m$ with $m>3/2$, and the condition lead to the global existence and the blow-up for the linear noise case. In particular, Tang studied the pathwise dissipative effect of the linear noise on the periodic peakons to the deterministic CH eqaution. The well-posedness for a generalized CH equation with higher order nonlinearities under random perturbation was studied by Miao, Rohde and Tang in \cite{Miao2024}. In \cite{Holden2023}, Holden et al. proved the existence of a global martingale solution in $H^1$ for the stochastic CH equation with a viscous term under gradient noise.

Based on the well-posedness theory established above, the global $H^1$ solution of the aforementioned stochastic CH-type equation has attracted research interest. Chen, Duan and Gao established the existence of global $H^1$ martingale solutions of the stochastic CH equation with linear multiplicative noise under the initial condition $u_0\in H^1$ and $m_0=(1-\partial_x^2)u_0\in\mathcal{M}^+$, see \cite{Chen2021}. In \cite{Holden2024JDE}, Holden, Karlsen and Pang proved the global existence of dissipative solutions to CH equation with transport noise.

In this paper, we are interested in global weak martingale solution of the initial-value problem for the stochastic CH equation with linear multiplicative noise:
\begin{equation}\label{CH u}
	\left\{
	\begin{array}{l}
		{\rm d}u+(u\partial_xu+\partial_xP[u])\,{\rm d}t=\beta u\,{\rm d}W,\\
		u(0,x)=u_0(x),\\
		P[u]=(1-\partial_x^2)^{-1}\left(u^2+\frac 1 2(\partial_x u)^2\right)
	\end{array}
	\right.
\end{equation}
for $(t,x)\in [0,T]\times\mathbb{S}$, where $\mathbb{S}=\mathbb{S}/(2\pi\mathbb{Z})$ is the 1D torus, $T$ is a positive final time and $W$ is a 1D Wiener process defined on a standard filtered probability space $\mathcal{S}=(\Omega,\mathcal{F},\{\mathcal{F}_t\}_{t\in [0,T]},\mathbb{p})$, henceforth called a stochastic basis. Moreover, the elliptic equation for $P$ can be solved to supply
\begin{equation}\label{P convolution}
	P=P[u]:=K\ast\left(u^2+\frac 1 2(\partial_x u)^2\right),\quad K(x)=\frac{\cosh(x-2\pi\,{\rm int}(\frac{x}{2\pi})-\pi)}{2\sinh\,\pi},
\end{equation}
where $K$ is the Green's function of $1-\partial_x^2$ on $\mathbb{S}$, ${\rm int}(x)$ is the integer part of $x$, and $\ast$ means convolution in $x$.
Denote $m=(1-\partial_x^2)u$, then $m$ satisfies
\begin{equation}\label{CH m}
	\left\{
	\begin{array}{l}
		{\rm d}m+(u\partial_xm+2m\partial_xu)\,{\rm d}t=\beta m\,{\rm d}W,\\
		m(0,x)=m_0(x).\\
	\end{array}
	\right.
\end{equation}
For the sake of simplicity, in the rest of the paper, we set $\beta=1$.

Our work proves the existence of global weak martingale solution without assuming $m_0=(1-\partial_x^2)u_0\in\mathcal{M}^+$ for the initial data, and demonstrates that the corresponding properties of the deterministic CH equation’s admissible weak solutions, originally established by Xin and Zhang in \cite{Xin2000}, are preserved in the stochastic framework.

Our main result is the following theorem:

\begin{theo}[Existence and properties of $H^1$ martingale solution]\label{main theorem}
	Fix some $p_0>4$, for any initial probability distribution $\Lambda$ supported on $H^1(\mathbb{S})$,satisfying
	$$
	\int_{H^1 (\mathbb{S})}\Vert v\Vert_{H^1(\mathbb{S})}^{p_0}\,\Lambda({\rm d}v)<\infty,
	$$
	then there exists a weak martingale solution $(\tilde{\mathcal{S}},\tilde{u},\tilde{W})$ to the stochastic CH equation \eqref{CH u} with random initial data $\tilde{u}_0$ distributed according to $\Lambda$, where $\tilde{\mathcal{S}}=(\tilde{\Omega},\tilde{\mathcal{F}},\{\tilde{\mathcal{F}_t}\}_t\in[0,T],\tilde{\mathbb{P}})$ is a stochastic basis. Besides, the following energy inequality holds $\tilde{\mathbb{P}}$-a.s., for a.e. $s\in[0,T)$ and every $t$ with $s<t\le T$,
	\begin{equation}\label{energy inequality}
		\int_\mathbb{S}\tilde{u}^2+|\partial_x\tilde{u}|^2\,{\rm d}x\big|_s^t\le\int_s^t\int_\mathbb{S}\tilde{u}^2+|\partial_x\tilde{u}|^2\,{\rm d}x\,{\rm d}t+2\int_s^t\int_\mathbb{S}\tilde{u}^2+|\partial_x\tilde{u}|^2\,{\rm d}x\,{\rm d}\tilde{W}.
	\end{equation}
	In particular, it holds for $s=0$ and any $t\in(0,T]$, with $\int_{\mathbb{S}}\tilde{u}^2+|\partial_x\tilde{u}|^2\,{\rm d}x\big|_{s=0}$ replaced by $\int_{\mathbb{S}}\tilde{u}_0^2+|\partial_x\tilde{u}_0|^2\,{\rm d}x$.
	
	Moreover, the weak martingale solution $(\tilde{\mathcal{S}},\tilde{u},\tilde{W})$ satisfies the following properties:
	\begin{enumerate}
		\item One-sided supernorm estimate: The following one-sided $L^\infty$ norm estimate on the fisrt-order spatial derivative holds in the sense of distribution a.s.:
		\begin{equation}\label{main property one side}
			\tilde{\mathbb{P}}\left\{\lim_{t\rightarrow\infty}\partial_x \tilde{u}\le0,\ \forall x\in\mathbb{S}\right\}=1.
		\end{equation}
		\item Space-time higher integrability estimate: For any $r<3$, there exists a positive constant $C$ independent of $r$ such that
		\begin{equation}\label{main property space time higher}
			\tilde{E}\Vert \partial_x\tilde{u}\Vert_{L^r([0,T]\times\mathbb{S})}^r\le C.
		\end{equation}
		\item Large-time behavior: The weak martingale solution $\tilde{u}$ approaches $0$ pointwise almost surely as time goes to infinity, i.e.,
		\begin{equation}\label{main property large time}
			\lim_{t\rightarrow\infty}|\tilde{u}(t,x)|=0,\quad\forall\ x\in\mathbb{S},\ \tilde{\mathbb{P}}-{\rm a.s.}
		\end{equation}
	\end{enumerate}
\end{theo}

\begin{rema}
	In contrast to the $H^1$ weak solutions to deterministic CH equation constructed in \cite{Xin2000}, the $H^1$ martingale solutions exhibit enhanced decay properties, which are attributable to the linear multiplicative noise. This result is consistent with the dissipative effect of linear noise on periodic peakons as shown in \cite{Tang2018}.
\end{rema}

\subsection{Organization of paper}

\quad Now we outline the structure of our paper and the corresponding proof ideas. Similar to the deterministic case, we make use of the vanishing viscosity method and weak convergence techniques; cf. \cite{Holden2024JDE,Xin2000}.

First, in Section \ref{3}, by means of the Galerkin approximation method, we construct a global regular pathwise solution $u_\varepsilon$ to
\begin{equation}\label{CH viscous u}
	\left\{
	\begin{array}{l}
		{\rm d}u_\varepsilon-\varepsilon\partial_x^2u_\varepsilon\,{\rm d}t+(u_\varepsilon\partial_xu_\varepsilon+\partial_xP[u_\varepsilon])\,{\rm d}t=u_\varepsilon\,{\rm d}W,\\
		u_\varepsilon(0,x)=u_{0}(x),\\
		P[u_\varepsilon]=(1-\partial_x^2)^{-1}\left(u_\varepsilon^2+\frac 1 2(\partial_x u_\varepsilon)^2\right).
	\end{array}
	\right.
\end{equation}
Moreover, starting from smooth finite-energy initial data, we obtain some $\varepsilon$-uniform statistical estimates, including
$$
\mathbb{E}\Vert u_\varepsilon\Vert_{C^\theta([0,T];L^2(\mathbb{S}))}^p\lesssim 1\quad\text{and}\quad\mathbb{E}\Vert \partial_xu_\varepsilon\Vert_{L^r([0,T]\times\mathbb{S})}^r\lesssim 1
$$
with $p>2$, $\theta>0$ small enough and $r\in[2,3)$, 
which play a crucial role in the subsequent proof of the existence of global weak solutions to \eqref{CH u} and in deriving their relevant properties.

After having obtained sufficient a priori estimates, we expect that the limiting random variables, obtained as the limits of the viscous solutions, satisfies equation \eqref{CH u}. Therefore, in section \ref{4}, we first aim to establish the limits of viscous solutions the tightness of the corresponding probability laws of the desired random variables in appropriate quasi-Polish spaces respectively. This allows us to apply the Skorokhod-Jakubowski theorem to obtain a Skorokhod-Jakubowski representations on a new probability space $(\tilde{{\Omega}},\tilde{\mathcal{F}},\tilde{\mathbb{P}})$, along with the limiting random variables converging in the corresponding quasi-Polish space. We refer to Brze\'{z}niak and Ondrej\'{a}t \cite{Brzazniak2013,Ondrejat2010,Tang2018} for applications of the Skorokhod-Jakubowski theorem to stochastic partial differential equations. In particular, in order to obtain the SPDE of the limit $\tilde{u}$, the weak convergence of products like $S'(\tilde{q}_n)\,\tilde{P}_n$ for a suitable class of linearly growing nonlinearities $S(\cdot)$ is vital important. Moreover, the treatment of the $\tilde{q}_n^2$ term is the most difficult part; therefore, we construct a family of functions $v\mapsto S_l(v)$ defined as \ref{S_l} to approximate it. Hence, we need to consider the Skorokhod-Jakubowski representations of $S_l(q_\varepsilon)$. According to the equality of the laws and the weak convergences of Skorokhod-Jakubowski representations, we can obtain the SPDEs of $\tilde{u}_n$ and $\tilde{u}$ with assumption $\overline{q^2}=\tilde{q}^2$ a.e. in $\tilde{{\Omega}}\times[0,T]\times\mathbb{S}$.

The final Section \ref{5} is devoted to verifying the assumption $\overline{q^2}=\tilde{q}^2$ a.e. in $\tilde{{\Omega}}\times[0,T]\times\mathbb{S}$. The proof amounts to upgrading the weak convergence to strong convergence via a study of the defect measure:
\begin{equation}\label{defect measure}
	\mathbb{D}=\mathbb{D}(\tilde{\omega},t,x)=\frac 1 2\left(\overline{q^2}-\tilde{q}^2\right)\ge 0.
\end{equation}
Unfortunately, the most difficult challenge is the lack of the integrability of $\tilde{q}^2$ and $\overline{q^2}$. The way to overcome this difficulty is to work with renormalized formulations of the SPDEs for $\tilde{q}_n$, $\tilde{q}$ based on linearly growing approximations $S_l(v)$ of $v^2$ and eventually let $l$ tend to $\infty$. More precisely, we split $v$ into its positive $v_+$ and negative parts $v_-$ and then consider the SPDEs satisfied by the nonlinear compositions $S_l\big((\tilde{q}_n)_\pm\big)$, $S_l(\tilde{q}_\pm)$.

Next, we derive a transport-type SPDE, up to an inequality, for the evolution of $\mathbb{D}$ from the SPDEs of the nonlinear compositions, then by the time-continuity of $\mathbb{D}$ and the initial data $\mathbb{D}(0)=0$, we can deduce that $\mathbb{D}(t)=0$ for all $t>0$. Thereby, the existence of global weak martingale solution to \eqref{CH u} is proved.

\section{Notations and Preliminaries}\label{2}
\begin{defi}[Weak martingale solution]\label{CH H^1}
	Let $\Lambda$ be a probability measure on $H^1 $, for some $p_0>4$,
	$$\int_{H^1 }\Vert v\Vert_{H^1 }^{p_0}\,\Lambda({\rm d}v )<\infty.$$
	The triple $(\mathcal{S},u,W)$ is a weak martingale solution to the stochastic CH equation \eqref{CH u} with initial distribution $\Lambda$ if:
	\begin{enumerate}
		\item $\mathcal{S}=(\Omega,\mathcal{F},\{\mathcal{F}_t\}_{t\geq 0},\mathbb{P})$ is a stochastic basis;
		\item $W$ is a standard Wiener process on $\mathcal{S}$;
		\item $u:\Omega\times [0,T]\rightarrow L^2 $ is a progressively measurable stochastic process with paths $u(\omega)\in C([0,T];L^2 )\cap C([0,T];H^1-w)$, for $\mathbb{P}$-{\rm a.e.} $\omega\in\Omega$. Moreover, $u\in L^2(\Omega;L^\infty([0,T];H^1 ))$;
		\item initial data $\mathbb{P}\{u(0)\in Y\}=\Lambda(Y)$,\ $\forall\ Y\in\mathcal{B}(H^1 )$;
		\item the following equation holds in the sense of It\^{o}, for all $t\in[0,T]$ and for all $\varphi\in C^2(\mathbb{S}) $,
		\begin{equation}\label{CH integral u}
			{\rm d}\int_{\mathbb{S}}u\varphi\,{\rm d}x=\int_{\mathbb{S}}\left(\frac 1 2u^2+P[u]\right)\partial_x\varphi\, {\rm d}x\,{\rm d}t+\int_{\mathbb{S}}u\varphi\, {\rm d}x\,{\rm d}W,\ \mathbb{P}{\rm-a.s. };
		\end{equation}
		\item temporal right-continuity in $H^1 $: for a.e. $(\omega,t_0)\in\Omega\times[0,T]$,
		$$\lim_{t\downarrow t_0}\Vert u(t)-u(t_0)\Vert_{H^1}=0.$$
	\end{enumerate}
\end{defi}

In \cite{Holden2024JDE}, Galimberti, Holden, Karlsen, and Pang established a tightness criterion for the probability laws of a sequence of random variables taking values in a quasi-Polish space equipped with the weak spatial topology, which allows passing to the limit in products such as $q_\varepsilon P_\varepsilon$ towards $qP$.

\begin{lemm}[Compactness criterion]\label{compactness criterion}\cite{Holden2024JDE}
	Fix some integrability indices $p_1, p_2 \in (1, \infty)$, and consider the space $L^{p_1}(L^{p_2}_w) = L^{p_1}([0, T]; L^{p_2}(\mathbb{S}^1)-w)$, defined by weak topology \eqref{weak topology}. Let $\mathcal{K}$ be a subset of $L^{p_1}(L^{p_2}_w)$ for which the following conditions hold uniformly in $Q \in K$:
	\begin{enumerate}
		\item $\|Q\|_{L^{p_1}([0,T]; L^{p_2}(\mathbb{S}))} \lesssim 1$,
		\item $\|Q\|_{L^{\bar{p}_1}([0,T]; L^1(\mathbb{S}))} \lesssim 1$, for some $\bar{p}_1>p_1$,
		\item $\int_0^{T-\tau} \left| \int_{\mathbb{S}} \varphi(x) \bigl( Q(t+\tau, x) - Q(t, x) \bigr)\,{\rm d}x \right| \,{\rm d}t \rightarrow 0$ as $\tau\downarrow0$, $\forall\ \varphi \in C^\infty(\mathbb{S})$.
	\end{enumerate}
	Then $\mathcal{K}$ is relatively sequentially compact in $L^{p_1}(L^{p_2}_w)$.
\end{lemm}

\begin{lemm}[Tightness criterion]\label{tightness criterion}\cite{Holden2024JDE}
	Fix $p_1, p_2 \in (1, \infty)$, and consider the quasi-Polish space $L^{p_1}(L^{p_2}_w) = L^{p_1}([0, T]; L^{p_2}(\mathbb{S}^1)-w)$ defined by weak topology \eqref{weak topology}. Let $\{Q_n\}_{n\in\mathbb{N}}$ be a sequence of random variables defined on a standard probability space $(\Omega,\mathcal{F},\mathbb{P})$, that take values in $L^{p_1}(L^{p_2}_w)$. Supposing the following conditions holds uniformly in $n\in\mathbb{N}$:
	\begin{enumerate}
		\item $\mathbb{E}\|Q_n\|_{L^{p_1}([0,T]; L^{p_2}(\mathbb{S}))} \lesssim 1$,
		\item $\mathbb{E}\|Q_n\|_{L^{\bar{p}_1}([0,T]; L^1(\mathbb{S}))} \lesssim 1$, for some $\bar{p}_1>p_1$,
		\item $\forall \ \varphi \in C^\infty(\mathbb{S})$ and $\vartheta\in(0,T\wedge 1)$, $$\mathbb{E}\sup_{\tau\in(0,\vartheta)}\int_0^{T-\tau} \left| \int_{\mathbb{S}} \varphi(x) \bigl( Q(t+\tau, x) - Q(t, x) \bigr)\,{\rm d}x \right| \,{\rm d}t \le C_\varphi\vartheta^\alpha,$$
		for some $\alpha\in(0,1)$ and $C_\varphi$ independent of $n$.
	\end{enumerate}
	Then the sequence $\{\mathcal{L}(Q_n)\}_{n\in\mathbb{N}}$ of probability laws is tight on $L^{p_1}(L^{p_2}_w)$.
\end{lemm}

\section{Viscous Approximate Solutions}\label{3}

\quad In this section, we construct the global well-posedness of the viscous problem \eqref{CH viscous u}. At first, we present the definitions of the $H^m$ martingale and pathwise solution for the viscous SPDEs, and then we state the result concerning its global well-posedness.

\begin{defi}[$H^m$ martingale solution of viscous SPDE]
	For fixed $m\geq 1$ and some $p_0>4$, let $\Lambda$ be a probability measure on $H^m(\mathbb{S}) $, satisfying 	$$\int_{H^m(\mathbb{S}) }\Vert v\Vert_{H^m(\mathbb{S}) }^{p_0}\,\Lambda({\rm d}v )<\infty.$$
	The triple $(\mathcal{S},u,W)$ is a $H^m$ martingale solution to the stochastic viscous CH equation \eqref{CH viscous u} with initial distribution $\Lambda$ if:
	\begin{enumerate}
		\item $\mathcal{S}=(\Omega,\mathcal{F},\{\mathcal{F}_t\}_{t\geq 0},\mathbb{P})$ is a stochastic basis;
		\item $W$ is a standard Wiener process on $\mathcal{S}$;
		\item $u:\Omega\times [0,T]\rightarrow H^1(\mathbb{S}) $ is adapted, with $u\in L^{p_0}(\Omega;C([0,T];H^1(\mathbb{S}) ))$.  Moreover, $$u\in L^\infty([0,T];H^m(\mathbb{S}) )\cap L^2([0,T];H^{m+1}(\mathbb{S}) ) \quad\text{ and }\quad u\in L^2(\Omega;L^2([0,T];H^2 (\mathbb{S})))\  {\rm a.s.};$$
		\item initial data $\mathbb{P}\{u(0)\in Y\}=\Lambda(Y)$,\ $\forall\ Y\in\mathcal{B}(H^m(\mathbb{S}) )$;
		\item the following equation holds in the sense of It\^{o}, for all $t\in[0,T]$ and for all $\varphi\in C^1 $,
		\begin{equation}\label{CH integral viscous u}
			\begin{aligned}
				&\int_\mathbb{S}u(t)\varphi\,{\rm d}x-\int_{\mathbb{S}}u_0\varphi\,{\rm d}x\\
				=&\int_0^t\int_{\mathbb{S}}-u\partial_x u\varphi+\left(P[u]-\varepsilon\partial_xu\right)\partial_x\varphi\,{\rm d}x\,{\rm d}s+\int_0^t\int_{\mathbb{S}}u\varphi\,{\rm d}x\,{\rm d}W(s),\ \mathbb{P}{\rm-a.s. }.
			\end{aligned}
		\end{equation}
	\end{enumerate}
\end{defi}

\begin{defi}[$H^m$ pathwise solution of viscous SPDE]
	Let $u_0\in L^{p_0}(\Omega;H^m )$ for some $p_0>4$. Consider a fixed stochastic basis $\mathcal{S}=(\Omega,\mathcal{F},\{\mathcal{F}_t\}_{t\in[0,T]},\mathbb{P})$, we say $u$ defined on $\mathcal{S}$ is a $H^m$ pathwise solution to \eqref{CH viscous u} with initial data $u(0)=u_0$ if for a given Wiener process $W$ defined on $\mathcal{S}$, the triple $(\mathcal{S},u,W)$ constitutes a $H^m$ martingale solution to \eqref{CH viscous u} with initial distribution $\Lambda(Y)=\mathbb{P}\{u_0\in Y\}$,\ $\forall\ Y\in\mathcal{B}(H^m (\mathbb{S}))$.
\end{defi}

\begin{theo}\label{viscous Hm global}
	Fix $m\ge 1$ and $p_0>4$. Suppose $u_0\in L^{p_0}(\Omega,H^m(\mathbb{S}))$, then there exists a unique $H^m$ pathwise solution to \eqref{CH viscous u} with initial data condition $u|_{t=0}=u_0$.
\end{theo}

\subsection{A priori estimate of Galerkin approximate solutions}
To establish the global existence of $H^m$ pathwise solutions to the viscous equation \eqref{CH viscous u}, we employ the Galerkin method to construct a sequence of approximate solutions. Let $\{e_1,e_2,\dots\}\subseteq H^1 $ be an orthonormal basis of $L^2 $ that is dense in $H^1 $ and set $H_n={\rm span}\{e_1,e_2,\dots,e_n\}$, in particular, we take $e_{2j}=\cos(2\pi jx)$ and $e_{2j+1}=\sin(2\pi j x)$ for $x\in [0,1]$. Let $\Pi_n:\left(H^1 \right)^\ast\rightarrow H_n$ be defined by
$$\Pi_n u:=\sum_{i=1}^n\langle u,e_i\rangle_{L^2 }e_i,$$
restricted to $L^2 $, $\Pi_n$ is the orthogonal projection onto $H_n$. It can be easily deduced that for any $f\in H^1 $ and $n\in\mathbb{N}$, the spatial derivative of $\Pi_n f$ satisfies 
$$\partial_x\Pi_n f=\Pi_n\partial_x f,$$
and since $\Pi_n$ is self-adjoint and idempotent on $L^2(\mathbb{S})$, then for any $f\in L^2(\mathbb{S})$,
$$\int_\mathbb{S}u_n\Pi_n f\,{\rm d}x=\int_\mathbb{S}u_n f\,{\rm d}x.$$

For each $n\in\mathbb{N}$, we consider the Galerkin approximation of \eqref{CH integral viscous u} on $H_n$, i.e. for $u_n(\omega,t,x)=\sum_{i=1}^nw_i(\omega,t)e_i(x)$ satisfies the following equation:
\begin{equation}\label{Galerkin}
	\left\{
	\begin{array}{l}
		0={\rm d}u_n-\varepsilon\partial_x^2u_n\,{\rm d}t+\Pi_n(u_n\partial_x u_n+\partial_x P[u_n])\,{\rm d}t-u_n\,{\rm d}W,\\
		u_n(0)=\Pi_n u_0.
	\end{array}
	\right.
\end{equation}
Here $\mathbb{E}\Vert u_0\Vert_{H^1(\mathbb{S}) }^2<\infty$.

\begin{theo}\cite{Holden2023}
	For any fixed $n$, there exists a unique $C([0,T];H_n)$-valued adapted process $u_n$ that is a strong solution of \eqref{Galerkin}.
\end{theo}

\begin{prop}[$H^1$ estimates]\label{Galerkin estimate}
	For each $n\in\mathbb{N}$, let $u_n$ be a solution to \eqref{Galerkin} with $\mathbb{E}\Vert u_0\Vert_{H^1 }^2<\infty$, there exists a constant $C=C(T,\mathbb{E}\Vert u_0\Vert_{H^1 (\mathbb{S})}^p)$ independent of $n$ and $\varepsilon$, such that
	\begin{equation}\label{Galerkin estimate 2}
		\mathbb{E}\Vert u_n\Vert_{L^\infty([0,T];H^1 (\mathbb{S}))}^2+\varepsilon\mathbb{E}\int_0^T\Vert\partial_xu_n(t)\Vert_{H^1(\mathbb{S}) }^2\,{\rm d}t\leq C.
	\end{equation}
	Moreover, if for fixed $p\in [4,\infty)$, $\mathbb{E}\Vert u_0\Vert_{H^1(\mathbb{S}) }^p<\infty$, there exists a constant $C=C(p,T,\mathbb{E}\Vert u_0\Vert_{H^1(\mathbb{S}) }^2)$, such that
	\begin{equation}\label{Galerkin estimate p}
		\mathbb{E}\Vert u_n\Vert_{L^\infty([0,T];H^1(\mathbb{S}) )}^p+\varepsilon^{p/2}\left( \mathbb{E}\int_0^T\Vert\partial_xu_n(t)\Vert_{H^1(\mathbb{S}) }^2 \,{\rm d}t \right)^{p/2}\leq C.
	\end{equation}
\end{prop}
\begin{proof}

	By \eqref{Galerkin} and It\^{o}'s formula, since $\partial_x\Pi_n f=\Pi_n\partial_x f$, we obtain that
	$$
	\begin{aligned}
		&\frac 1 2{\rm d}\int_\mathbb{S}|u_n|^2\,{\rm d}x+\varepsilon\int_{\mathbb{S}}|\partial_x u_n|^2\,{\rm d}x\,{\rm d}t\\
		=&-\int_{\mathbb{S}}u_n^2\partial_x u_n+u_n\partial_x K\ast\left(u_n^2+\frac 1 2(\partial_x u_n)^2\right)\,{\rm d}x\,{\rm d}t
		+\frac 1 2\int_\mathbb{S}u_n^2\,{\rm d}x\,{\rm d}t
		+\int_{\mathbb{S}}u_n^2\,{\rm d}x\,{\rm d}W,\\
		&\frac 1 2{\rm d}\int_\mathbb{S}|\partial_xu_n|^2\,{\rm d}x+\varepsilon\int_{\mathbb{S}}|\partial_x^2 u_n|^2\,{\rm d}x\,{\rm d}t\\
		=&\int_{\mathbb{S}}u_n\partial_x u_n\partial_x^2 u_n+\partial_x^2u_n\partial_x K\ast\left(u_n^2+\frac 1 2(\partial_x u_n)^2\right)\,{\rm d}x\,{\rm d}t
		+\frac 1 2\int_{\mathbb{S}}|\partial_xu_n|^2\,{\rm d}x\,{\rm d}t
		+\int_{\mathbb{S}}|\partial_xu_n|^2\,{\rm d}x\,{\rm d}W,
	\end{aligned}
	$$
	adding the previous two equations, we arrive at
	$$\frac 1 2{\rm d}\Vert u_n\Vert_{H^1(\mathbb{S})}^2+\varepsilon\int_{\mathbb{S}}\Vert\partial_x u_n\Vert_{H^1(\mathbb{S})}^2\,{\rm d}x\,{\rm d}t=I_1^n\,{\rm d}t+I_2^n\,{\rm d}t+I_3^n\,{\rm d}W$$
	where
	$$
	\begin{aligned}
		&I_1^n=-\int_{\mathbb{S}}u_n^2\partial_x u_n+u_n\partial_x K\ast\left(u_n^2+\frac 1 2(\partial_x u_n)^2\right)\,{\rm d}x\\
		&\qquad\ +\int_{\mathbb{S}}u_n\partial_x u_n\partial_x^2 u_n+\partial_x^2u_n\partial_x K\ast\left(u_n^2+\frac 1 2(\partial_x u_n)^2\right)\,{\rm d}x,\\
		&I_3^n=2I_2^n=\int_{\mathbb{S}}\Vert u_n\Vert_{H^1(\mathbb{S})}^2\,{\rm d}x.
	\end{aligned}
	$$
	According to the kernel property of $K$ that $K-\partial_x^2 K=\delta$, the Dirac mass, we can easily deduce that $I_1^n=0$, thus for any $T>0$, integrating the above equality with respect to time from $0$ to $T$ yields:
	\begin{equation}\label{p=2 estimate Galerkin}
		\Vert u_n(t)\Vert_{H^1(\mathbb{S}) }^2-\Vert u_{n0}\Vert_{H^1 (\mathbb{S})}^2+2\varepsilon\int_0^T\Vert\partial_x u_n\Vert_{H^1(\mathbb{S}) }^2\,{\rm d}t\leq2\left|\int_0^t\Vert u_n\Vert_{H^1(\mathbb{S}) }^2\,{\rm d}W\right|+\int_0^T\Vert u_n\Vert_{H^1(\mathbb{S}) }^2\,{\rm d}t,
	\end{equation}
	taking a supremum over $t\in[0,T]$ and employing the BDG inequality and the Cauchy-Schawarz inequality yields that for some constant $C>0$
	$$
	\begin{aligned}
		\mathbb{E}\Vert u_n\Vert_{L^\infty([0,T];H^1(\mathbb{S}) )}^2+2\varepsilon\mathbb{E}\int_0^T&\Vert\partial_x u_n\Vert_{H^1(\mathbb{S}) }^2\,{\rm d}t\\
		\leq&\mathbb{E}\Vert u_{n0}\Vert_{H^1(\mathbb{S}) }^2+
		\frac 1 2\mathbb{E}\Vert u_n\Vert_{L^\infty([0,T];H^1(\mathbb{S}) )}^2+C\int_0^T\mathbb{E}\Vert u_n\Vert_{L^\infty([0,T];H^1(\mathbb{S}) )}^2\,{\rm d}t.
	\end{aligned}$$
	Therefore we can conclude \eqref{Galerkin estimate 2} by Gronwall's inequality.
	
	For $p>4$, by \eqref{p=2 estimate Galerkin} we can deduce that
	$$
	\begin{aligned}
		&\Vert u_n(t)\Vert_{H^1 (\mathbb{S})}^p+\left(2\varepsilon\int_0^T\Vert\partial_x u_n\Vert_{H^1(\mathbb{S}) }^2\,{\rm d}t\right)^{p/2}\\
		\leq& C_p\left(\Vert u_{n0}\Vert_{H^1(\mathbb{S}) }^p+\left|\int_0^t\Vert u_n\Vert_{H^1(\mathbb{S}) }^2\,{\rm d}W\right|^{p/2}+\left(\int_0^T\Vert u_n\Vert_{H^1(\mathbb{S}) }^2\,{\rm d}t\right)^{p/2}\right).
	\end{aligned}
	$$
	Since the convexity of $x\mapsto x^{p/4}$, also taking a supremum over $t\in[0,T]$ and using the BDG inequality, the Cauchy-Schwarz inequality and the Gronwall inequality, we can finally obtain \eqref{Galerkin estimate p}.
\end{proof}

In order to construct the tightness of laws $\{u_n\}$, we need the temporal continuity of $u_n$.
\begin{lemm}\label{Galerkin temporal continuity}
	For each $n\in\mathbb{N}$, let $\tilde{u}_n$ be a solution to \eqref{Galerkin} with $\mathbb{E}\Vert u_0\Vert_{H^1(\mathbb{S})}^p<\infty$ for $p>2$. For any $\theta\in [0,\frac{p-2}{4p})$, there exists a constant 
	$$C=C\left(T,p,\theta,\mathbb{E}\Vert u_0\Vert_{H^1(\mathbb{E})}^{4p}\right),$$
	independent of $n$ and $\varepsilon$, such that
	$$\mathbb{E}\Vert u_n\Vert_{C^\theta([0,T];L^2(\mathbb{S}))}^{2p}\le C.$$
\end{lemm}
\begin{proof}
	First, we separate the spatial integral as
	$$\int_\mathbb{S}(u_n(t)-u_n(s))^2\,{\rm d}x=\int_\mathbb{S} (u_n(t)-u_n(s))\int_s^t{\rm d}u_n(r)\,{\rm d}x=\sum_{i=1}^4 I_i,$$
	where
	$$
	\begin{aligned}
		&I_1=-\int_\mathbb{S} (u_n(t)-u_n(s))\int_s^t u_n(r)\partial_xu_n(r)\,{\rm d}x\,{\rm d}x,\\
		&I_2=-\int_\mathbb{S} (u_n(t)-u_n(s))\int_s^t \partial_x P[u_n](r)\,{\rm d}x\,{\rm d}x,\\
		&I_3=\varepsilon \int_\mathbb{S} (u_n(t)-u_n(s))\int_s^t \partial_x^2u_n\,{\rm d}x\,{\rm d}x,\\
		&I_2=-\int_\mathbb{S} (u_n(t)-u_n(s))\int_s^t u_n\,{\rm d}W(r)\,{\rm d}x.
	\end{aligned}
	$$
	
	By the Minkowski inequality and $H^1\hookrightarrow L^\infty$, we have 
	$$
	\begin{aligned}
		\mathbb{E}|I_1|^p
		&\le C\mathbb{E}\left(|t-s|\int_\mathbb{S} \|u_n\|_{L^\infty_T}^2\Vert\partial_x u_n\|_{L^\infty_T}\,{\rm d}x\right)^p
		\le C|t-s|^p\mathbb{E}\left(\|u_n\|_{L^2_xL^\infty_T}\| u_n\|_{L^\infty_xL^\infty_T}\|\partial_xu_n\|_{L^2_xL^\infty_T}\right)^p\\
		&\le C\left(\mathbb{E}\|u_n\|_{L^\infty([0,T];H^1(\mathbb{S}))}^{4p}\right)^{1/2}\left(\mathbb{E}\|u_n\|_{L^\infty([0,T];H^1(\mathbb{S}))}^{2p}\right)^{1/2}|t-s|^p.
	\end{aligned}$$
	
	Using Young's convolution inequality, we obtain
	$$\|\partial_x P[u_n]\|_{L^2(\mathbb{S})}\le C\|u_n+\frac 1 2(\partial_x u_n)^2\|_{L^1(\mathbb{S})}\le C\|u_n\|_{H^1(\mathbb{S})},$$
	and therefore
	$$
	\begin{aligned}
		\mathbb{E}|I_2|^p
		&\le C\mathbb{E}\left(\|u_n\|_{L^\infty([0,T];L^2(\mathbb{S}))}^p\|u_n\|_{L^\infty([0,T];H^1(\mathbb{S}))}\right)|t-s|^p\\
		&\le C\left(\mathbb{E}\|u_n\|_{L^\infty([0,T];H^1(\mathbb{S}))}^{4p}\right)^{1/2}\left(\mathbb{E}\|u_n\|_{L^\infty([0,T];H^1(\mathbb{S}))}^{2p}\right)^{1/2}|t-s|^p.
	\end{aligned}
	$$
	
	Integrating by parts, we can easily deduce that
	$$\mathbb{E}|I_3|^p\le C\varepsilon^p\mathbb{E}\|u_n\|_{L^\infty([0,T];H^1(\mathbb{S}))}^{2p}|t-s|^p.$$
	
	Finally, by the BDG inequality, Jensen's inequality and \eqref{viscous estimate p}, pointwise in $x$
	$$
	\begin{aligned}
		\mathbb{E}|I_4|^p
		&\le\left(\mathbb{E}\|u_n\|_{L^\infty([0,T];H^1(\mathbb{S}))}^{2p}\right)^{1/2}\left(\int_\mathbb{S}\mathbb{E}\left|\int_s^tu_n\,{\rm d}W(r)\right|^{2p}{\rm d}x\right)^{1/2}\\
		&\le C\left(\mathbb{E}\|u_n\|_{L^\infty([0,T];H^1(\mathbb{S}))}^{2p}\right)^{1/2}\left(\int_\mathbb{S}\mathbb{E}\left|\int_s^tu_n^2\,{\rm d}r\right|^{p}{\rm d}x\right)^{1/2}
		\le C\mathbb{E}\|u_n\|_{L^\infty([0,T];H^1(\mathbb{S}))}^{2p} |t-s|^{p/2}.
	\end{aligned}
	$$
	
	Combining above inequalities, we conclude that
	$$\mathbb{E}\|u_n(t)-u_n(s)\|_{L^2(\mathbb{S})}^{2p}\le C|t-s|^{p/2}=C |t-s|^{1+\frac {p-2} 2},$$
	where the constant $C$ is independent of $\varepsilon$. By Kolmogorov's continuity criterion, there is a version of $u_n$ in $C^\theta([0,T];L^2)$ for any $\theta\in[0,\frac {p-2} {4p})$.
\end{proof}

For $m>1$, for $f\in H^m(\mathbb{S})$, the convergence $\Vert \Pi_n f-f\Vert_{H^m(\mathbb{S})}\rightarrow0$ also holds. Thus, we consider the higher regularity of the Galerkin approximate solutions in order to build an $H^m$ pathwise solution to \eqref{CH viscous u}.

\begin{prop}[$H^m$ estimates up to a stopping time]\label{Galerkin estimate Hm}
	Let $u_n$ be a solution to \eqref{Galerkin} and initial data $u_0\in L^{2p}(\Omega;H^m(\mathbb{S}))$ for some $p\in[2,\infty)$. For $R>1$, let $\eta_R^n$ be the stopping time
	\begin{equation}\label{Gelerkin Hm stopping time}
		\eta_R^n:=\inf\left\{t\in[0,T]:\int_0^t\Vert u_n\Vert_{W^{1,\infty}(\mathbb{S})}^2\,{\rm d}s>R\right\},
	\end{equation}
	then $\eta_R^n\rightarrow T$, $\mathbb{P}$-a.s. uniformly in $n$ as $R\rightarrow\infty$. 
	
	Therefore, there exists a constant $C=C(p,T,R,\varepsilon,E\Vert u_0\Vert_{H^m(\mathbb{S})}^{2p})$ independent of $n$, such that 
	$$\mathbb{E}\Vert u_n\Vert_{L^\infty([0,\eta_R^n];H^m (\mathbb{S}))}^p\le C.$$
	In particular, for $p=2$, we have
	$$\mathbb{E}\Vert u_n\Vert_{L^2([0,\eta_R^n];H^{m+1} (\mathbb{S}))}^2\le C.$$
\end{prop}
\begin{proof}
	According to the embedding $H^2(\mathbb{S})\hookrightarrow W^{1,\infty}(\mathbb{S})$ and Chebyshev's inequality,
	$$\begin{aligned}
		\mathbb{P}(\{\eta_R^n<T\})
		&\le\mathbb{P}\left(\left\{
		\int_0^T\Vert u_n(t)\Vert_{W^{1,\infty}(\mathbb{S})}^2\,{\rm d}t>R\right\}\right)\\
		&\le\frac C R\,\mathbb{E}\int_0^T\Vert u_n(t)\Vert_{H^2 (\mathbb{S})}^2\,{\rm d}t\overset{\eqref{Galerkin estimate 2}}{\le}\frac{C_\varepsilon}{R}\longrightarrow0
	\end{aligned}
	$$
	as $R$ tends to $\infty$, uniformly in $n$.
	
	Taking $l$-th derivative of \eqref{Galerkin}, 
	$${\rm d}\partial_x^l u_n-\varepsilon\,\partial_x^{l+2}u_n\,{\rm d}t+\left(\partial_x^l\Pi_n\,(u_n\partial_x u_n)+\partial_x^{l+1}P[u_n]\right)\,{\rm d}t=\partial_x^l u_n\,{\rm d}W.$$
	Apply It\'{o}'s formula, we can obtain that
	$$
	\begin{aligned}
		&\frac 1{2p}\,\Vert\partial_x^l u_n\Vert_{L^2(\mathbb{S})}^{2p}\bigg|_0^t+\varepsilon\,\int_0^t\Vert\partial_x^l u_n(s)\Vert_{L^2(\mathbb{S})}^{2p-2}\,\Vert\partial_x^{l+1}u_n(s)\Vert_{L^2(\mathbb{S})}^2\,{\rm d}s\\
		=&\frac 1 2\int_0^t\Vert\partial_x^l u_n(s)\Vert_{L^2(\mathbb{S})}^{2p-2}\int_{\mathbb{S}}\partial_xu_n\,(\partial_x^l u_n)^2\,{\rm d}x\,{\rm d}s	-\int_0^t\Vert\partial_x^l u_n(s)\Vert_{L^2(\mathbb{S})}^{2p-2}\int_{\mathbb{S}}\partial_x^{l+1}P[u_n]\partial_x^l u_n\,{\rm d}x\,{\rm d}s\\
		&+\int_0^t\Vert\partial_x^l u_n(s)\Vert_{L^2(\mathbb{S})}^{2p-2}\int_{\mathbb{S}}\left(u_n\,\partial_x^{l+1}u_n-\partial_x^l(u_n\,\partial_xu_n)\right)\partial_x^lu_n\,{\rm d}x\,{\rm d}s\\
		&+\int_0^t\Vert\partial_x^l u_n(s)\Vert_{L^2(\mathbb{S})}^{2p-2}\int_{\mathbb{S}}(\partial_x^l u_n)^2\,{\rm d}x\,{\rm d}W\\
		=&:\sum_{i=1}^3\int_0^t I_i\,{\rm d}s+\int_0^t \Vert \partial_x^l u_n(s)\Vert_{L^2(\mathbb{S})}^{2p}\,{\rm d}W.
	\end{aligned}
	$$
	Obviously, the stochastic integral $\int_0^{t\wedge\eta_R^n}\Vert \partial_x^l u_n(s)\Vert_{L^2(\mathbb{S})}^{2p}\,{\rm d}W$ is a square-integrable martingale.
	
	Applying the Leibniz rule and Gagliardo-Nirenberg inequality to $I_1$ and $I_2$, as in \cite{Xin2000}, we obtain
	$$
	\begin{aligned}
		|I_1|+|I_2|\le &C_\varepsilon\,\Vert\partial_x u_n\Vert_{L^\infty(\mathbb{S})}\,\Vert\partial_x^l u_n\Vert_{L^2(\mathbb{S})}^{2p}\\
		&+ C_\varepsilon\,\Vert u_n\Vert_{W^{1,\infty}(\mathbb{S})}\,\Vert\partial_x^l u_n\Vert_{L^2(\mathbb{S})}^{2p}
		+\frac \varepsilon 2\,\Vert\partial_x^l u_n\Vert_{L^2(\mathbb{S})}^{2p-2}\,\Vert\partial_x^{l+1}u_n\Vert_{L^2(\mathbb{S})}^2.
	\end{aligned}
	$$
	For $I_3$, we employ the commutator estimate as Lemma 2.97 in \cite{BCD2011} and Proposition 4.2 in \cite{Taylor2003},
	$$|I_3|\le C\,\Vert\partial_x u_n\Vert_{L^\infty(\mathbb{S})}\,\Vert\partial_x^l u_n\Vert_{L^2(\mathbb{S})}^{2p}.$$
	
	Therefore, an application of the stochastic Gronwall inequality gives
	$$
	\begin{aligned}
		&\left(\mathbb{E}\sup_{t\in[0,\eta_R
			^n]}
		\left(\Vert\partial_x^lu_n(s)\Vert_{L^2(\mathbb{S})}^{2p}+\frac \varepsilon 2\,\int_0^s\Vert\partial_x^lu_n(s')\Vert_{L^2(\mathbb{S})}^{2p-2}\,\Vert\partial_x^{l+1}u_n(s')\Vert_{L^2(\mathbb{S})}^{2}\,{\rm d}s'\right)^{1/2}
		\right)^2\\[2ex]
		&\le C\,\mathbb{E}\Vert\partial_x^lu_n(0)\Vert_{L^2(\mathbb{S})}^{2p},
	\end{aligned}
	\quad l=0,\dots,m.
	$$
	In particular, take $p=2$ we therefore obtain the second inequality in proposition.
	
\end{proof}

\subsection[Existence of Global Hm Solution]{Existence of Global $H^m$ Solution}
\quad We first establish the existence and uniqueness of a global $H^1$ pathwise solution. In order to upgrade the martingale solution to pathwise solution, below pathwise uniqueness is essential.

\begin{prop}[Pathwise uniqueness in $H^1$]\label{pathwise uniqueness}
	Let $u,v$ be strong $H^1$ solutions to the viscous stochastic CH equation \eqref{CH viscous u} with initial data $u_0\in L^{p_0}(\Omega;H^1(\mathbb{S}))$ for some $p_0>4$. Then
	$$\mathbb{E}\Vert u-v\Vert_{L^\infty([0,T];H^1(\mathbb{S}))}=0.$$
\end{prop}
\begin{proof}
	Suppose $u$ and $v$ are pathwise solutions defined relative to the same stochastic basis $(\Omega,\mathcal{F},$ $\{\mathcal{F}_t\}_{t\ge 0},\mathbb{P})$ and Brownian motion $W$. The difference $w=u-v$ obeys
	$$
	\begin{aligned}
		&{\rm d}w-\varepsilon\partial_x^2w\,{\rm d}t+(u\,\partial_x u-v\,\partial_xv)\,{\rm d}t+\partial_x K\ast\left(w\,(u+v)+\frac 1 2w\,\partial_x(u+v)\right)\,{\rm d}t=w\,{\rm d}W,\\
		&{\rm d}w_x-\varepsilon\partial_x^2w_x\,{\rm d}t+\partial_x(u\,\partial_xu-v\,\partial_xv)\,{\rm d}t+\partial_x^2K\ast \left(w\,(u+v)+\frac 1 2w\,\partial_x(u+v)\right)\,{\rm d}t=w_x\,{\rm d}W.
	\end{aligned}$$
	Mollify above two equations by $J_\delta$ the standard Friedrichs mollifier with notation $w_\delta=w\ast J_\delta$. The SPDEs can be understood in the pointwise sense:
	\begin{equation}
		\begin{aligned}
			&{\rm d}w_\delta-\varepsilon\partial_x^2w_\delta\,{\rm d}t+(u_\delta\,\partial_x u_\delta-v_\delta\,\partial_xv_\delta)\,{\rm d}t\\
			&\quad+\partial_x K\ast\left(w_\delta\,(u_\delta+v_\delta)+\frac 1 2w\,\partial_x(u_\delta+v_\delta)\right)\,{\rm d}t+E_\delta\,{\rm dt}=w\,{\rm d}W\\
			&{\rm d}\partial_xw_\delta-\varepsilon\partial_x^2\partial_xw_\delta\,{\rm d}t+\partial_x(u_\delta\,\partial_xu_\delta-v_\delta\,\partial_xv_\delta)\,{\rm d}t\\
			&\quad+\partial_x^2K\ast \left(w_\delta\,(u_\delta+v_\delta)+\frac 1 2w_\delta\,\partial_x(u_\delta+v_\delta)\right)\,{\rm d}t+\partial_x E_\delta\,{\rm d}t=\partial_xw_\delta\,{\rm d}W,\\
			&E_\delta=J_\delta\ast(u\,\partial_x u-v\,\partial_xv)-(u_\delta\,\partial_x u_\delta-v_\delta\,\partial_xv_\delta)\\
			&\qquad+J_\delta\ast\partial_x K\ast\left(w\,(u+v)+\frac 1 2w\,\partial_x(u+v)\right)\\
			&\qquad-\partial_x K\ast\left(w_\delta\,(u_\delta+v_\delta)+\frac 1 2w\,\partial_x(u_\delta+v_\delta)\right).
		\end{aligned}
	\end{equation}
	Using standard energy estimation methods and the convergence of error term $E_\delta$(see in Lemma 7.1 in \cite{Holden2023}), we introduce a stopping time $\eta_R=\inf\left\{t\in\mathbb{R}_+:\int_0^{t\wedge T}\Vert u+v\Vert_{W^{1,\infty}(\mathbb{S})}^2(t)\,{\rm d}s>R\right\}$ to ensure that the stochastic integral term is a square-integrable martingale. After obtaining the convergence of $\mathbb{E}\sup_{t\in[0,\eta_R]}\Vert w_\delta(t)\Vert_{H^1(\mathbb{S})}\rightarrow0$ as $\delta\rightarrow0$, we extend the stopping time to $T$ and take the limit with respect to $w_\delta(t)\rightarrow w(t)$ in $H^1(\mathbb{S})$, then we can prove the pathwise uniqueness stated in the proposition.
	
	This method is standard in SPDE uniqueness analysis and is omitted here.
\end{proof}

\begin{theo}\label{Galerkin H^1}
	Suppose $p_0>4$ and $u_0\in L^{p_0}(\Omega;H^1(\mathbb{S}))$. There exists a unique $H^1$ pathwise solution to \eqref{CH viscous u} with initial data $u|_{t=0}=u_0$.
\end{theo}
\begin{proof}
	According to the a priori estimate of the Galerkin approximate solutions $u_n$ in Proposition \ref{Galerkin estimate}, Lemma \ref{Galerkin temporal continuity}, we can establish the tightness of probability laws of $\{u_n\}$ in $L^2([0,T];H^1(\mathbb{S}))$ and $C([0,T];H^1_w(\mathbb{S}))$. Therefore, using Skorokhod-Jakubowski representation theorem Lemma \ref{SJ reprensentation theorem} in quasi-Polish space and the convergence of the projection operator $\Pi_n$, we can obtain the $H^1$ martingale solution to \eqref{CH viscous u}. Next, since we claim the pathwise uniqueness of $H^1$ solution to \eqref{CH viscous u} in Proposition \ref{pathwise uniqueness}, according to Gy\"{o}ngy-Krylov theorem and Lemma \ref{map of quasi Polish}, we can conclude that the $H^1$ martingale solution we built above is actually a unique $H^1$ pathwise solution of \eqref{CH viscous u}. Finally, also by the Friedrichs mollifier and convergence of error terms, the following energy equality holds:
	$$  
	\mathbb{E}\Vert u(r)\Vert_{H^1(\mathbb{S})}^2\big|_s^t+2\varepsilon\,\mathbb{E}\int_s^t\Vert \partial_x u\Vert_{H^1(\mathbb{S})}^2\,{\rm d}r =\mathbb{E}\int_s^t\mathbb{S}\Vert u(r)\Vert_{H^1(\mathbb{S})}^2\,{\rm d}x\,{\rm d}r.
	$$
	thus, we can deduce that $t\mapsto\mathbb{E}\Vert u(t)\Vert_{H^1(\mathbb{S})}^2$ is continuous. Combining with  $u\in C([0,T];H^1_w(\mathbb{S}))$, the temporal continuity is proved. Therefore, we finally obtain the unique $H^1$ pathwise solution to \eqref{CH viscous u}.
	
	The above procedure is the standard Galerkin approximation approach for pathwise solutions. As it is lengthy and not central to this paper, we only sketch the main steps and omit the detailed proof. 
\end{proof}

\begin{theo}
	Suppose $p_0>4$ and $u_0\in L^{p_0}(\Omega;H^m(\mathbb{S}))$ for some $m>1$. There exists a unique $H^1$ pathwise solution to \eqref{CH viscous u} with initial data $u|_{t=0}=u_0$.
\end{theo}
\begin{proof}
	The pathwise uniqueness of $H^m$ solution can be deduce from Lemma \ref{pathwise uniqueness} and uniqueness of weak-strong limits in $L^1(\Omega;L^\infty([0,T;H^m(\mathbb{S})]))$. Therefore, with the same argument employed in Theorem \ref{Galerkin H^1}, we conclude that the global well-posedness of \eqref{CH viscous u} in $H^m(\mathbb{S})$.
\end{proof}

\subsection{Properties of Solution to Viscous SPDE}
\begin{prop}\label{viscous estimate}
	Let $u_\varepsilon$ be a solution to \eqref{CH viscous u} with $\mathbb{E}\Vert u_0\Vert_{H^1}^2<\infty$, for any $T>0$, there exists a constant $C=C(T,\mathbb{E}\Vert u_0\Vert_{H^1}^2)$ independent of $\varepsilon$, such that
	\begin{equation}\label{viscous estimate 2}
		\mathbb{E}\Vert u_\varepsilon\Vert_{L^\infty([0,T];H^1)}^2+\varepsilon\mathbb{E}\int_0^T\Vert\partial_xu_\varepsilon(t)\Vert_{H^1}^2\,{\rm d}t\leq C.
	\end{equation}
	Moreover, if for fixed $p\in (4,\infty)$, $\mathbb{E}\Vert u_0\Vert_{H^1}^p<\infty$, there exists a constant $C=C(p,T,\mathbb{E}\Vert u_0\Vert_{H^1}^2)$, such that
	\begin{equation}\label{viscous estimate p}
		\mathbb{E}\Vert u_\varepsilon\Vert_{L^\infty([0,T];H^1)}^p+\varepsilon^{p/2}\left( \mathbb{E}\int_0^T\Vert\partial_xu_\varepsilon(t)\Vert_{H^1}^2 \,{\rm d}t \right)^{p/2}\leq C.
	\end{equation}
	In particular, if $m>3$, for fixed $p\ge2$, $\forall\ t\ge0$,
	\begin{equation}\label{viscous estimate 2 conservation}\Vert u_\varepsilon(\omega,t)\Vert_{H^1}\le\eta(\omega,t)\|u_0\|_{H^1},\ \mathbb{P}-{\rm a.s.},
	\end{equation}
	where $\eta(\omega,t)=e^{W(t)-\frac t 2}$. 
\end{prop}
\begin{proof}
	\eqref{viscous estimate 2} and \eqref{viscous estimate p} can be obtained in the same way in Proposition \ref{Galerkin estimate}, and we omit the details for brevity.
	
	 Consider the Girsanov type transform \cite{Nathan2014,Michael2014,Tang2018}, denote $\eta(\omega,t)=e^{W(t)-\frac t 2}$, $v_\varepsilon=\eta^{-1}(\omega,t)u_\varepsilon$. Since $u_\varepsilon$ is the global $H^m$ solution to \eqref{CH viscous u} for $m>3$, by It\^{o}'s formula, $v_\varepsilon$ satisfies
	\begin{equation}\label{v epsilon}
		\partial_t v_\varepsilon+\eta v_\varepsilon\partial_xv_\varepsilon-\varepsilon\partial_x^2 v_\varepsilon+\eta\partial_x P[v_\varepsilon]=0,
	\end{equation}
	and $n_\varepsilon=(1-\partial_x^2)v_\varepsilon$ satisfies
	\begin{equation}\label{n epsilon}
		\partial_t n_\varepsilon+\eta v_\varepsilon\partial_xn_\varepsilon-\varepsilon\partial_x^2 n_\varepsilon+2\eta n_\varepsilon\partial_x v_\varepsilon=0.
	\end{equation}
	Integrating by parts, we can infer from \eqref{v epsilon} and \eqref{n epsilon} that for fixed $\omega\in\Omega$,
	$$
	\begin{aligned}
		\frac{\rm d}{{\rm d}t}\|v_\varepsilon\|_{H^1}^2+2\varepsilon\|\partial_x v_\varepsilon\|_{H^1}^2&=\frac{\rm d}{{\rm d}t}\int_{\mathbb{S}}v_\varepsilon\cdot n_\varepsilon\,{\rm d}x-\varepsilon\int_{\mathbb{S}}\partial_x^2v_\varepsilon\cdot n_\varepsilon+\partial_x^2n_\varepsilon\cdot v_\varepsilon\,{\rm d}x\\
		&=-\eta(\omega,t)\int_{\mathbb{S}}\left(v_\varepsilon\partial_x v_\varepsilon+\partial_xP[v_\varepsilon]\right)\cdot n_\varepsilon+\left(v_\varepsilon\partial_x n_\varepsilon+n_\varepsilon\partial_x v_\varepsilon\right)\cdot v_\varepsilon\,{\rm d}x=0
	\end{aligned}
	$$
	Since $\mathbb{E}\|u_0\|_{H^1}^2<\infty$, then for fixed $\omega\in\Omega$, $\|v_0\|_{H^1}=\|u_0\|_{H^1}\le C(\omega)$. Therefore we can conclude that for all $t\ge 0$,
	\begin{equation}\label{v epsilon H1 conservation law}
		\|v_\varepsilon(t)\|_{H^1}^2+2\varepsilon\int_0^t\|\partial_x v_\varepsilon\|_{H^1}^2\,{\rm d}t=\|v_0\|_{H^1}^2.
	\end{equation}
	Thus we can deduce that for a.e. $\omega\in\Omega$ and all $t>0$,
	$$\|u_\varepsilon(\omega,t)\|_{H^1}=\eta(\omega,t)\|v_\varepsilon(\omega,t)\|_{H^1}\le\eta(\omega,t)\|v_0(\omega,t)\|_{H^1}=\eta(\omega,t)\|u_0(\omega,t)\|_{H^1},$$ 
	which implies \eqref{viscous estimate 2 conservation}.
\end{proof}

\begin{prop}[Temporal $L^2$ continuity]\label{viscous estimate Holder}
	Let $u_\varepsilon$ be a solution to \eqref{CH viscous u} with $\mathbb{E}\Vert u_0\Vert_{H^1 }^p<\infty$ for $p>2$. For any $\theta\in[0,\frac {p-2}{4p})$, for any $T>0$, there exists a constant
	$$C=C(T,p,\theta,\mathbb{E}\|u_0\|_{H^1}^{4p})>0,$$
	independent of $\varepsilon$, such that
	\begin{equation}
		\mathbb{E}\|u_\varepsilon\|_{C^\theta([0,T];L^2)}^{2p}\le C.
	\end{equation}
\end{prop}
\begin{proof}
	The proof is similar to Lemma \ref{Galerkin temporal continuity}, thus omitted.
\end{proof}

\begin{lemm}\label{viscous estimate infty}
	Let $u_\varepsilon$ be a solution to \eqref{CH viscous u} with $\mathbb{E}\|u_0\|_{H^m}^{p_0}<\infty$ for some $m>1$ and $p_0>4$. Then for any $T>0$, there exists a constant $C=C(T,\ \|u_0\|_{L^{p_0}(\Omega;H^1)})>0$ independent of $\varepsilon$, such that for any $p\in[1,p_0/2]$
	$$
	\mathbb{E}\|u_\varepsilon\|_{L^\infty([0,\infty)\times\mathbb{S})}^{p_0}\le C,\quad \mathbb{E}\|P_\varepsilon\|_{L^\infty([0,\infty)\times\mathbb{S})}^{p}\le C,
	$$
	where $P_\varepsilon=P[u_\varepsilon]$ defined in \eqref{CH u}.
\end{lemm}
\begin{proof}
	The first part is a direct consequence of Proposition \ref{viscous estimate} and the one-dimensional embedding $H^1\hookrightarrow L^\infty$. By the Young inequality, we have
	$$
	\|P_\varepsilon\|_{L^\infty([0,T]\times\mathbb{S})}\lesssim \|u_\varepsilon\|_{L^\infty([0,T];L^2)}^2+\Vert \partial_x u_\varepsilon\Vert_{L^\infty([0,T];L^2)}^2\lesssim\|u_\varepsilon\Vert_{L^\infty([0,T];H^1)}^2,
	$$
	similarly, we obtain the third part by Proposition \ref{viscous estimate}.
\end{proof}

For the sake of convenience, in the sequel we denote $q_\varepsilon=\partial_x u_\varepsilon$, which satisfies the equation: 
\begin{equation}\label{viscous q}
	0={\rm d}q_\varepsilon-\varepsilon\partial_x^2 q_\varepsilon\,{\rm d}t+\left(\partial_x(u_\varepsilon q_\varepsilon)-\frac 1 2 q_\varepsilon^2-u_\varepsilon^2+P[u_\varepsilon]\right)\,{\rm d}t-q_\varepsilon\,{\rm d}W.
\end{equation}
In order to prove the space-time higher integrability estimate for $q_\varepsilon$ we examine the equation governing the composite function $S(q_\varepsilon)$. Applying It\^{o}'s formula yields
\begin{equation}\label{viscous S(q)}
	\begin{aligned}
		0=&{\rm d}S(q_\varepsilon)-\varepsilon\partial_x^2 S(q_\varepsilon)\,{\rm d}t+\varepsilon S''(q_\varepsilon)|\partial_x q_\varepsilon|^2+\partial_x(u_\varepsilon S(q_\varepsilon))\,{\rm d}t\\
		&-\left(S(q_\varepsilon)q_\varepsilon-\frac 1 2S'(q_\varepsilon)q_\varepsilon^2\right)\,{\rm d}t+\left(S'(q_\varepsilon)(P_\varepsilon-u_\varepsilon^2)-\frac 1 2S''(q_\varepsilon)q_\varepsilon^2\right)\,{\rm d}t-S'(q_\varepsilon)q_\varepsilon\,{\rm d}W.
	\end{aligned}
\end{equation}

\begin{prop}[Space-time higher integrability]\label{viscous q higher}
	Let $u_\varepsilon$ be the $H^m$ regular solution of the viscous SPDE \eqref{CH viscous u} with initial data $u(0) = u_0$ satisfying with $\mathbb{E}\|u_0\|_{H^m}^{p_0}<\infty$ for some $m>1$ and $p_0>4$, and denote by $q_\varepsilon = \partial_x u_\varepsilon$ the spatial gradient of $u_\varepsilon$. For fixed $\alpha \in (0, 1)$, there exists a constant $C = C \bigl( \alpha, T, \| u_0 \|_{L^{p_0}(\Omega; H^1(\mathbb{S}))} \bigr)$, independent of $\varepsilon > 0$, such that
	$$
	\mathbb{E} \| q_\varepsilon \|_{L^{2+\alpha}([0,T] \times \mathbb{S})}^{2+\alpha} \leq C.
	$$
\end{prop}
\begin{proof}
	Consider the function $S(v):=v(|v|+1)^\alpha$, which satisfies
	$$S'(v)=(|v|+1)^\alpha+\alpha|v|(|v|+1)^{\alpha-1}\ \ {\rm and}\ \  S''(v)=\alpha \,{\rm sgn}(v)(|v|+1)^{\alpha-2}(2+(\alpha+1)|v|).$$
	Obviously, $|S'(v)|\le C(|v|+1)$ and $|S''(v)|\le C$ for all $v\in\mathbb{R}$. 
	then integrating the SPDE \eqref{viscous S(q)} over $x\in\mathbb{S}$ gives
	\begin{equation}\label{L}
		L\,{\rm d}t={\rm d}\int_{\mathbb{S}} S(q_\varepsilon)\,{\rm d}x+I_1\,{\rm d}x+I_2\,{\rm d}W,
	\end{equation}
	where
	$$
	\begin{aligned}
		&L=\int_\mathbb{S}\left(S(q_\varepsilon)q_\varepsilon-\frac 1 2S'(q_\varepsilon)q_\varepsilon^2\right)\,{\rm d}x,\\
		& I_1=\int_{\mathbb{S}} S'(q_\varepsilon)(P_\varepsilon-u_\varepsilon^2)+\varepsilon S''(q_\varepsilon)|\partial_x q_\varepsilon|^2-\frac 1 2S''(q_\varepsilon)q_\varepsilon^2\,{\rm d}x,\\
		&I_2=-\int_\mathbb{S}S'(q_\varepsilon)q_\varepsilon\,{\rm d}x.
	\end{aligned}
	$$
	Since $|S'(v)v|\lesssim 1+|v|^2$,so that
	$$\mathbb{E}\int_0^T|I_2|^2\,{\rm d}t\le C(\alpha,T)(1+\mathbb{E}\|q_\varepsilon\|_{L^\infty([0,T];L^2)}^4)\le C(\alpha,T,u_0),
	$$
	by \eqref{viscous estimate 2}, $I_2\in L^2(\Omega\times[0,T])$, so that $\mathbb{E}\int_0^T I_2\,{\rm d}W=0$. 
	
	Continuing, by $|S''(v)| \lesssim 1$, \eqref{viscous estimate 2} and Lemma \ref{viscous estimate infty}, we thus arrive at
	$$
	I_1 \leq \int_{\mathbb{S}} \varepsilon S''(q_\varepsilon) |\partial_x q_\varepsilon|^2\,{\rm d}x + \frac{1}{2} \int_{\mathbb{S}} \bigl| S'(q_\varepsilon) \bigr|^2 \,{\rm d}x + \frac{1}{2} \int_{\mathbb{S}} \bigl( P_\varepsilon - u^2_\varepsilon \bigr)^2 \,{\rm d}x+\frac 1 2\int_{\mathbb{S}}|S''(q_\varepsilon)|q_\varepsilon^2\,{\rm d}x\le C(\alpha,T,u_0).
	$$
	
	Finally, note that
	$$
	S(v)v - \frac{1}{2} S'(v)v^2 = \frac{1}{2} v^2 (|v|+1)^\alpha - \frac{\alpha}{2} |v|^3 (|v|+1)^{\alpha-1} \geq \frac{1-\alpha}{2} |v|^{2+\alpha};
	$$
	hence
	$$
	L \geq \frac{1-\alpha}{2} \int_{\mathbb{S}} |q_\varepsilon|^{2+\alpha} \,{\rm d}x.
	$$
	After integrating \eqref{L} in time, making use of the estimates and also $S(v) \lesssim_\alpha 1+|v|^2$, we arrive at
	$$
	\frac{1-\alpha}{2} \mathbb{E} \int_0^T \int_{\mathbb{S}} |q_\varepsilon|^{2+\alpha} \,{\rm d}x \lesssim C(\alpha,T)\left(1 + \mathbb{E} \|q_\varepsilon\|_{L^\infty([0,T];L^2)}^2\right)\le C(\alpha,T,u_0).
	$$
	This concludes the proof.
\end{proof}

In the following proof, we consider the composite function $S\in W^{2,\infty}_{\rm loc}(\mathbb{R})$ satisfies
\begin{equation}\label{S bound}
	|S(v)|\lesssim|v|^2,\ |S'(v)|\lesssim|v|,\ |S''(v)|\lesssim1,\ {\rm and}\ \left|S(v)v-\frac 1 2S'(v)v^2\right|\lesssim|v|^2,\quad\forall\ v\in\mathbb{R}.
\end{equation}

\begin{prop}[Temporal translation estimate]\label{temporal translation estimate}
	Let $u_\varepsilon$ be the $H^m$ regular solution of the viscous SPDE \eqref{CH viscous u} with initial data $u(0) = u_0$ satisfying with $\mathbb{E}\|u_0\|_{H^m}^{p_0}<\infty$ for some $m>1$ and $p_0>4$, and denote by $q_\varepsilon = \partial_x u_\varepsilon$ the spatial gradient of $u_\varepsilon$. Fix a nonlinear function $S\in W^{2,\infty}_{\rm loc}(\mathbb{R})$ that satisfies \eqref{S bound}. Set $Q_\varepsilon=S(q_\varepsilon)$, then for all $\vartheta\in(0,T\wedge1)$ and $\varphi\in C^\infty(\mathbb{S})$,
	\begin{equation}\label{S temporal estimate}
		\mathbb{E}\sup_{\tau\in(0,\vartheta)}\int_0^{T-\tau}\left|\int_\mathbb{S}\varphi(x)\left(Q_\varepsilon(t+\tau)-Q_\varepsilon(t)\right)\,{\rm d}x\right|\,{\rm d}t\lesssim\|\varphi\|_{C^2(\mathbb{S})}\vartheta^{1/2}.
	\end{equation}
\end{prop}
\begin{proof}
	The nonlinear composition $Q_\varepsilon=S(q_\varepsilon)$ satisfies \eqref{viscous S(q)}, then $\forall\ \vartheta\in(0,T\wedge 1)$, $\varphi\in C^\infty(\mathbb{S})$, we have
	$$
	\left|\int_\mathbb{S}\varphi\left(Q_\varepsilon(t+\tau)-Q_\varepsilon(t)\right)\,{\rm d}x\right|\le \sum_{i=1}^7\int_t^{t+\tau}\int_\mathbb{S} I_i\,{\rm d}x\,{\rm d}t+\left|\int_t^{t+\tau}\int_{\mathbb{S}}\varphi S'(q_\varepsilon)q_\varepsilon\,{\rm d}x\,{\rm d}W\right|,
	$$
	where
	$$
	\begin{aligned}
		&I_1=|\partial_x\varphi| |S(q_\varepsilon)| |u_\varepsilon|,\quad I_2=\varepsilon|\partial_x^2\varphi| |S(q_\varepsilon)|,\quad I_3=\varepsilon|\varphi| |S''(q_\varepsilon)| |\partial_x q_\varepsilon|^2,\\
		&I_4=|\varphi| |S(q_\varepsilon)q_\varepsilon-\frac 1 2 S'(q_\varepsilon)q_\varepsilon^2|,\quad I_5=|\varphi| |S'(q_\varepsilon)| |P_\varepsilon|,\quad I_6=|\varphi| |S'(q_\varepsilon)| |u_\varepsilon|^2,\\
		&I_7=\frac 1 2|\varphi| |S''(q_\varepsilon)| |q_\varepsilon|^2.
	\end{aligned}
	$$
	This implies that 
	$$
		\begin{aligned}
			&\mathbb{E}\sup_{\tau\in(0,\vartheta)}\int_t^{T-\tau}\left|\int_\mathbb{S}\varphi\left(Q_\varepsilon(t+\tau)-Q_\varepsilon(t)\right)\,{\rm d}x\right|\,{\rm d}t\\
			\le& \vartheta \sum_{i=1}^7\mathbb{E}\Vert I_i\Vert_{L^1([0,t]\times\mathbb{S})}+\int_0^{T-\vartheta}\mathbb{E}\sup_{\tau\in(0,\vartheta)}\left|\int_t^{t+\tau}\int_{\mathbb{S}}\varphi S'(q_\varepsilon)q_\varepsilon\,{\rm d}x\,{\rm d}W(s)\right|\,{\rm d}t.
		\end{aligned}
	$$
	
	By the H\"{o}leder inequality, \eqref{S bound}, Proposition \ref{viscous estimate}, and Proposition \ref{viscous estimate infty}, we have
	$$
	\begin{aligned}
		\mathbb{E}\Vert I_1\Vert_{L^1([0,T]\times\mathbb{S})}
		&\lesssim\Vert\varphi\Vert_{C^2(\mathbb{S})}\mathbb{E}\left(\Vert q_\varepsilon^2\Vert_{L^1([0,T]\times\mathbb{S})}\Vert u_\varepsilon\Vert_{L^\infty([0,T]\times\mathbb{S})}\right)\\
		&\lesssim \Vert\varphi\Vert_{C^2(\mathbb{S})}\left(\mathbb{E}\Vert u_\varepsilon\Vert_{L^\infty([0,T]\times\mathbb{S})}^2\right)^{1/2}\left(\mathbb{E}\Vert q_\varepsilon\Vert_{L^2([0,T]\times\mathbb{S})}^4\right)^{1/2}
		\lesssim \Vert\varphi\Vert_{C^2(\mathbb{S})},\quad\quad\quad\quad
	\end{aligned}
	$$
	$$
	\begin{aligned}
		&\mathbb{E}(\Vert I_2\Vert_{L^1([0,T]\times\mathbb{S})}+\Vert I_4\Vert_{L^1([0,T]\times\mathbb{S})}+\Vert I_7\Vert_{L^1([0,T]\times\mathbb{S})})
		\lesssim \varepsilon\Vert\varphi\Vert_{C^2(\mathbb{S})}\mathbb{E}\Vert q_\varepsilon\Vert_{L^2([0,T]\times\mathbb{S})}^2
		\lesssim \Vert\varphi\Vert_{C^2(\mathbb{S})},\\
		&\mathbb{E}\Vert I_3\Vert_{L^1([0,T]\times\mathbb{S})}
		\lesssim \Vert\varphi\Vert_{C^2(\mathbb{S})}\cdot\varepsilon\mathbb{E}\Vert \partial_x^2 u_\varepsilon\Vert_{L^2([0,T]\times\mathbb{S})}^2
		\lesssim \Vert\varphi\Vert_{C^2(\mathbb{S})},\\
	\end{aligned}
	$$
	According to the Young inequality and the definition of $P_\varepsilon$, we have
	$$
	\begin{aligned}
		\Vert P_\varepsilon\Vert_{L^2([0,T]\times\mathbb{S})}
		&\le C(T)\Vert K\Vert_{L^2(\mathbb{S})}\Vert u_\varepsilon^2+\frac 1 2(\partial_x u_\varepsilon)^2\Vert_{L^\infty([0,T];L^1))}\\
		&\le C(\Vert u_\varepsilon\Vert_{L^\infty([0,T];L^2)}^2+\Vert \partial_x u_\varepsilon\Vert_{L^\infty([0,T];L^2)}^2),
	\end{aligned}	
	$$
	thus for $I_5$ and $I_6$, we have
	$$
	\mathbb{E}\left(\Vert I_5\Vert_{L^1([0,T]\times\mathbb{S})}+\Vert I_6\Vert_{L^1([0,T]\times\mathbb{S})}\right)\lesssim \Vert\varphi\Vert_{C^2(\mathbb{S})}\mathbb{E}\left(\Vert q_\varepsilon\Vert_{L^2([0,T]\times\mathbb{S})}\Vert u_\varepsilon\Vert_{L^\infty([0,T];H^1)}^2\right)\lesssim \Vert\varphi\Vert_{C^2(\mathbb{S})}.
	$$
	
	Finally, we turn to the stochastic integral, by the BDG inequality,
	$$
	\begin{aligned}
		\mathbb{E}\sup_{\tau\in(0,\vartheta)}\left|\int_t^{t+\tau}\int_{\mathbb{S}}\varphi S'(q_\varepsilon)q_\varepsilon\,{\rm d}x\,{\rm d}W(s)\right|
		&\lesssim \Vert\varphi\Vert_{C^2(\mathbb{S})}\mathbb{E}\sup_{\tau\in(0,\vartheta)}\left|\int_t^{t+\tau}\int_\mathbb{S}q_\varepsilon^2\, {\rm d}x\,{\rm d}W\right|\\
		&\lesssim \Vert\varphi\Vert_{C^2(\mathbb{S})}\mathbb{E}\sup_{\tau\in(0,\vartheta)}\left(\int_t^{t+\tau}\left(\int_\mathbb{S}q_\varepsilon^2\, {\rm d}x\right)^2\,{\rm d}s\right)\\
		&\lesssim \Vert\varphi\Vert_{C^2(\mathbb{S})}\vartheta^{1/2}\mathbb{E}\Vert q_\varepsilon\Vert_{L^\infty([0,T];L^2)}^2
		\lesssim
		\vartheta^{1/2}\Vert \varphi\Vert_{C^2(\mathbb{S})}.
	\end{aligned}
	$$
	The above estimates yield \eqref{S temporal estimate}.
\end{proof}

\begin{prop}[One-sided supernorm estimate]\label{one side estimate}
	Let $u_\varepsilon$ be the $H^m$ regular solution of the viscous SPDE \eqref{CH viscous u} with initial data $u(0) = u_0$ satisfying with $\mathbb{E}\|u_0\|_{H^m}^{p_0}<\infty$ for some $m>3$ and $p_0>4$, and denote by $q_\varepsilon = \partial_x u_\varepsilon$ the spatial gradient of $u_\varepsilon$. Then the following one-sided $L^\infty$ norm estimate on the first-order spatial derivative holds:
	\begin{equation}\label{one side estimate q}
		\mathbb{P}\left\{\lim_{t\rightarrow\infty}\partial_x u_\varepsilon(t,x)\le0,\ \forall x\in\mathbb{S}\right\}=1.
	\end{equation}
\end{prop}
\begin{proof}
	Consider $v_\varepsilon=\eta^{-1}(\omega,t)u_\varepsilon$ which satisfies \eqref{v epsilon}, with $\eta(\omega,t)=e^{W(t)-\frac t 2}$.
	Since $v_0=u_0$, according to \eqref{v epsilon H1 conservation law}, the following inequality can be proved analogously to Proposition \ref{viscous estimate infty}:
	$$\|P[v_\varepsilon]\|_{L^\infty([0,\infty)\times\mathbb{S})}\le C(\omega,u_0),\quad \|v_\varepsilon\|_{L^\infty([0,\infty)\times\mathbb{S})}\le C(\omega,u_0).
	$$
	
	Consider the equation of $\partial_x v_\varepsilon$,
	\begin{equation}\label{vx epsilon}
		\partial_t\partial_xv_\varepsilon+\eta v_\varepsilon\partial_x(\partial_xv_\varepsilon)-\varepsilon\partial_x^2(\partial_x v_\varepsilon)+\frac \eta 2(\partial_x v_\varepsilon)^2=-\eta(P[v_\varepsilon]-v_\varepsilon^2),\quad\forall\ t\in\mathbb{R}^+,\ x\in\mathbb{S}.
	\end{equation}
	Obviously, $0\le\eta(\omega,t)\le C(\omega)$, then there exists a positive constant $R$, depending only on $\omega$ and $\|u_0(\omega)\|_{H^1}$, so that
	$$
	\partial_t\partial_xv_\varepsilon+\eta v_\varepsilon\partial_x(\partial_xv_\varepsilon)-\varepsilon\partial_x^2(\partial_x v_\varepsilon)\le R,\quad\forall\ t\in\mathbb{R}^+,\ x\in\mathbb{S}.
	$$ 
	Now let $G_\varepsilon$ solve the following ODE system:
	$$
	\left\{
	\begin{array}{l}
		\frac{\rm d}{{\rm d}t}G_\varepsilon=R,\\
		G_\varepsilon(t=0)=\max(0,\partial_x u_0).
	\end{array}
	\right.
	$$     
	Then the comparison principle for a parabolic equation yields for fixed $\omega\in\Omega$,
	$$
	\partial_xv_\varepsilon(\omega,t,x)\le G_\varepsilon= \max(0,\partial_x  u_0(\omega,x))+Rt,\quad \forall\ t\ge 0,\ x\in\mathbb{S}.
	$$
	Therefore, for all $(t,x)\in[0,\infty)\times\mathbb{S}$,
	$$\partial_x u_\varepsilon(\omega,t,x)=\eta(\omega,t)\partial_xv_\varepsilon(\omega,t,x)\le \eta(\omega,t)\big(\max(0,\partial_x  u_0(\omega,x))+Rt\big),
	$$
	since $\lim_{t\rightarrow\infty}\eta(\omega,t)\cdot t=0$, $\mathbb{P}$-a.s., we obtain \eqref{one side estimate q} by passing the limit to $t\rightarrow\infty$ on both sides of the above inequality.
\end{proof}

\begin{theo}[Strong well-posedness of viscous SPDE]\label{viscous H^m uniformly bound}
	Fix $m\ge 1$ and $p_0>4$. Suppose $u_0\in L^{p_0}(\Omega,H^m(\mathbb{S}))$, then there exists a unique $H^m$ pathwise solution to \eqref{CH viscous u} with initial data condition $u_0$. Denoting this solution by $u_\varepsilon$, the following properties and $\varepsilon$-uniform bounds hold:
	\begin{enumerate}
		\item Total energy balance: for any $0\le s\le t\le T$,
		\begin{equation}
			\begin{aligned}
				&\int_\mathbb{S}u_\varepsilon^2+|\partial_x u_\varepsilon|^2\,{\rm d}x\,\bigg|_s^t+2\varepsilon\int_s^t\int_{\mathbb{S}}|\partial_x u_\varepsilon|^2+|\partial_x^2 u_\varepsilon|^2\,{\rm d}x\,{\rm d}t'\\
				=&\int_s^t \int_\mathbb{S}u_\varepsilon^2+|\partial_x u_\varepsilon|^2\,{\rm d}x\,{\rm d}t'+2\int_s^t\int_\mathbb{S}u_\varepsilon^2+|\partial_x u_\varepsilon|^2\,{\rm d}x\,{\rm d}W,\quad \mathbb{P}-a.s.
			\end{aligned}
		\end{equation}
		Furthermore, there exists an $\varepsilon$-independnt positive constant
		$$C=C(p_0,T,\|u_0\Vert_{L^{p_0}(\Omega;H^1(\mathbb{S}))})$$
		such that
		$$\mathbb{E}\Vert u_\varepsilon\Vert_{L^\infty([0,T];H^1)}^p+\varepsilon^{p/2}\left( \mathbb{E}\int_0^T\Vert\partial_xu_\varepsilon(t)\Vert_{H^1}^2 \,{\rm d}t \right)^{p/2}\leq C.$$
		\item For any $\theta\in[0,\frac{p-2}{4p})$, $p\in[2,p_0]$, there exists an $\varepsilon$-independent positive constant $C=C(\theta,T$, $\Vert u_0\Vert_{L^2(\Omega;H^1(\mathbb{S}))})$ such that
		$$
		\mathbb{E}\|u_\varepsilon\|_{C^\theta([0,T];L^2)}^{2/(1-4\theta)}\le C.$$
	\end{enumerate}
\end{theo}

\section{Existence of Martingale Solutions}\label{4}
\quad In what follows, we fixed a sequence $\{\varepsilon_n\}_{n=1}^\infty$ of positive numbers such that $\varepsilon_n\rightarrow 0$ as $n\rightarrow\infty$.

To establish the existence of an $H^1$ martingale solution for equation \eqref{CH u}, we aim to obtain the tightness of the probability laws of the nonlinear terms in the equation by utilizing a family of solutions constructed from the viscous SPDE \eqref{CH viscous u} and their uniform prior estimates with respect to $\varepsilon$. Among these, the most challenging part is the nonlinear term of the form $q_\varepsilon^2$ involving $q_\varepsilon$.In order to verify the tightness of the different probability laws and the limits of relative random variables, we introduce the random mappings:
\begin{equation}\label{F defi}
	F_{n}^q=\left(q_{\varepsilon_n},(q_{\varepsilon_n})_+,(q_{\varepsilon_n})_-\right),\quad F_n^{q^2}=\left(q_{\varepsilon_n}^2,(q_{\varepsilon_n})_+^2,(q_{\varepsilon_n})_-^2\right).
\end{equation} 
Here we denote by $f_+$ and $f_-$ the positive and negative parts of a function $f$, thus, $f=f_+ + f_-$. And the notation $(q_{\varepsilon_n})_\pm^2$ is a concise representation of $\left((q_{\varepsilon_n})_\pm\right)^2$. 

Following the approach in \cite{Holden2024JDE}, we consider the following sequence of convex functions $\{S_l(v)\}_{l\in\mathbb{N}}$ that approximates $\frac 1 2 v^2$ up to some cut-off $|v|\le l$:
\begin{equation}\label{S_l}
	S_l(v) = 
	\begin{cases}
		\frac{1}{2} v^2, & |v| \leq l,\\
		-\frac{1}{6l} |v|^3 + v^2 - \frac{1}{2}l |v| + \frac{1}{6}l^2, & l < |v| < 2l,\\
		\frac{3}{2} l|v| - \frac{7}{6}l^2, & |v| \geq 2l.
	\end{cases}
\end{equation}

\begin{equation}\label{S_l'}
	S_l'(v) = 
	\begin{cases}
		v, & |v| \leq l,\\
		{\rm sgn}(v) \left(2|v| - \frac{1}{2l} v^2 - \frac{1}{2}l\right), & l < |v| < 2l,\\
		\frac{3}{2} \operatorname{sgn}(v)\,l, & |v| \geq 2l.
	\end{cases}
\end{equation}

\begin{equation}\label{S_l''}
	S_l''(v) = 
	\begin{cases}
		1, & |v| \leq l,\\
		\frac 1 l(2l - |v|), & l < |v| < 2l,\\
		0, & |v| \geq 2l.
	\end{cases}
\end{equation}
In particular, $S_l\in W^{3,\infty}_{\rm loc}$, $|S_l(v)|\lesssim_l |v|$, $|S_l'(v)|\lesssim_l 1$, $|S_l''(v)|\lesssim \mathds{1}_{\{|v|\le 2l\}}$, and $|S_l(v)v-\frac 1 2S_l'(v)v^2|\lesssim_l|v|^2$.

The random mappings that we introduce below are motivated by the need to pass to the weak limit in various nonlinear compositions of $q_{\varepsilon_n}$, based on

\begin{equation}\label{S_l of q}
	S(v_\pm),\; S'(v_\pm),\; S''(v_\pm) v^2,\; S(v_\pm) - S'(v_\pm) v,\; S(v_\pm) v - \frac{1}{2} S'(v_\pm) v^2.
\end{equation}

By direct computation, $ (v_\pm)' = \mathds{1}_{\{|v_\pm| > 0\}} $, $ (v_\pm^2)' = 2v_\pm $, $ (v_\pm^2)'' = 2\mathds{1}_{\{|v_\pm| > 0\}} $, and so $ S(v) = v_\pm^2 $ belongs to $ W^{2,\infty}_{\rm loc}(\mathbb{R}) $ and satisfies \eqref{S bound}. And the nonlinear compositions satisfy:
\begin{equation}\label{S_l of q exact}
	\begin{array}{c}
		S_l(v_\pm)' = S_l'(v_\pm), \quad S_l(v_\pm)'' = S_l''(v_\pm) \mathds{1}_{\{|v_\pm| > 0\}},\\
		\\
		S_l(v_\pm) - S_l'(v_\pm) v =
		\begin{cases}
			-\frac{1}{2} v_\pm^2, & |v_\pm| \leq l, \\
			\frac{1}{3l} |v_\pm|^3 - v_\pm^2 + \frac{1}{6}l^2, & l < |v_\pm| < 2l, \\
			-\frac{7}{6}l^2, & |v_\pm| \geq 2l,
		\end{cases}\\
		\\
		S_l(v_\pm) v - \frac{1}{2} S_l'(v_\pm) v^2 =
		\begin{cases}
			0, & |v_\pm| \leq l, \\
			\frac{1}{12l} |v_\pm|^3 v_\pm - \frac{1}{4} |v_\pm| v_\pm l + \frac{1}{6} v l^2_\pm, & l < |v_\pm| < 2l, \\
			\frac{3}{4} |v_\pm| v_\pm l - \frac{7}{6} v_\pm l^2, & |v_\pm| \geq 2l.
		\end{cases}
	\end{array}
\end{equation}
Thus, the nonlinear compositions $v\mapsto S_l(v_\pm)$ belongs to $W^{2,\infty}_{\rm loc}(\mathbb{R})$ and cater to similar bounds.

According to above computations, one can note that $ v \mapsto S_l''(v_\pm) v^2 = S_l''(v_\pm) v^2$ is a continuous function, but $ S_l''(v_\pm)$ is not. Also, the function $ \beta(v) = S_l'(v_\pm) v $ belongs to $ W^{2,\infty}_{\rm loc}(\mathbb{R}) $ although $ S_l'(v_\pm) $ does not, and $\beta(v)$ satisfies \eqref{S bound}. Indeed,
\begin{equation}
	\begin{aligned}
		&\beta'(v) = S_l''(v_\pm) \mathds{1}_{\{|v_\pm| > 0\}} v + S_l'(v_\pm),\\
		&\beta''(v) = S_l'''(v_\pm) \mathds{1}_{\{|v_\pm| > 0\}} v + 2 S_l''(v_\pm) \mathds{1}_{\{|v_\pm| > 0\}},
	\end{aligned}
\end{equation}
so that $|\beta(v)| \lesssim_l |v| $, $ |\beta'(v)| \lesssim_l 1 $, and $ |\beta''(v)| \lesssim_l 1 $.

Equipped with the above renormalization functions, we now consider the following random mappings:
\begin{equation}\label{map Fn l}
	F^{\xi_l,\pm}_n=\xi_l|_{v=q_{\varepsilon_n}},\quad\xi_l\in\mathbf{S}_l^\pm,
\end{equation}
where $\mathcal{S}_l^\pm$ denote the collections of nonlinear functions
\begin{equation}\label{collection Sl +-}
	\mathbf{S}_l^\pm=\left\{S_l(v_\pm),S_l'(v_\pm)v,S_l''(v_\pm)v^2,S_l(v_\pm)v,S_l'(v_\pm)v^2,S_l'(v_\pm)\right\}.
\end{equation}
And also we make use of $F_n^{\mathbf{S}}$ as a notation for the gathering of all these $l$-dependent mappings:
\begin{equation}\label{map Fn}
	F^\mathbf{S}_n=\left\{ \{F^{\xi_l,+}_n,\ \xi_l\in\mathbf{S}_l^+\}_{l\in\mathbb{N}}, \quad\{F^{\xi_l,-}_n,\ \xi_l\in\mathbf{S}_l^-\}_{l\in\mathbb{N}} \right\}.
\end{equation}
Finally, we use $X_n$ as a collective symbol for all the countable mappings:
\begin{equation}\label{map Xn}
	X_n=\left(u_{\varepsilon_n}, F^q_n, F^{q^2}_n, W, z_n, F^{\mathbf{S}}_n\right),
\end{equation}
where $W$ is the Wiener process in \eqref{CH viscous u} and $\{z_n\}_{n=1}^\infty$ is a sequence of $C^\infty$ approximations of initial data $u_0$ satisfying
\begin{equation}\label{u0 approximation}
	z_n\in L^{p_0}(\Omega;C^\infty(\mathbb{S})),\ z_n\xrightarrow{n\rightarrow\infty}u_0\ {\rm in}\ L^{p_0}(\Omega;H^1(\mathbb{S})).
\end{equation}
For convenience, we denote the factors of the infinite vector $X_n$ by $X_n^{(k,l)}$ and the corresponding factor spaces by $\mathcal{X}_{k,l},\ k=1,\dots,21,\ l\in\mathbb{N}$:
$$
\left(X_n^{(k,l)}\right)_{k=1,\dots,9}=\left(u_{\varepsilon_n},F_n^q, F_n^{q^2}, W, z_n\right),\quad \forall\ l\in\mathbb{N}
$$
$$
\left(X_n^{(k,l)}\right)_{k=10,\dots,21}=\left(F^{\xi_l,+}_n,F^{\xi_l,-}_n\right).
$$
Also, we denote that for all $l\in\mathbb{N}$, $\mu_k$ is the corresponding marginals of $\mu_n^{(k,l)}=\mathcal{L}(X_n^{(k,l)})$ defined on $ (\mathcal{X}_k,\mathcal{B}_{\mathcal{X}_k}) :=(\mathcal{X}_{k,l},\mathcal{B}_{\mathcal{X}_{k,l}})$.

The goal is to establish the tightness of the joint probability laws $\mu_n=\mathcal{L}(X_n)$ of random mappings $X_n:(\Omega,\mathcal{F},\mathbb{P})\rightarrow(\mathcal{X},\mathcal{B}_\mathcal{X})$ for the suitable quasi-Polish space $\mathcal{X}$. For fixed $r\in[1,3/2)$, denote the following spaces for the marginals for all $l\in\mathbb{N}$:
\begin{equation}\label{tight quasi Polish}
	\begin{aligned}
		&\mathcal{X}_u=C_tL^2_x,\quad \mathcal{X}_q=\mathcal{X}_{q_\pm}=L^{2r}_{t,x}-w,\quad\mathcal{X}_{q^2}=\mathcal{X}_{q_\pm^2}=L^{r}(L^r_w),
		\\
		&\mathcal{X}_W=C_t,\quad \mathcal{X}_{u_0}=H^1_x,\quad
		\mathcal{X}_{\xi}=L^{2r}_{t,x}-w,\ \xi=S_l(v_\pm)',\\
		&\mathcal{X}_{\xi}=L^{2r}(L^{2r}_w),\ \xi=S_l(v_\pm),\ S_l(v_\pm)'v, \\
		&\mathcal{X}_{\xi}=L^{r}_{t,x}-w,\ \xi=S_l(v_\pm)''v^2,\ S_l(v_\pm)v,\ S_l(v_\pm)'v^2.
	\end{aligned}
\end{equation}
Here the $L^r$ space equipped with weak topology is defined as \eqref{weak topology}.

The path space $\mathcal{X}$ for the joint laws $\{\mu_n\}$ is taken as
$$\mathcal{X}=\left(\prod_{i=1}^9\mathcal{X}_i\right)\bigotimes \left(\prod_{l=1}^\infty(\prod_{i=10}^{21}\mathcal{X}_i)\right),\quad \mathcal{B}_\mathcal{X}=\mathcal{B}(\mathcal{X}),
$$
which carries the product topology for its infinitely many factors. By Lemma \ref{quasi Polish product}, $\mathcal{X}$ is a quasi-Polish space.

\subsection{Tightness and a.s. representation}
\begin{lemm}[Tightness]\label{tightness}
	Consider the random variables $X_n:(\Omega,\mathcal{F},\mathbb{P})\rightarrow(\mathcal{X},\mathcal{B}_\mathcal{X})$ as defined in the preceding definition. Then the sequence $\{\mu_n=\mathcal{L}(X_n)\}_{n\in\mathbb{N}}$ of joint laws is tight, as probability measures defined on $(\mathcal{X},\mathcal{B}_\mathcal{X})$.
\end{lemm}
\begin{proof}
	According to Tychonoff's theorem, the tightness of the joint laws on $\mathcal{X}$ can be deduced from the tightness of the product measures defined on the corresponding qausi-Polish spaces. Hence, it is sufficient to prove that for arbitrary $\delta>0$, for each $k=1,\cdot,21$, and $l\in\mathbb{N}$,there exists a compact sets $\mathcal{K}_{k,\delta}\subset\mathcal{X}_{k}$ such that
	$$\mu_n^{(k,l)}(\mathcal{K}_{k,\delta})>1-\delta\ \Longleftrightarrow\ \mu_n^{(k,l)}(\mathcal{K}_{k,\delta}^c)=\mathbb{P}(X_n^{(k,l)}\in\mathcal{K}_{k,\delta}^c)\leq \delta,
	$$
	uniformly in $n$.
	
	\noindent$\bullet$ Tightness of $\mu_n^u$ in $C_tL^2_x$.
	
	According to the Arzel\`{a}-Ascoli theorem \cite{Simon1987}, we consider the ball $B_M^u$ defined as
	$$B_M^u=\left\{\Vert f\Vert_{C^\theta([0,T];L^2(\mathbb{S}))}+\Vert f\Vert_{L^\infty([0,T];H^1(\mathbb{S}))}\le M\right\},
	$$
	where $\theta\in(0,1/4)$ is fixed. $B_M
	^u$ is relatively compact in $C([0,T];L^2(\mathbb{S}))$. By the Chebyshev's inequality, Proposition \ref{viscous estimate} and Poposition \ref{viscous estimate Holder},
	$$
	\mathbb{P}(u_{\varepsilon_n}\in (B_M^u)^c)<\frac 1 M\left(\mathbb{E}\Vert u_{\varepsilon_n}\Vert_{L^\infty([0,T];H^1(\mathbb{S}))}+\mathbb{E}\Vert u_{\varepsilon_n}\Vert_{C^\theta([0,T];L^2(\mathbb{S}))}\right)\lesssim\frac 1 M,
	$$
	choosing $M$ sufficiently large enough, we obtain the tightness of $\mu_n^u$ in $C_tL^2_x$.
	
	\noindent$\bullet$ Tightness of $\mu_n^q$, $\mu_n^{q_\pm}$ and $\mu_n^{S_l(q_\pm)'}$ in $L^{2r}_{t,x}-w$.
	
	According to Proposition \ref{viscous q higher}, there exists $r\in(1,3/2)$, such that
	$$\mathbb{E}\Vert q_{\varepsilon_n}\Vert_{L^{2r}([0,T]\times\mathbb{S})}^{2r}\le C.$$
	Consider $B_M^q$ defined as
	$$B_M^q=\left\{\Vert f\Vert_{L^{2r}([0,T]\times\mathbb{S})}^{2r}\le M\right\}.$$
	By the Banach-Alaoglu theorem and the reflexivity of $L^r_{t,x}$, $B_M^q$ is a compact subset of $L^r_{t,x}-w$, also by the Chebyshev inequality,
	$$\mathbb{P}\left(q_{\varepsilon_n}\in (B_M^q)^c\right) <\frac 1 M\mathbb{E}\Vert q_{\varepsilon_n}\Vert_{L^{2r}([0,T]\times\mathbb{S})}^{2r}\lesssim \frac 1 M.$$
	Similarly, since $|S_l(v_{\pm})'|\lesssim_l |v|$, the probability laws of $(q_{\varepsilon_n})_\pm$ and $S_l\left((q_{\varepsilon_n})_\pm\right)'$ also tight on $L^{2r}_{t,x}-w$.
	
	\noindent$\bullet$ Tightness of $\mu_n^W$ in $C_t$.
	
	Since the law of $W$ is tight as a Radon measure on the Polish space $C([0,T])$, then there is a compact subset $\mathcal{K}_{W,\delta}$ of $C([0,T])$ such that $\mu_n^W(\mathcal{K}_{W,\delta}^c)\le\delta$.
	
	\noindent$\bullet$ Tightness of $\mu_n^{u_0}$ in $H^1_x$.
	
	By the hypothesis \eqref{u0 approximation} and Chebyshev's inequality, we can easily obtain the tightness of $\mu_n^{u_0}$ on $H^1_x$. 
	
	\noindent$\bullet$ Tightness of $\mu_n^{S_l(q_\pm)q}$, $\mu_n^{S_l(q_\pm)'q^2}$, and $\mu_n^{S_l(q_\pm)''q^2}$ in $L^{r}_{t,x}-w$.
	
	By the definition of $S_l$ \eqref{S_l}, $|S_l(v_+)v|\lesssim_l |v|^2$. Therefore, also by Proposition \ref{viscous q higher},
	$$\mathbb{E}\Vert S_l\left((q_{\varepsilon_n})_+\right)q_{\varepsilon_n}\Vert_{L^r([0,T]\times\mathbb{S})}^r\lesssim_l 1,$$
	here the integrability index $r$ is same as the one one mentioned earlier. Along the same line in the proof of the tightness of $\mu_n^q$ in $L^{2r}_{t,x}-w$, we conclude the tightness of $\mu_n^{S_l(q_\pm)q}$, $\mu_n^{S_l(q_\pm)'q^2}$, and $\mu_n^{S_l(q_\pm)''q^2}$ in $L^{r}_{t,x}-w$.
	
	\noindent$\bullet$ Tightness of $\mu_n^{q^2}$ and $\mu_n^{q_\pm^2}$ in $L^r(L^r_w)$.
	
	Denote $Q_n=q_{\varepsilon_n}^2$, by Propostion \ref{viscous q higher}, 
	$$\mathbb{E}\|Q_n\|_{L^r([0,T]\times\mathbb{S})}\lesssim 1,$$
	uniformly in $n\in\mathbb{N}$.
	
	According to \eqref{viscous estimate 2}, for any $p\in[1,p_0/2]$ with fixed $p_0>4$, 
	$$\left(\mathbb{E}\|Q_n\|_{L^p([0,T];L^1(\mathbb{S}))}\right)^p\lesssim \mathbb{E}\|Q_n\|_{L^p([0,T];L^1(\mathbb{S}))}^p\lesssim_T\mathbb{E}\|q_{\varepsilon_n}\|_{L^\infty([0,T];L^2(\mathbb{S}))}^{2p}\lesssim 1.$$
	Thus, by Proposition \ref{temporal translation estimate} and tightness criterion Lemma \ref{tightness criterion}, the sequence $\{\mathcal{L}(q_{\varepsilon_n}^2)\}_{n\in\mathbb{N}}$ of probability laws is tight on $L^r(L^r_w)$. We can handle $\mu_n^{q_\pm^2}$ the same way.
	
	\noindent$\bullet$ The tightness of $\mu_n^{S_l(q_\pm)}$ and $\mu_n^{S_l(q_\pm)'q}$ in $L^{2r}(L^{2r}_w)$.
	
	Also consider the definition of $S_l$, $S_l(v_\pm),\ S_l(v_\pm)'v\in W^{2,\infty}_{\rm loc}$ and
	$|S_l(v_\pm)|,\ |S_l(v_\pm)'v|\lesssim_l |v|$, the remainder of the proof is similar and thus omitted.
\end{proof}

With the tightness of the joint laws of $\mu_n$ established, we can directly apply the Skorokhod-Jakubowski a.s. theorem to derive the following result:
\begin{prop}[Skorokhod-Jakubowski representations]\label{SJ representations}
	Fix a sequence $\{\varepsilon_n\}$ of positive numbers with $\varepsilon_n\rightarrow0$ as $n\rightarrow\infty$, and consider the corresponding $H^m$ pathwise solution $u_{\varepsilon_n}$ of the viscous SPDE \eqref{CH viscous u} with initial data $z_n$ defined in \eqref{u0 approximation}. Denote the spatial gradient by $q_{\varepsilon_n}=\partial_xu_{\varepsilon_n}$.
	
	Consider the mappings $X_n:(\Omega,\mathcal{F},\mathbb{P})\rightarrow(\mathcal{X},\mathcal{B}_\mathcal{X})$ defined by \eqref{map Xn}. There exists a new probability space $(\tilde{\Omega},\tilde{\mathcal{F}},\tilde{\mathbb{P}})$ and $\mathcal{X}$-valued random variables
	\begin{equation}\label{tilde X}
		\tilde{X}_n = \bigl( \tilde{u}_n, \tilde{F}^q_n, \tilde{F}^{q^2}_n, \tilde{W}_n, \tilde{u}_{0,n}, \tilde{F}^\mathbf{S}_n \bigr), \quad
		\tilde{X} = \bigl( \tilde{u},\overline{\tilde{F}^q}, \overline{\tilde{F}^{q^2}}, \tilde{W}, \tilde{u}_0, \overline{\tilde{F}^\mathbf{S}}  \bigr)
	\end{equation}
	defined on $(\tilde{\Omega}, \tilde{\mathcal{F}}, \tilde{\mathbb{P}})$, with
	\begin{equation}\label{tilde F}
		\begin{aligned}
			&\tilde{F}_n^{\xi_l,\pm}=\xi_l|_{v=\tilde{q}_{\varepsilon_n}},\quad \xi_l\in\mathbf{S}_l^\pm,\ l\in\mathbb{N},\\
			&\tilde{F}_n^\mathbf{S}=\left\{
			\left\{\tilde{F}_n^{\xi_l,+},\ \xi_l\in\mathbf{S}_l^+\right\},\ 
			\left\{\tilde{F}_n^{\xi_l,-},\ \xi_l\in\mathbf{S}_l^-\right\}
			\right\},
		\end{aligned}
	\end{equation}
	where $\mathbf{S}_l^\pm$ denote the collections of nonlinearities given by \eqref{collection Sl +-}.
	Then along a subsequence (notationally not relabelled) the joint laws of $\tilde{X}_n$ and $X_n$ coincide for all $n$, and $\tilde{X}_n\xrightarrow{n\rightarrow\infty}\tilde{X}$ a.s. in the product topology on $\mathcal{X}$. 
	More explicitly, we have the following $\tilde{\mathbb{P}}$-a.s. convergences:
	\begin{equation}\label{SJ convergence}
		\begin{aligned}
			&\tilde{u}_n\rightarrow\tilde{u}\ {\text in}\ C_tL^2_x,\quad \tilde{W}_n\rightarrow\tilde{W}\ {\text in}\ C_t,\quad \tilde{u}_{0,n}\rightarrow\tilde{u}_0\ {\text in}\ H^1_x,\\
			&\tilde{q}_n\rightarrow\tilde{q},\ (\tilde{q}_n)_\pm\rightarrow\overline{\tilde{q}_\pm},\ S_l\left((\tilde{q}_n)_\pm\right)'\rightarrow\overline{S_l\left((\tilde{q})_\pm\right)'}\ {\text in}\ L^{2r}_{t,x},\\
			&\tilde{q}_n^2\rightarrow\overline{\tilde{q}^2},\ (\tilde{q}_n)_\pm^2\rightarrow\overline{\tilde{q}_\pm^2}\ {\text in}\ L^r(L^r_w),\\
			&S_l\left((\tilde{q}_n)_\pm\right)\rightarrow\overline{S_l\left((\tilde{q})_\pm\right)},\ S_l\left((\tilde{q}_n)_\pm\right)'\tilde{q}_n\rightarrow\overline{S_l\left((\tilde{q})_\pm\right)'\tilde{q}}\ {\text in}\ L^{2r}(L^{2r}_w),\\
			&S_l\left((\tilde{q}_n)_\pm\right)''\tilde{q}_n^2\rightarrow\overline{S_l\left((\tilde{q})_\pm\right)''\tilde{q}^2},\ 
			S_l\left((\tilde{q}_n)_\pm\right)\tilde{q}_n\rightarrow\overline{S_l\left((\tilde{q})_\pm\right)\tilde{q}}\ {\text in}\ L^r_{t,x}\\
			&S_l\left((\tilde{q}_n)_\pm\right)'\tilde{q}_n^2\rightarrow\overline{S_l\left((\tilde{q})_\pm\right)'\tilde{q}^2}\ {\text in}\ L^r_{t,x},
		\end{aligned}
	\end{equation}
	as $n\rightarrow\infty$, for all $l\in\mathbb{N}$. In particular, $\overline{\tilde{q}}=\tilde{q}$.
\end{prop}
\begin{proof}
	This is an application of Theorem \ref{SJ reprensentation theorem} and Lemma \ref{representation nonlinear}.
\end{proof}

For the sake of convenience, we drop the tilde under the overline indicating a weak limit in the following sections, for example, $\overline{S_l(\tilde{q}_\pm)}$ we write as $\overline{S_l(q_\pm)}$.

\subsection{Properties of a.s. representation}
\quad At the moment, we do not have the SPDE satisfied by the Skorokhod-Jakubowski representations $\tilde{u}_n$ and $\tilde{u}$. To derive the SPDE satisfied by the limit $\tilde{u}$ of $\tilde{u}_n$, we need a priori regularity estimates for $\tilde{u}_n$ to justify the SPDE that $\tilde{u}_n$ satisfies. The law shared by $u_{\varepsilon_n}$ and $\tilde{u}_n$ can be then integrated against over better function spaces to derive bounds for $\tilde{u}_n$ from $u_{\varepsilon_n}$.

Using the KLS theorem, we can relate the prior estimates of the a.s. representations $\tilde{u}_n$ and $\tilde{q}_n$ in the new probability space $(\tilde{\Omega},\tilde{F},\tilde{\mathbb{P}})$ to those of $u_{\varepsilon_n}$ and $q_{\varepsilon_n}=\partial_x u_{\varepsilon_n}$ in the original probability space respectively, and establish the corresponding prior estimates $\tilde{u}$ via the convergence between $\tilde{u}_n$ and $\tilde{u}$. Since this argument relies solely on the prior estimates of the viscous equation \eqref{CH viscous u} and some basic functional analysis methods, the detailed proof is presented in a concise manner.

\begin{lemm}[A priori estimates]\label{priori estimate}
	Let $p_0>4$ be as specified in Theorem \ref{viscous H^m uniformly bound} and $r\in[1,3/2)$ as fixed in \eqref{tight quasi Polish}. Let $\tilde{u}_n$, $\tilde{q}_n$ be the Skorokhod-Jakubowski representations and $\tilde{u}$, $\tilde{q}$, $\overline{q^2}$ be the a.s. limits from Proposition \ref{SJ representations}. 
	\begin{enumerate}
		\item Then $\tilde{u}_n\in L^2([0,T];H^m(\mathbb{S}))$, $\mathbb{P}$-a.s., for any $m\ge 1$, and $\tilde{q}_n(t,x)=\partial_x\tilde{u}_n(t,x)$ pointwise in $(t,x)$, $\mathbb{P}$-a.s. And
		there exists a constant $C$ independent of $n$, such that
		$$\tilde{\mathbb{E}}\Vert \tilde{u}_n\Vert_{L^\infty([0,T];H^1(\mathbb{S}))}^{p_0}\le C,\quad \tilde{\mathbb{E}}\Vert \tilde{q}_n\Vert_{L^\infty([0,T];L^2(\mathbb{S}))}^{p_0}\le C,\quad \tilde{\mathbb{E}}\Vert \tilde{q}_n\Vert_{L^{2r}([0,T]\times\mathbb{S})}^{2r}\le C.
		$$ 
		Besides, there exists an event $\tilde{F}$ with $\tilde{\mathbb{P}}(\tilde{F})=1$, such that for any $\tilde{{\omega}}\in\tilde{F}$ there exists $E(\tilde{\omega})\subset[0,T]\times\mathbb{S}$ of full measure, then
		$$\partial_x\tilde{u}(\tilde{{\omega}},t,x)=\tilde{q}(\tilde{{\omega}},t,x)\text{ for a.e }(t,x)\in E(\tilde{{\omega}}).$$
		\item There exists a constant of $n$, such that
		$$
		\tilde{\mathbb{E}}\Vert \tilde{u}_n\Vert_{L^\infty([0,T]\times\mathbb{S})}^{p_0}\le C,\quad \tilde{\mathbb{E}}\Vert P[\tilde{u}_n]\Vert_{L^\infty([0,T]\times\mathbb{S})}^{p_0}\le C.
		$$
		In particular, we have
		$$
		\tilde{\mathbb{E}}\left\Vert K\ast\left(\tilde{u}_n^2+\frac 1 2\tilde{q}_n^2\right)\right\Vert_{L^\infty([0,T]\times\mathbb{S})}^{p}\le C,\quad p\in [1,p_0/2].
		$$
		\item The limit $\tilde{u}\in C([0,T];H^1(\mathbb{S})-w)$, $\tilde{\mathbb{P}}$-a.s., and there exists a constant $C$ such that
		$$\tilde{\mathbb{E}}\Vert \tilde{u}\Vert_{L^\infty([0,T];H^1(\mathbb{S}))}^{p_0}\le C,\quad \tilde{\mathbb{E}}\Vert \tilde{q}\Vert_{L^\infty([0,T];L^2(\mathbb{S}))}^{p_0}\le C,\quad \tilde{\mathbb{E}}\Vert \tilde{q}\Vert_{L^{2r}([0,T]\times\mathbb{S})}^{2r}\le C.
		$$
		\item For $\overline{q^2}$, there exists a constant $C$ such that for any $p\in[1,p_0/2]$,
		\begin{equation}\label{priori estimate q}
			\tilde{\mathbb{E}}\int_0^T\int_{\mathbb{S}}\left|\overline{q^2}\right|^r\,{\rm d}x\,{\rm d}t\le C,\quad \tilde{\mathbb{E}}\int_0^T\left\|\overline{q^2}\right\|^p_{H^{-1}(\mathbb{S})}\,{\rm d}t\le C.
		\end{equation}
		Furthermore,
		\begin{equation}\label{priori estimate q convolution}
			\tilde{\mathbb{E}}\int_0^T\left\|K\ast\left(\tilde{u}^2+\frac 1 2\overline{q^2}\right)(t)\right\|^p_{L^\infty(\mathbb{S})}\,{\rm d}t\le C.
		\end{equation}
	\end{enumerate}
\end{lemm}
\begin{proof}
	\textbf{(1)} Since the continuous injection from $C([0,T];H^1(\mathbb{S}))\hookrightarrow C([0,T];L^2(\mathbb{S}))$ and $L^{2r}([0,T]\times\mathbb{S})\hookrightarrow L^{2r}([0,T]\times\mathbb{S})-w$, the Kuratowski–Lusin–Souslin theorem \cite{Ondrejat2011,Ondrejat2010} and the equality of laws implies the a priori estimates given in (1).
	
	Let $\{\Psi_j\}_{j=1}^\infty$ be a countable dense subset of $C([0,T];C^1(\mathbb{S}))$, denote $G_j(u,q)=\int_0^T\int_\mathbb{S}u\,\partial_x\Psi_j\,{\rm d}x\,{\rm d}t+\int_0^T\int_{\mathbb{S}}q\,\Psi_j\,{\rm d}x\,{\rm d}t$. Since $G_j$ is continuous, by the equality of joint laws,
	$$\tilde{\mathbb{P}}\left(\left\{G_j(\tilde{u}_n,\tilde{q}_n)=0\right\}\right)=\mathbb{P}\left(\left\{G_j(u_{\varepsilon_n},q_{\varepsilon_n}=0)\right\}\right)=1.$$
	
	Due to the pairs $(n,j)$ is countable, there exists a set $\tilde{\mathcal{G}}$ of full $\tilde{\mathbb{P}}$-measure such that $G_j(\tilde{u}_n(\tilde{{\omega}}),\tilde{q}_n(\tilde{\omega}))$ $=0$ for all $(n,j)$, $\tilde{\omega}\in\tilde{\mathcal{G}}$. According to the convergence \eqref{SJ convergence}, we can deduce that
	$$\int_0^T\int_{\mathbb{S}}\Psi\,\tilde{q}\,{\rm d}x\,{\rm d}t=-\int_0^T\int_{\mathbb{S}}\partial_x\Psi\,\tilde{u}\,{\rm d}x\,{\rm d}t,\quad\tilde{\mathbb{P}}-{\rm a.s.}$$
	
	Since $u_{\varepsilon_n}$ is the strong solution of viscous SPDE \eqref{CH viscous u} with initial data $u_{\varepsilon_n}(0)=z_n$, thus, $u_{\varepsilon_n}\in L^2([0,T];H^m(\mathbb{S}))$ for any $m> 1$. Moreover, by the equality of laws, $\tilde{u}_n$ also belongs to $L^2([0,T];H^m(\mathbb{S}))$, and $\tilde{q}_n$ is the weak $x$-derivative of $\tilde{u}_n$. Since $H^m(\mathbb{S})\hookrightarrow C^{m-1/2}(\mathbb{S})$, the $x$-derivative of $\tilde{u}_n$ is classical, thus, $\tilde{q}(t,x)=\partial_x\tilde{u}_n(t,x)$ pointwise in $(t,x)$, $\tilde{\mathbb{P}}$-a.s.
	
	\noindent\textbf{(2)} This can be obtained by the standard H\"{o}lder and Young inequalities, and estimates in (1).
	
	\noindent\textbf{(3)} By the estimate in (1), $\tilde{\mathbb{E}}\Vert \tilde{u}_n\Vert_{L^{\bar{r}}([0,T];H^1(\mathbb{S}))}^{p_0}\le C T^{p_0/\bar{r}}$, for any $\bar{r}\in[1,\infty)$,i.e., $\tilde{u}_n$ uniformly bounded in $L^{p_0}\left(\tilde{\Omega};L^{\bar{r}}([0,T];H^1(\mathbb{S}))\right)$ for any finite $\bar{r}$. Since $L^{p_0}\left(\tilde{\Omega};L^{\bar{r}}([0,T];H^1(\mathbb{S}))\right)$ is reflexive, according to Kakutani' theorem, up to a subsequences
	\begin{equation}
		\tilde{u}_n\rightharpoonup v^{(\bar{r})}\quad {\rm in}\ L^{p_0}\left(\tilde{\Omega};L^{\bar{r}}([0,T];H^1(\mathbb{S}))\right).
	\end{equation}
	Besides, for any $\bar{r}<\infty$,
	$$\tilde{\mathbb{E}}\left\Vert v\right\Vert_{L^{\bar{r}}([0,T];H^1(\mathbb{S}))}^{p_0}\le\liminf_{n\rightarrow\infty}\tilde{\mathbb{E}}\left\Vert\tilde{u}_n\right\Vert_{L^{\bar{r}}([0,T];H^1(\mathbb{S}))}^{p_0}\le C T^{p_0/\bar{r}}.$$
	By the monotone convergence theorem,
	$$\tilde{\mathbb{E}}\lim_{\bar{r}\rightarrow\infty}\Vert v\Vert_{L^{\bar{r}}([0,T];H^1(\mathbb{S}))^{p_0}}\le C.$$
	Since the $L^r$ norm depends continuously on the index $r$ for any measurable function $f : [0,T] \to H^1(\mathbb{S})$ for which  $\lim_{r \to \infty} \left( \int_0^T \| f(t) \|^r_{H^1(\mathbb{S})} \, dt \right)^{1/r} < \infty,$
	it follows that  
	\begin{equation}\label{v estimate}
		\tilde{\mathbb{E}} \| v \|^{p_0}_{L^\infty([0,T]; H^1(\mathbb{S}))} \le C
		\text{ and }  
		v \in L^\infty([0,T]; H^1(\mathbb{S})), \quad \tilde{\mathbb{P}}\text{-a.s.}
	\end{equation}

	Consider test functions:
	$$\phi(\tilde{\omega},t,x)=\psi(\tilde{\omega})\varphi(t,x),\quad \psi\in L^\infty(\tilde{\Omega}),\varphi\in L^{\bar{r}'}([0,T];L^2(\mathbb{S})).$$
	Therefore,
	$$\tilde{\mathbb{E}}\left(\psi\int_0^T\int_\mathbb{S}\varphi(t,x)\left(\tilde{u}_n(t,x)-v(t,x)\right)\,{\rm d}x\,{\rm d}t\right)\rightarrow 0.$$
	On the other hand, by \eqref{SJ convergence}
	$$\psi\int_0^T\int_{\mathbb{E}}\varphi(t,x)\left(\tilde{u}_n(t,x)-\tilde{u}(t,x)\right)\,{\rm d}x\,{\rm d}t\rightarrow0,\quad\tilde{\mathbb{P}}-{\rm a.s.}$$
	then by the estimate of $\tilde{u}_n$ claimed in (1) and Vitali's convergence theorem, for any $1\le p<p_0$,
	$$\tilde{\mathbb{S}}\left|\psi\int_0^T\int_{\mathbb{S}}\varphi(t,x)\left(\tilde{u}_n(t,x)-\tilde{u}(t,x)\right)\,{\rm d}x\,{\rm d}t\right|^p\rightarrow0,$$
	consequently,
	\begin{equation}\label{tilde u=v}
		\tilde{\mathbb{E}}\left(\psi\int_0^T\int_{\mathbb{S}}\varphi(t,x)\left(\tilde{u}(t,x)-v(t,x)\right)\,{\rm d}x\,{\rm d}t\right)=0.
	\end{equation}
	
	Denote $I_z(\varphi)=\int_0^T\int_{\mathbb{S}}\varphi(t,x)\,z(\cdot,t,x)\,{\rm d}x\,{\rm d}t$, \eqref{v estimate} and \eqref{tilde u=v} yield that for any fixed $\varphi$, $I_{\tilde{u}}(\varphi)=I_v(\varphi)\in L^p(\tilde{\Omega})$, for any $1\le p<p_0$. For any $1<\bar{r}'<\infty$ there exists a full $\tilde{\mathbb{P}}$-measure set on which $I_{\tilde{u}}(\varphi)=I_v(\varphi)$ holds for all $\varphi\in L^{\bar{r}'}([0,T];L^2(\mathbb{S}))$. By separability of $L^{\bar{r}'}([0,T];L^2(\mathbb{S}))$, we can deduce that $\tilde{u}=v$, $\tilde{\mathbb{P}}\otimes {\rm d}t\otimes{\rm d}x$-almost everywhere. Therefore, we conclude that
	$$\tilde{\mathbb{E}} \| \tilde{u} \|^{p_0}_{L^\infty([0,T]; H^1(\mathbb{S}))} \le C.$$
	
	By Lemma \ref{tightness}, for any $L\in\mathbb{N}$, there exists a compact set $K_L$ such that $\inf_{n\in\mathbb{N}}\mathbb{P}\left(\left\{u_{\varepsilon_n}\in K_L\right\}\right)>1-1/L$. Thus, by the equality of laws,
	$$ \inf_{n\in\mathbb{N}}\tilde{\mathbb{P}}\left(\left\{\tilde{u}_n\in K_L\right\}\right)>1-1/L.$$
	Pick an arbitrary subsequence $\{n_j\}_{j\in\mathbb{N}}$ and set $A_{j,L}=\left\{\tilde{u}_{n,j}\in K_L\right\}$. Then
	$$\tilde{\mathbb{P}}\left(\limsup_j A_{j,L}\right)>1-1/L,$$
	where $\limsup_j A_{j,L}=\cup_{J=1}^\infty\cap_{j>J} A_{j,L}$. 
	
	Introduce the $C_t L^2_x$ convergence set $F=\{\tilde{\omega}\in\tilde{\Omega}:\tilde{u}_n(\tilde{\omega})\rightarrow\tilde{u}(\tilde{\omega})\ {\rm in}\ C_t L^2_x\}$. According to \eqref{SJ convergence}, $\tilde{\mathbb{P}}(F)=1$. Therefore,
	$$\tilde{\mathbb{P}}\left(F\cap\limsup_j A_{j,L}\right)>1-1/L.$$
	
	Select an arbitrary $\tilde{\omega}_0 \in F \cap \limsup_{j} A_{j,L}$, by construction, there is a subsequence $\tilde{u}_{n_{j_k}}(\tilde{\omega}_0) \in K_L$ for all $k \in \mathbb{N}$. Moreover,
	$
	\tilde{u}_{n_{j_k}}(\tilde{\omega}_0) \rightarrow \tilde{u}(\tilde{\omega}_0) \quad \text{in } C_t L_x^2
	$
	and whence 
	$
	\tilde{u}_{n_{j_k}}(\tilde{\omega}_0) \rightarrow\tilde{u}(\tilde{\omega}_0) \quad \text{in } C_t H_x^1-w.
	$
	Since $\{n_j\}_{j \in \mathbb{N}}$ was arbitrary, and $K_L$ is metrisable in $C_t H_x^1-w$, then
	$$
	\tilde{u}_n(\tilde{\omega}_0) \rightarrow \tilde{u}(\tilde{\omega}_0) \quad \text{in } C_t H_x^1-w
	\quad
	\text{and}\quad\tilde{u}(\tilde{\omega}_0) \in K_L.$$ 
	In other words, the convergence set 
	$$
	M_L = \left\{ \tilde{\omega} \in \tilde{\Omega} : \tilde{u}_n(\tilde{\omega}) \rightarrow \tilde{u}(\tilde{\omega}) \text{ in } C_t H_x^1-w,\; \tilde{u}(\tilde{\omega}) \in K(a_L) \right\}
	$$
	satisfies $M_L \supseteq F \cap \limsup_j A_{j,L}$, and so $\tilde{P}(M_L) > 1 - 1/L$, for any $L \in \mathbb{N}$. Denote 
	$$
	M = \left\{ \tilde{\omega} \in \tilde{\Omega} : \tilde{u}_n(\tilde{\omega}) \rightarrow \tilde{u}(\tilde{\omega}) \text{ in } C_t H_x^1-w \right\}.
	$$
	Then $M \supseteq M_L \supseteq F\cap\limsup_j A_{j,L} \in \tilde{\mathcal{F}}$, for any $L \in \mathbb{N}$, and therefore we obtain 
	$
	M \supseteq \bigcup_{L=1}^\infty \bigg( F \cap $ 
	$\limsup_j A_{j,L} \bigg) \in \tilde{\mathcal{F}}.
	$
	This implies that 
	$$
	\tilde{P}\left( \bigcup_{L=1}^\infty \left( F \cap \limsup_j A_{j,L} \right) \right) \ge \tilde{P}\left( F \cap \limsup_j A_{j,L} \right) > 1 - \frac{1}{L},
	$$
	for each $L \in \mathbb{N}$, i.e., $\tilde{P}\left( \bigcup_{L=1}^\infty \left( F \cap \limsup_j A_{j,L} \right) \right) = 1$. By the completeness of $(\tilde{\Omega}, \tilde{\mathcal{F}}, \tilde{P})$, it follows that $M \in \tilde{\mathcal{F}}$ and $\tilde{P}(M) = 1$; thus $\tilde{u} \in C_t H_x^1-w$, $\tilde{P}$–a.s.
	
	The remaining claims in (3) can be deduce from the estimate of $\tilde{u}$ and the corresponding inequality in (1), here we omit the relevant computational proofs.

	\noindent\textbf{(4)} Since $\tilde{q}_n^2$ uniformly bounded in $L^r(\tilde{\Omega}\times[0,T]\times\mathbb{S})$, by compactness argument we can assume $\overline{q^2}L^r(\tilde{\Omega}\times[0,T]\times\mathbb{S})$, according to the embedding $L^r_w(\mathbb{S})\hookrightarrow H^{-1}(\mathbb{S})$ and
	$\Vert f\Vert_{H^{-1}(\mathbb{S})}\le \Vert f\Vert_{L^2(\mathbb{S})}$, this can also be obtained by the standard H\"{o}lder and Young inequalities, interpolation inequality and estimates in (3).
\end{proof}

We denote by $\sum(E)$ the smallest $\sigma$-algebra containing a collection $E$ of subsets of $\tilde{\Omega}$. Given $\{\tilde{\mathcal{F}_t^n}\}_{t\in[0,T]}$ and $\{\tilde{\mathcal{F}}_t\}_{t\in[0,T]}$ to be the $\tilde{\mathbb{P}}$-argumented canonical filtrations of the processes $\tilde{X}_n$ and $\tilde{X}$, respectively. More specifically, for example, the filtration and corresponding stochastic basis of $\tilde{X}$ can be written in the following form:
\begin{equation}\label{filtration}
	\begin{aligned}
		\tilde{F}_t&=\sum\left(\sum(\tilde{X}|[0,t])\cup\left\{N\in\tilde{F}:\tilde{\mathbb{P}}(N)=0\right\}\right),\ t\in [0,T],\\
		&{\rm and}\  \tilde{S}=(\tilde{\Omega},\tilde{F},\{\tilde{\mathcal{F}_t}\}_t\in[0,T],\tilde{\mathbb{P}}).
	\end{aligned}
\end{equation}
 Similarly, we can also define $\{\tilde{\mathcal{F}_t^n}\}_t\in[0,T]$ and $\tilde{S}^n$.

Next, by L\'{e}vy's characterization theorem, under the given filtration, $\tilde{W}_n$ and $\tilde{W}$ are verified to be a Wiener processes relative to the filtrations $\{\tilde{\mathcal{F}_t^n}\}_{t\in[0,T]}$ and $\{\tilde{\mathcal{F}}_t\}_{t\in[0,T]}$, respectively.

\begin{lemm}[Wiener process]\label{Wiener process}
	The a.s. representations $\tilde{W}_n$ and $\tilde{W}$ from Proposition \ref{SJ representations} are Wiener processes defined on the stochastic basis $\tilde{\mathcal{S}}^n$ and $\tilde{\mathcal{S}}$ respectively, cf. \eqref{filtration}.
\end{lemm}
\begin{proof}
	Since $\tilde{W}_n$ and $W$ share the same law, and $W(t)$ is $\mathcal{F}_t$-measurable and $W(t)-W(s)$ is independent of $\mathcal{F}_s$ for all $s<t$, thus, it is easy to check that $\tilde{W}_n$ is a Wiener process on $\tilde{S}^n$. By L\'{e}vy's characterization theorem \cite{Karatzas1991}, it remains to prove that $\tilde{W}$ is a $\tilde{F}_t$ martingale.
	
	Let $\gamma:\mathcal{X}|_{[0,s]}\rightarrow[0,1]$ be a continuous function. Clearly, $\mathcal{X}|_{[0,s]}$ is quasi-Polish space, and the restriction operator $\mathcal{R}_s:Z\rightarrow Z|_{[0,s]}$ is continuous. Hence, $X_n|_{[0,s]}=\mathcal{R}_s\circ X_n$ is $\mathcal{F}_s/\otimes_{l\in\mathbb{N}}\mathcal{B}_{\mathcal{X}_l}|_{[0,s]}$ measurable, where the conutable vector $X_n$ defined in \eqref{representation nonlinear} and $\mathcal{B}_{\mathcal{X}_l}|_{[0,s]}$ denotes the Borel $\sigma$-algebra of the restricted space $\mathcal{X}_l|_{[0,s]}$.
	
	According to the equality of laws $\tilde{W}_n$ and $W$, for any $0\le s\le t\le T$ and for any $n\in \mathbb{N}$,
	$$\tilde{\mathbb{E}}\left[\gamma\left(\tilde{X}_n|_{[0,s]}\right) \left(\tilde{W}_n(t)-\tilde{W}_n(s)\right)\right]=\mathbb{E}\left[\gamma\left(X_n|_{[0,s]}\right) \left(W(t)-W(s)\right)\right].$$
	Since for any finite $p$, by the BDG inequality, $\tilde{\mathbb{E}}\Vert\tilde{W}_n\Vert_{C([0,T])}^p=\mathbb{E}\|W\|_{C([0,T])}\le C(T,p)$.
	
	Hence, by Vitali's convergence theorem \cite{Kuo2006} and \eqref{SJ convergence}, we can pass to the limit $n\rightarrow\infty$ in the left-hand side of the above inequality, $$\tilde{\mathbb{E}}\left[\gamma\left(\tilde{X}|_{[0,s]}\right)\left(\tilde{W}(t)-\tilde{W}(s)\right)\right]=0.$$
\end{proof}

\subsection{Martingale solutions}
\quad Next, we claim that the Skorokohod-Jakubowski representation $\tilde{u}_n$ satisfies the same SPDE of $u_{\varepsilon_n}$ on the new probability space. Here we are going to a general method discovered by Brze\'{z}niank and Ondrej\'{a}t \cite{Ondrejat2011,Ondrejat2010}, and used in several other works  \cite{Breit2016,Debussche2016,Holden2024JDE,Hofmanov2013,Tang2018}.

To facilitate the proof, we first define the functional notation required in the proof as follows, defined for $(u,v,z)\in\mathcal{X}_u\times\mathcal{X}_{q^2}\times\mathcal{X}_{u_0}$ and $t\in [0,T]$:
\begin{equation}\label{Mn}
	\begin{array}{c}
		M_n[u,v,z](t)=\int_{\mathbb{S}}\varphi\,u(t)\,{\rm d}x-\int_{\mathbb{S}}\varphi\,z\,{\rm d}x-\varepsilon_n\int_0^t\int_{\mathbb{S}}\partial_x^2\varphi\,u\,{\rm d}x\,{\rm d}s-\int_0^t\int_{\mathbb{S}}\partial_x\varphi\left(\frac 1 2 u^2+P[u,v]\right),\\[2ex]
		M[u,v,z](t)=\int_{\mathbb{S}}\varphi\,u(t)\,{\rm d}x-\int_{\mathbb{S}}\varphi\,z\,{\rm d}x-\int_0^t\int_{\mathbb{S}}\partial_x\varphi\left(\frac 1 2 u^2+P[u,v]\right),
	\end{array}
\end{equation}
\begin{equation}\label{Dn}
		D_n[u,v,z](t)=M_n[u,v,z](t)-\int_0^t\int_\mathbb{S}\varphi\,u\,{\rm d}x\,{\rm d}s,\quad D[u,v,z]=M[u,v,z](t)-\int_0^t\int_\mathbb{S}\varphi\,u\,{\rm d}x\,{\rm d}s,
\end{equation}
\begin{equation}\label{R}
	R[u](t)=\int_0^t\left|\int_\mathbb{S}\varphi\, u\,{\rm d}x\right|^2\,{\rm d}t,
\end{equation}
\begin{equation}\label{N}
	N[u](t)=\int_0^t\int_\mathbb{S}\varphi\, u\,{\rm d}x\,{\rm d}t,
\end{equation}
here $P[u,v]=K\ast(u^2+\frac 1 2 v)$, and the test function $\varphi\in C^\infty(\mathbb{S})$ is fixed.

Besides, we denote $\langle X\rangle(t)$ as the quadratic variation $\langle X,X\rangle(t)$ of a process $X$, $\langle X,Y\rangle(t)$ for the co-variation between two processes $X$ and $Y$.

\begin{lemm}[SPDE of $\tilde{u}_n$]\label{SPDE tilde un}
	Let $\tilde{u}_n$, $\tilde{q}_n$, $\tilde{W}_n$, $\tilde{u}_{0,n}$ be the Skorokhod-Jakubowski representations from Poprosition \ref{SJ representations}. Then for any $\varphi\in C^\infty(\mathbb{S})$ and $t\in[0,T]$,
	\begin{equation}\label{equation tilde un}
		\begin{aligned}
			&\int_\mathbb{S}\varphi\,\tilde{u}_n\,{\rm d}x-\int_{\mathbb{S}}\varphi\,\tilde{u}_{0,n}\,{\rm d}x\\
			=&\int_0^t\int_\mathbb{S}\partial_x\varphi\left(\frac 1 2\tilde{u}_n^2+\tilde{P}_n\right)\,{\rm d}x\,{\rm d}t+\varepsilon_n\int_0^t\int_\mathbb{S}\partial_x^2\varphi\,\tilde{u}_n\,{\rm d}x\,{\rm d}t+\int_0^t\int_\mathbb{S}\varphi\,\tilde{u}_n\,{\rm d}x\,{\rm d}\tilde{W}_n,
		\end{aligned}		
	\end{equation}
	$\tilde{\mathbb{P}}$-a.s., where $\tilde{P}_n=P[\tilde{u}_n,\tilde{q}_n]$.
\end{lemm}
\begin{proof}
	For notation brevity, let $\tilde{M}_n=M_n[\tilde{u}_n,\tilde{q}_n^2,\tilde{u}_{0,n}]$, $\tilde{D}_n=D_n[\tilde{u}_n,\tilde{q}_n^2,\tilde{u}_{0,n}]$, $\tilde{R}_n=R[\tilde{u}_n]$, and $\tilde{N}_n=N[\tilde{u}_n]$. 
	It is easy to claim that $M_n$, $D_n$, $R$ and $N$ are measurable on corresponding space defined above by checking the continuity of these maps and $\int_0^t\int_\mathbb{S}\varphi\,u\,{\rm dx}\,{\rm d} W$ is square integrable martingale under $\mathbb{P}$ on the corresponding stochastic basis $\mathcal{S}=(\Omega,\mathcal{F},\mathbb{P})$ for any $u\in\mathcal{X}_u$. Since the second condition above evidently holds, it suffices to prove the continuity of $M_n$, $R$ and $N$.
	
	Fix $u_i\in\mathcal{X}_u$, $v_i\in\mathcal{X}_{q^2}$ and $z_i\in\mathcal{X}_{u_0}$, $i=1,2$. By repeated applications of H\"{o}lder's inequality and integration by parts,
	$$\begin{aligned}
		&|M_n[u_1,v_1,z_1]-M_n[u_2,v_2,z_2]|\\
		\lesssim& \Vert u_1-u_2\Vert_{C([0,T];L^2(\mathbb{S}))}\Vert\varphi\Vert_{L^2(\mathbb{S})}+\Vert z_1-z_2\Vert_{L^2(\mathbb{S})}\Vert\varphi\Vert_{L^2(\mathbb{S})}+\varepsilon_n\Vert\partial_x^2\varphi\Vert_{L^2(\mathbb{S})}\Vert u_1-u_2\Vert_{L^2([0,T]\times\mathbb{S})}\\
		&+\Vert\varphi\Vert_{L^\infty(\mathbb{S})}\Vert u_1+u_2\Vert_{L^2([0,T]\times\mathbb{S})}\Vert u_1-u_2\Vert_{L^2([0,T]\times\mathbb{S})}+\left|\int_0^T\int_{\mathbb{S}}\varphi(-x)\ast \partial_x K(v_1-v_2)\,{\rm d}x\,{\rm dt}\right|,
	\end{aligned}
	$$
	since $\varphi(-x)\ast \partial_x K(v_1-v_2)\in L^{r'}([0,T]\times\mathbb{S})$ with $1/r+1/r'=1$ is easy to obtained, by the weak topology on $L^r([0,T]\times\mathbb{S})-w$ we can deduce that $M_n$ is continuous on $\mathcal{X}_u\times\mathcal{X}_{q^2}\times\mathcal{X}_{u_0}$. The continuity of $R$ and $N$ can be proved in the same way.
	
	For any c\`{a}dl\`{a}g process $X$ on ${0,T}$ and $0\le s<t\le T$, denote $\Delta_{s,t} X=X(t)-X(s)$. Let $\gamma:\mathcal{X}|_{[0,s]}\rightarrow[0,1]$ be an arbitrary continuous function. By the martingale property of $M_n:=M_n[u_{\varepsilon_n},q_{\varepsilon_n}^2,z_{0,n}]$, $M_n^2-R[u_{\varepsilon_n}]$ and $M_n\,W-N[u_{\varepsilon_n}]$ and the equality of laws in Proposition \ref{SJ representations}, we obtain that
	$$
	\tilde{E}\left[\gamma\left(\tilde{X}_n|_{[0,s]}\right)\Delta_{s,t}\tilde{M}_n\right]=0,\quad \tilde{E}\left[\gamma\left(\tilde{X}_n|_{[0,s]}\right)\Delta_{s,t}(\tilde{M}_n^2-\tilde{R}_n)\right]=0,
	$$
	$$ \tilde{E}\left[\gamma\left(\tilde{X}_n|_{[0,s]}\right)\Delta_{s,t}(\tilde{M}_n\,\tilde{W}_n-\tilde{N}_n)\right]=0,
	$$
	which implies that $\tilde{M}_n$, $\tilde{M}_n^2-\tilde{R}_n$ and $\tilde{M}_n\,\tilde{W}_n-\tilde{N}_n$ are $\{\tilde{\mathcal{F}}_t^n\}$-martingale.
	According to the modified martingale representation theorem \cite{Hofmanov2013}, $\langle \tilde{M}_n\rangle=\tilde{R}_n$ and $\langle \tilde{M}_n,\tilde{W}_n\rangle=\tilde{N}_n$. Therefore, $\tilde{D}_n$ has vanishing quadratic variation, the Doob–Meyer decomposition theorem yields $\tilde{D}_n=0$, which is the sought-after equation \eqref{equation tilde un}.
\end{proof}

Now we have obtained the SPDE satisfied by $\tilde{u}_n$, we hope to obtain the SPDE satisfied by $\tilde{u}$ through the convergence \eqref{SJ convergence} of the Skorokhod--Jakubowski representations. However, we have not yet established the connection between the limits of the nonlinear terms, more precisely, we lack the identification $\overline{q^2}=\tilde{q}^2$, which is equivalent to the strong convergence of $\tilde{q}_n$ towards $\tilde{q}$ in $L^2_{\tilde{\omega},t,x}$. This identification, which will be proved in detail in the next section, would allow us to obtain the SPDE satisfied by $\tilde{u}$.

\begin{prop}[SPDE of $\tilde{u}$]\label{SPDE tilde u}
	Suppose the assumption of Theorem \ref{main theorem} hold. Let $\tilde{u}$, $\tilde{q}$ ,$\overline{q^2}$, $\tilde{W}$, $\tilde{u}_0$ be the Skorokhod-Jakubowski representations from Proposition \ref{SJ representations}, and let $\tilde{S}$ be the stochastic basis defined in \eqref{filtration}. Suppose the following identification holds:
	\begin{equation}\label{q^2 identification}
		\overline{q^2}=\tilde{q}^2,\quad \tilde{\mathbb{P}}\otimes{\rm d}t\otimes{\rm d}x-\text{a.e. in } \tilde{\Omega}\times[0,T]\times\mathbb{S}.
	\end{equation}
	Then $(\tilde{S},\tilde{u},\tilde{W})$ is a weak martingale solution of
	$$
	\begin{aligned}
		&{\rm d}\tilde{u}+(\tilde{u}\partial_x\tilde{u}+\partial_x P[\tilde{u}])\,{\rm d}t=\tilde{u}\,{\rm d}\tilde{W},\\
		&P[\tilde{u}]=K\ast\left(\tilde{u}^2+\frac 1 2\tilde{q}^2\right),\quad\tilde{u}(0)=\tilde{u}_0\sim u_0,
	\end{aligned}
	$$
	in the sense of Definition \ref{CH H^1}.
\end{prop}
\begin{proof}
	Consider the functionals $M[u,v,z]$, $D[u,v,z]$, $R[u]$ and $N[u]$ defined as above. To simplify the notation, set $\tilde{M}=M[\tilde{u},\tilde{q}^2,\tilde{u}_0]$, $\tilde{D}=D[\tilde{u},\tilde{q}^2,\tilde{u}_0]$ $\tilde{R}=R[\tilde{u}]$ and $\tilde{N}=N[\tilde{u}]$.
	
	Since $\gamma$ is bounded and continuous, and $\tilde{X}_n\rightarrow\tilde{X}$ a.s., it follows that
	\begin{equation}\label{gamma convergence}
		\gamma\left(\tilde{X}_n|_{[0,s]}\right) \rightarrow \gamma\left(\tilde{X}|_{[0,s]}\right)\ \text{in}\ L^p(\tilde{\Omega}),\ \text{for any finite}\ p.
	\end{equation} 
	
	According to Proposition \ref{SJ representations}, we have $\tilde{u}_{0,n}\rightarrow\tilde{u}_0$ in $H^1(\mathbb{S})$ and $\tilde{u}_n\rightarrow\tilde{u}$ in $C([0,T];L^2(\mathbb{S}))$, $\tilde{\mathbb{P}}$-a.s. This implies, for fixed $\varphi\in C^\infty(\mathbb{S})$,
	$$
	\begin{aligned}
		&\left|\int_\mathbb{S}\varphi\,(\tilde{u}_{0,n}-\tilde{u}_0)\,{\rm d}x\right|\rightarrow0,\quad \tilde{\mathbb{P}}-{\rm a.s.},\\
		&\sup_{t\in[0,T]}\left|\int_\mathbb{S}\varphi\,(\tilde{u}_n-\tilde{u})(t)\,{\rm d}x\right|\rightarrow0,\quad \tilde{\mathbb{P}}-{\rm a.s.},\\
		&\varepsilon_n\left|\int_0^t\int_\mathbb{S}\partial_x^2\varphi\, \tilde{u}_n\,{\rm d}x\,{\rm d}s\right|\rightarrow0,\quad \tilde{\mathbb{P}}-{\rm a.s.},\\
		&\left|\int_0^t\int_\mathbb{S}\partial_x\varphi\,(\tilde{u}_n^2-\tilde{u}^2)\,{\rm d}x\,{\rm d}s\right|\rightarrow0,\quad \tilde{\mathbb{P}}-{\rm a.s.},\\
	\end{aligned}
	$$
	
	Using $\tilde{u}_n^2\rightarrow\tilde{u}^2$ in $C([0,T];L^1(\mathbb{S}))$, $\tilde{\mathbb{P}}$-a.s., $\tilde{q}_n^2\rightarrow\overline{q^2}$ in $L^r([0,T];L^r_w(\mathbb{S}))$, $\tilde{\mathbb{P}}$-a.s., and the weak limit identification \eqref{q^2 identification}, we have
	$$
	\begin{aligned}
		&\left|\int_0^t\int_\mathbb{S}\partial_x\varphi\,\left(
		P[\tilde{u}_n,\tilde{q}_n^2]-P[\tilde{u},\overline{q^2}]\right)\,{\rm d}x\,{\rm d}s\right|\\
		\le&\|\partial_x\varphi\|_{L^\infty(\mathbb{S})}\|K\|_{L^1(\mathbb{S})}\left\|\tilde{u}_n^2-\tilde{u}^2\right\|_{L^1([0,T]\times\mathbb{S})}\\
		&\qquad\qquad\qquad\qquad\qquad\qquad+\left|\int_0^t\int_{\mathbb{S}}\left(\int_\mathbb{S}\partial_x\varphi(x)\, K(x-y)\,{\rm d}x\right)(y)\left(\tilde{q}_n^2-\tilde{q}^2\right)(s,y)\,{\rm d}y\,{\rm d}s\right|,
	\end{aligned}
	$$
	since $\int_\mathbb{S}\partial_x\varphi(x)\, K(x-y)\,{\rm d}x\in L^{r'}([0,T]\times\mathbb{S})$, we conclude that
	$$\left|\int_0^t\int_\mathbb{S}\partial_x\varphi\,\left(
	P[\tilde{u}_n,\tilde{q}_n^2]-P[\tilde{u},\overline{q^2}]\right)\,{\rm d}x\,{\rm d}s\right|\rightarrow 0\quad\tilde{\mathbb{P}}-{\rm a.s.}
	$$
	And
	$$
	\begin{aligned}
		&\left|
		\int_0^t\left(\int_\mathbb{S}\varphi\,\tilde{u}_n\,{\rm d}x\right)^2-\left(\int_\mathbb{S}\varphi\,\tilde{u}\,{\rm d}x\right)^2
		\,{\rm d}s\right|\\
		\le&
		\left|
		\int_0^t
		\left(\int_\mathbb{S}\left|\varphi\,\tilde{u}_n-\varphi\,\tilde{u}\right|\,{\rm d}x\right)
		\left(\int_\mathbb{S}\left|\varphi\,\tilde{u}_n+\varphi\,\tilde{u}\right|\,{\rm d}x\right)
		\,{\rm d}s\right|\\
		\le&
		\|\varphi\|_{L^2(\mathbb{S})}^2\|\tilde{u}_n+\tilde{u}\|_{L^2([0,T]\times\mathbb{S})}\|\tilde{u}_n-\tilde{u}\|_{L^2([0,T]\times\mathbb{S})}\rightarrow 0,\quad\tilde{\mathbb{P}}-{\rm a.s.}
	\end{aligned}
	$$
	
	By Proposition \ref{priori estimate}, $\tilde{E}|\tilde{M}_n(t)|^{p_0}\lesssim_\varphi 1$, $\tilde{E}|\tilde{R}_n(t)|^{p_0/2}\lesssim_\varphi 1$, $\tilde{E}|\tilde{N}_n(t)|^{p_0/2}\lesssim_\varphi 1$, and by Cauchy-Schwarz's inequality
	$$\tilde{E}\left|\tilde{M}_n(t)\tilde{W}_n(t)\right|^{p_0/2}\le(\tilde{E}\left|\tilde{M}_n(t)\right|^{p_0})^{1/2}(\tilde{E}\left|\tilde{W}_n(t)\right|^{p_0})^{1/2}\lesssim_\varphi 1,$$
	thus, by Vitali's convergence theorem, we obtain that for all $t\in[0,T]$,
	\begin{equation}\label{tilde n and tilde}
		\begin{array}{c}
			\tilde{M}_n(t)\rightarrow\tilde{M}(t)\ {\rm in}\ L^p(\tilde{\Omega}),\ \forall\ p\in[1,p_0),\\[2ex]
			\tilde{M}_n^2(t)\rightarrow\tilde{M}^2(t)\ {\rm in}\ L^p(\tilde{\Omega}),\ \forall\ p\in[1,p_0/2),\\[2ex]
			\tilde{R}_n(t)\rightarrow\tilde{M}^2(t)\ {\rm in}\ L^p(\tilde{\Omega}),\ \forall\ p\in[1,p_0/2)\\[2ex]
			\tilde{M}_n(t)\tilde{W}_n(t)\rightarrow\tilde{M}(t)\tilde{W}(t)\ {\rm in}\ L^p(\tilde{\Omega}),\ \forall\ p\in[1,p_0/2),\\[2ex]
			\tilde{N}_n(t)\rightarrow\tilde{N}(t)\ {\rm in}\ L^p(\tilde{\Omega}),\ \forall\ p\in[1,p_0).
		\end{array}
	\end{equation}
	
	According to \eqref{tilde n and tilde} and the martingale properties of $\tilde{M}_n$, $\tilde{M}_n^2-\tilde{R}_n$, $\tilde{M}_n\,\tilde{W}_n-\tilde{N}_n$, follow the same line in the proof of Lemma \ref{SPDE tilde un}, we verify that $\tilde{M}$, $\tilde{M}^2-\tilde{R}$ and $\tilde{M}\,\tilde{W}-\tilde{N}$ are all $\{\tilde{F}_t\}$-martingales.
	
	Similarly, we finally obtain that $\tilde{D}=0$, which is the sought-after equation \eqref{SPDE tilde u}.
\end{proof}

\section{Identification of a Weak Limit}\label{5}
\subsection{Energy inequalities and right-continuity}
\begin{lemm}\label{tilde u energy}
	Let $\tilde{u}$, $\tilde{q}$, $\overline{q^2}$, $\tilde{W}$ and $\tilde{u}_0$ be the a.s. limits from Proposition \ref{SJ representations}, and $\tilde{S}$ be the stochastic basis defined in \eqref{filtration}. Then the energy inequality holds:
	\begin{equation}\label{tilde u energy inequality}
		\begin{aligned}
			&\frac{\rm d}{{\rm d}t}\int_\mathbb{S}\tilde{u}^2+\overline{q^2}\,{\rm d}x\le \int_{\mathbb{S}}\tilde{u}^2+\overline{q^2}\,{\rm d}x+2\int_{\mathbb{S}}\tilde{u}^2+\overline{q^2}\,{\rm d}x\,{\rm d}\dot{\tilde{W}},\quad{\rm in}\ \mathcal{D}'([0,T)),\ \tilde{\mathbb{P}}-{\rm a.s.,}\\
			&\int_{\mathbb{S}}(\tilde{u}^2+\overline{q^2})(0)\,{\rm d}x=\int_{\mathbb{S}}\tilde{u}_0^2+|\partial_x\tilde{u}_0|^2\,{\rm d}x.
		\end{aligned}
	\end{equation}
	i.e., for all non-negative $\varphi\in C^\infty([0,T))$,
	$$
	\begin{aligned}
		&-\int_0^T\partial_t\varphi\int_\mathbb{S}\tilde{u}^2+\overline{q^2}\,{\rm d}x\,{\rm d}t-\varphi(0)\int_{\mathbb{S}}(\tilde{u}^2+\overline{q^2})(0)\,{\rm d}x\\
		\le &\int_0^T\varphi\int_{\mathbb{S}}\tilde{u}^2+\overline{q^2}\,{\rm d}x\,{\rm d}t+2\int_0^T\varphi\int_{\mathbb{S}}\tilde{u}^2+\overline{q^2}\,{\rm d}x\,{\rm d}\tilde{W},\quad\tilde{\mathbb{P}}-{\rm a.s.}
	\end{aligned}
	$$
\end{lemm}
\begin{proof}
	According to Lemma \ref{priori estimate}, $\tilde{u}_n$ lies on $L^2([0,T];H^m(\mathbb{S}))\cap C([0,T];H^1(\mathbb{S}))$, $\tilde{\mathbb{P}}$-a.s., $\tilde{u}_n$ satisfies \eqref{SPDE tilde un} and $\tilde{q}_n=\partial_x\tilde{u}_n$, It\^{o}'s formula yields the total energy equation:
	\begin{equation}\label{tilde un energy equation}
		\begin{aligned}
			{\rm d}&\left(\frac{\tilde{u}_n^2+\tilde{q}_n^2}{2}\right)+\partial_x\left(\frac 1 2\tilde{u}_n\,\tilde{q}_n^2+\tilde{u}_n\,P[\tilde{u}_n]\right)\,{\rm d}t\\
			&-\varepsilon\partial_x^2\left(\frac{\tilde{u}_n^2+\tilde{q}_n^2}{2}\right)\,{\rm d}t+\varepsilon\left(|\partial_x\tilde{u}_n|^2+|\partial_x\tilde{q}_n|^2\right)\,{\rm d}t=\left(\frac{\tilde{u}_n^2+\tilde{q}_n^2}{2}\right)\,{\rm d}t
			+\left(\tilde{u}_n^2+\tilde{q}_n^2\right)\,{\rm d}W
		\end{aligned}
	\end{equation}
	Integrating this equation in $x$ and expressing the temporal differential as a time-derivative in $\mathcal{D}'([0,T))$, we arrive at
	\begin{equation}\label{tilde un energy inequality}
		\frac{\rm d}{{\rm d}t}\int_{\mathbb{S}}\tilde{u}_n^2+\tilde{q}_n^2\,{\rm d}x\le \int_{\mathbb{S}}\tilde{u}_n^2+\tilde{q}_n^2\,{\rm d}x+2\int_{\mathbb{S}}\tilde{u}_n^2+\tilde{q}_n^2\,{\rm d}x\,\dot{\tilde{W}}_n,\ {\rm in}\ \mathcal{D}'([0,T)),\ \tilde{\mathbb{P}}-{\rm a.s.,}
	\end{equation}
	where $\dot{\tilde{W}}_n=\frac{\rm d}{{\rm d}t}\tilde{W}_n$, $\int_{\mathbb{S}}(\tilde{u}_n^2+\tilde{q}_n^2)(0)\,{\rm d}x=\int_{\mathbb{S}}\tilde{u}_{n,0}^2+|\partial_x\tilde{u}_{n,0}|^2\,{\rm d}x$.
	
	By \eqref{priori estimate q}, we deduce that 
	$$
	\tilde{\mathbb{E}}\int_0^T\left|\int_{\mathbb{S}}\overline{q^2}\,{\rm d}x\right|^2\,{\rm d}t\le\tilde{\mathbb{E}}\int_0^T\Vert\overline{q^2}(t)\Vert_{H^{-1}(\mathbb{S})}^2\,{\rm d}t<\infty,
	$$
	thus, the process $\int_0^t\int_{\mathbb{S}}\overline{q^2}\,{\rm d}x\,{\rm d}\tilde{W}$ is a square-integrable martingale. Equipped with the a.s. convergences in \eqref{SJ convergence}, we obtain \eqref{tilde u energy inequality}.
\end{proof}

\begin{lemm}[One-sided temporal continuity at $t=0$]\label{continuity t=0}
	Let $\tilde{u}$, $\tilde{q}$, $\tilde{u}_0$ and $\overline{q^2}$ be the Skorokhod-Jakubowski representations from Proposition \ref{SJ representations}. Then, for $\tilde{u}$ and the nonlinearities $S(v)=S_l(v_\pm)$ defined as \eqref{S_l},
	\begin{equation}\label{tilde u continuity t=0}
		\lim_{t\downarrow0}\|\tilde{u}(t)-\tilde{u}_0\|_{H^1(\mathbb{S})}=0,\quad\tilde{\mathbb{P}}-{a.s.,}
	\end{equation}
	\begin{equation}\label{tilde S continuity t=0}
		\lim_{t\downarrow0}\|S\left(\tilde{q}(t)\right)-S(\partial_x\tilde{u}_0)\|_{L^1(\mathbb{S})}=0,\quad\tilde{\mathbb{P}}-{a.s.}
	\end{equation}
	Moreover,
	$$
	\lim_{t\downarrow0}\tilde{\mathbb{E}}\|\tilde{u}(t)-\tilde{u}_0\|_{H^1(\mathbb{S})}^2=0,\quad\lim_{t\downarrow0}\tilde{\mathbb{E}}\|S\left(\tilde{q}(t)\right)-S(\partial_x\tilde{u}_0)\|_{L^1(\mathbb{S})}=0.
	$$
\end{lemm}
\begin{proof}
	Since $\tilde{u}_n\rightarrow\tilde{u}$ in $C([0,T];H^1(\mathbb{S})-w)$, $\mathbb{\tilde{P}}$-a.s., employing the weak lower semicontinuity of $v\mapsto\|v\|_{H^1(\mathbb{S})}^2$, 
	\begin{equation}\label{lower semicontinuity inequality}
		\Vert\tilde{u}(t)\Vert_{H^1(\mathbb{S})}^2\le \liminf_{n\rightarrow \infty}\Vert\tilde{u}_n(t)\Vert_{H^1(\mathbb{S})}^2\le\limsup_{n\rightarrow\infty}\Vert\tilde{u}_n(t)\Vert_{H^1(\mathbb{S})}^2,\quad t>0.
	\end{equation}
	Define
	$$I_n(t)=\int_0^T\int_{\mathbb{S}}\tilde{u}_n^2+\tilde{q}_n^2\,{\rm d}x\,{\rm d}t+2\int_0^T\int_{\mathbb{S}}\tilde{u}_n^2+\tilde{q}_n^2\,{\rm d}x\,{\rm d}\tilde{W}_n,
	$$
	$$I(t)=\int_0^T\int_{\mathbb{S}}\tilde{u}^2+\overline{q^2}\,{\rm d}x\,{\rm d}t+2\int_0^T\int_{\mathbb{S}}\tilde{u}^2+\overline{q^2}\,{\rm d}x\,{\rm d}\tilde{W},
	$$
	According to the proof of Lemma \ref{tilde u energy}, $I_n(t)\rightarrow I(t)$ a.s., with $t>0$ fixed. By standard deterministic argument, we turn the distributional inequality \eqref{tilde un energy inequality} into the pointwise inequality
	$$\Vert\tilde{u}_n\Vert_{H^1(\mathbb{S})}^2\le\Vert\tilde{u}_{0,n}\Vert_{H^1(\mathbb{S})}^2+I_n(t),\quad{\rm a.s.,}\ \forall\,t>0.$$
	In the view of \eqref{lower semicontinuity inequality}, as $n\rightarrow\infty$, we conclude that
	\begin{equation}\label{tilde u pointwise inequality}
		\Vert\tilde{u}\Vert_{H^1(\mathbb{S})}^2\le\Vert\tilde{u}_0\Vert_{H^1(\mathbb{S})}^2+I(t),\quad{\rm a.s.,}\ \forall\,t>0.
	\end{equation}
	
	Since $\tilde{u}\in C([0,T];H^1(\mathbb{S})-w)$ a.s., i.e. $\tilde{u}(t)\rightharpoonup\tilde{u}_0$ in $H^1(\mathbb{S})$ as $t\downarrow 0$, whence, by the weak lower semicontinuity of $\Vert\cdot\Vert_{H^1(\mathbb{S})}^2$ and \eqref{tilde u pointwise inequality},
	$$\Vert\tilde{u}_0\Vert_{H^1(\mathbb{S})}^2\le \liminf_{t\downarrow0}\Vert\tilde{u}(t)\Vert_{H^1(\mathbb{S})}^2\le\limsup_{t\downarrow0}\Vert\tilde{u}(t)\Vert_{H^1(\mathbb{S})}^2\le\limsup_{t\downarrow0}\left(\Vert\tilde{u}_0\Vert_{H^1(\mathbb{S})}^2+I(t)\right)=\Vert\tilde{u}_0\Vert_{H^1(\mathbb{S})}^2.$$
	Thus, we have
	$\lim_{t\downarrow0}\Vert\tilde{u}(t)\Vert_{H^1(\mathbb{S})}^2=\Vert\tilde{u}_0\Vert_{H^1(\mathbb{S})}^2$, from weak convergence and norm convergence, we can obtain that \eqref{tilde u continuity t=0} holds.
	
	By the definition of $S_l(v)$ \eqref{S_l}, we can compute that for any $a,\ b\in\mathbb{R}$,
	$$|S_l(b)-S_l(a)|\le(|a|+|b|)|a-b|.$$
	Thus, for fixed $\omega\in\tilde{\Omega}$, according to \eqref{tilde u continuity t=0}
	$$
	\begin{aligned}
		\Vert S(\tilde{q}(t))-S(\partial_x\tilde{u}_0)\Vert_{L^1(\mathbb{S})}
		&\le\int_{\mathbb{S}}\left(|\tilde{q}(t,x)+|\partial_x\tilde{u}_0(x)|\right)|\tilde{q}(t,x)-\partial_x\tilde{u}_0(x)|\,{\rm d}x\\
		&\lesssim \Vert\tilde{q}\Vert_{L^\infty([0,T];L^2(\mathbb{S}))}\Vert\tilde{q}(t)-\partial_x\tilde{u}_0\Vert_{L^\infty([0,T];L^2(\mathbb{S}))}\\
		&\lesssim_\omega\Vert\tilde{q}(t)-\partial_x\tilde{u}_0\Vert_{L^\infty([0,T];L^2(\mathbb{S}))}\rightarrow0.
	\end{aligned}$$
	This concludes the proof of \eqref{tilde S continuity t=0}.
	
	The third part of this lemma can be deduce by \eqref{tilde u continuity t=0}, \eqref{tilde S continuity t=0} and Vitali's convergence theorem.
\end{proof}

If we have the identification \eqref{q^2 identification}, the energy inequality \eqref{tilde u energy inequality} becomes
\begin{equation}\label{tilde u energy inequality with identification}
	\begin{aligned}
		&\frac{\rm d}{{\rm d}t}\int_\mathbb{S}\tilde{u}^2+\tilde{q}^2\,{\rm d}x\le \int_{\mathbb{S}}\tilde{u}^2+\tilde{q}^2\,{\rm d}x+2\int_{\mathbb{S}}\tilde{u}^2+\tilde{q}^2\,{\rm d}x\,{\rm d}\dot{\tilde{W}},\quad{\rm in}\ \mathcal{D}'([0,T)),\ \tilde{\mathbb{P}}-{\rm a.s.,}\\
		&\int_{\mathbb{S}}(\tilde{u}^2+\tilde{q}^2)(0)\,{\rm d}x=\int_{\mathbb{S}}\tilde{u}_0^2+|\partial_x\tilde{u}_0|^2\,{\rm d}x.
	\end{aligned}
\end{equation}
with $\int_\mathbb{S}\tilde{u}^2+\tilde{q}^2\,{\rm d}x=\Vert\tilde{u}\Vert_{H^1(\mathbb{S})}^2$

\begin{lemm}[Energy inequality and one-sided temporal continuity]\label{energy and one-sided temporal}
	Suppose \eqref{tilde u energy inequality with identification} holds, then the total energy inequality \eqref{energy inequality} holds a.s., for a.e. $s\in[0,T)$ and every $t$ with $s<t\le T$. Consequently, a.s., for a.e. $t_0\in[0,T)$,
	\begin{equation}\label{tilde u continuity}
		\lim_{t\downarrow t_0}\Vert\tilde{u}(t)-\tilde{u}(t_0)\Vert_{H^1(\mathbb{S})}=0.
	\end{equation}
\end{lemm}
\begin{proof}
	According standard deterministic argument \cite{Evans1998}, consider the continuous piecewise linear test functions $0\le\beta_\delta\in W^{1,\infty}([0,T])$ with $\delta>0$ taking values in a sequence converging to 0. For given $s$ and $t$ with $0\le s<t\le T$,consider $\delta>0$ such that $s+\delta<t-\delta$. Define $\beta_\delta$ equals $1$ on $[s+\delta,t-\delta]$, $0$ on $[0,s]$ and $[t,T]$, and is linear on $[s,s+\delta]$ and $[t-\delta,t]$. It is obvious that $\beta_\delta\rightarrow\mathds{1}_{[s,t]}(t')$ for a.e. $t'\in[0,T]$ as $\delta\rightarrow0$. And define
	$$\Psi_\delta(t')=\int_\mathbb{S}\left(\tilde{u}^2+\tilde{q}^2\right)(t')\beta_\delta(t')\,{\rm d}x.$$
	
	Using $\beta_\delta$ as the test function in \eqref{tilde u energy inequality with identification}, we obtain that
	$$
	\frac 1 \delta\int_{t-\delta}^t\Vert\tilde{u}(t')\Vert_{H^1(\mathbb{S})}^2\,{\rm d}t'-\frac 1 \delta\int_s^{s+\delta}\Vert\tilde{u}(t')\Vert_{H^1(\mathbb{S})}^2\,{\rm d}t'\le \int_0^T\Psi_\delta(t')\,{\rm d}t'+2\int_0^T\Psi_\delta(t')\,{\rm d}\tilde{W}(t').
	$$
	According to Lebesgue's differentiation theorem, as $\delta$ tends to $0$, for fixed $\tilde{\omega}$ from a set $F$ of full $\tilde{\mathbb{P}}$-measure, there exists a subset $N(\tilde{\omega})\subset[0,T]$ of zero Lebesgue measure such that for every $t\in[0,T]\backslash N(\tilde{\omega})$,
	\begin{equation}\label{energy expect lebesgue 0}
		\Vert\tilde{u}(\tilde{\omega},t)\Vert_{H^1(\mathbb{S})}^2-\Vert\tilde{u}(\tilde{\omega},s)\Vert_{H^1(\mathbb{S})}^2\le \int_s^t \Psi(t')\,{\rm d}t'+2\int_s^t\Psi(t')\,{\rm d}\tilde{W}(t'),
	\end{equation}
	with $\Psi(t)=\int_{\mathbb{S}}\tilde{u}^2(t)+\tilde{q}^2(t)\,{\rm d}x$. More precisely, the convergence of the stochastic integrals $\int_0^T\Psi_\delta\,{\rm d}\tilde{W}$ can be obtain by Lemma \ref{priori estimate}, the BDG inequality and the Lebesgue dominated convergence theorem.
	
	Fix an arbitrary $t>s$ with $s\in[0,T)\backslash N(\tilde{\omega})$, $\tilde{\omega}\in F$. Then there exists a sequence $\{t_l\}_{l=1}^\infty\subset [0,T)\backslash N(\tilde{\omega})$ such that $t_l>s$ and $t_l$ tends to $t$ as $l\rightarrow\infty$. Since $\tilde{u}\in C([0,T];H^1(\mathbb{S})-w)$, then $\tilde{u}$ is a.s. weakly lower semicontinuous in $H^1(\mathbb{S})$,
	$$\Vert\tilde{u}(\tilde{\omega},t)\Vert_{H^1(\mathbb{S})}^2\le\liminf_{l\rightarrow\infty}\Vert\tilde{u}(\tilde{\omega},t_l)\Vert_{H^1(\mathbb{S})}^2
	\le\Vert\tilde{u}(\tilde{\omega},s)\Vert_{H^1(\mathbb{S})}^2+\int_s^t\Psi(t')\,{\rm d}t'+2\int_s^t\Psi(t')\,{\rm d}\tilde{W}(t').$$
	Similarly, for arbitrary $0\le s<t\le T$ we can obtain in the same way, thus \eqref{energy inequality} is proved.
	
	For fixed $\tilde{\omega}\in F$, the right-continuity of $\tilde{u}$ in $H^1(\mathbb{S})$ at Lebesgue point $s=t_0\in [0,T]\backslash N(\tilde{\omega})$ can be inferred from \eqref{energy expect lebesgue 0} and the a.s. weak lower semicontinuity of $\tilde{u}$. As a result, we can employ a similar reasoning as in the proof of Lemma \ref{continuity t=0} to conclude that the right-continuity claim \eqref{tilde u continuity} holds. In particular, for $s=0$, we utilizing the strong initial trace result \eqref{tilde u continuity t=0}.
\end{proof}

\subsection[Equation of the weak limit S(q)]{Equation of the weak limit $\overline{S(q)}$}
We need to know that products like $S'(\tilde{q}_n)\tilde{P}_n$ converges weakly. To achieve this, we need to refine the convergence of $\tilde{q}_n^2$ in \eqref{SJ convergence}.

By Lemma \ref{priori estimate}, $\tilde{q}_n^2$ is uniformly bounded in $L^r(\tilde{\Omega}\times[0,T]\times\mathbb{S})$, thus by weak compactness argument and \eqref{SJ convergence}, we may assume 
$$\overline{q^2}\in L^r(\tilde{\Omega}\times[0,T]\times\mathbb{S})\quad{\rm and}\quad \tilde{q}_n^2\rightharpoonup\overline{q^2}\quad{\rm in}\ L^r(\tilde{\Omega}\times[0,T]\times\mathbb{S}),$$
also, since $\tilde{q}_n$ is uniformly bounded in $L^{2p}_\omega(L^\infty(\times[0,T];L^2(\mathbb{S})))$, for any $p\in[1,p_0/2]$. We can easily deduce that $\tilde{q}_n^2$ is uniformly bounded in $L^p(\tilde{\Omega}\times[0,T];H^{-1}(\mathbb{S}))$, according to weak compactness argument and \eqref{SJ convergence} similarly,
$$\overline{q^2}\in L^p(\tilde{\Omega}\times[0,T];H^{-1}(\mathbb{S}))\quad{\rm and}\quad \tilde{q}_n^2\rightharpoonup\overline{q^2}\quad{\rm in}\ L^p(\tilde{\Omega}\times[0,T];H^{-1}(\mathbb{S})).$$

Therefore, we have for $p_0>4$,
$$\tilde{\mathbb{E}}\int_0^T\left\Vert\tilde{q}_n^2(t)-\overline{q^2}(t)\right\Vert_{H^{-1}(\mathbb{S})}^p\,{\rm d}t\lesssim 1,\quad \text{for any }p\in[1,p_0/2].$$
Since $\tilde{q}_n\rightarrow\overline{q^2}$ in $L^r(L^r_w)$, and $L^r_w(\mathbb{S})\hookrightarrow H^{-1}(\mathbb{S})$ for fixed $r\in[1,3/2)$, this yields $\tilde{q}_n\rightarrow\overline{q^2}$ in $L^r(H^{-1}(\mathbb{S}))$. Thus, by Vitali's convergence theorem,
$$\tilde{q}_n^2\rightarrow\overline{q^2}\quad{\rm in}\ L^r_{\tilde{\omega},t}(H^{-1}_x).$$
Finally, by the Lebesgue interpolation, we cam improve above convergence to
\begin{equation}\label{tilde qn convergence}
	\tilde{q}_n^2\rightarrow\overline{q^2}\quad{\rm in}\ L^p_{\tilde{\omega},t}(H^{-1}_x),\quad p\in[1,p_0/2).
\end{equation}

\begin{lemm}\label{tilde Pn convergence}
	Let $\tilde{u}_n$, $\tilde{u}$, $\tilde{q}_n$, $\overline{q^2}$ be the Skorokhod-Jakubowski representations form Proposition \ref{SJ representations}. Setting
	$$\tilde{P}_n= K\ast \left(\tilde{u}_n^2+\frac 1 2\tilde{q}_n^2\right),\quad \tilde{P}=K\ast\left(\tilde{u}^2+\frac 1 2\overline{q^2}\right),
	$$
	the following strong convergence holds:
	\begin{equation}\label{tilde P convergence r}
		\tilde{P}_n\rightarrow\tilde{P}\quad{\rm in}\ L^r([0,T]\times\mathbb{
		S}),\ \tilde{\mathbb{P}}-{\rm a.s.,}
	\end{equation}
	where $r\in[1,3/2)$ fixed in \eqref{tight quasi Polish}. In particular, for any $p\in[1,p_0/2)$,
	\begin{equation}\label{tilde P convergence p}
		\tilde{P}_n\rightarrow\tilde{P}\quad{\rm in}\ L^p(\tilde{\Omega}\times[0,T]\times\mathbb{
			S}),\ \tilde{\mathbb{P}}{\rm a.s.,}
	\end{equation}
	where $p_0>4$ is fixed.
\end{lemm}
\begin{proof}
	This is a direct consequence of the standard H\"{o}lder and Young inequalities combined with \eqref{tilde qn convergence}, the proof is omitted.
\end{proof}

The remaining part of this section is devoted to the study of the defect measure $\mathbb{D}$ defined in \eqref{defect measure}, which will be done by analyzing the related defects $\overline{S(q)}-S(\tilde{q})$, for an appropriate class of nonlinearities $S$. We want the a.s. weak limit $\overline{S(q)}$ to satisfy the following pathwise inequality in the sense of distributions:
$$
\partial_t \overline{S(q)}+\partial_x(\tilde{u}\,\overline{S(q)})+\left[
\overline{S'(q)}\left(P[\tilde{u},\overline{q^2}]-\tilde{u}^2\right)
-\left(\overline{S(q)\,q}-\frac 1 2\overline{S'(q)\,q^2}\right)
-\frac 1 2\overline{S''(q)\,q^2}
\right]
-\overline{S'(q)\,q}\,\dot{\tilde{W}}\le 0,
$$
along with the initial data $\overline{S(q)}(0)=S(\partial_x\tilde{u}_0)$. However, we cannot establish above inequality along the lines of Proposition \ref{SPDE tilde u}, the main obstacle is the lack of strong temporal compactness, which prevents us from passing to the limit in some terms. Instead we will furnish a “direct” weak convergence proof, relying on Lemma \ref{convergence of stochastic integrals} to establish the convergence of stochastic integrals of processes like $\int_{\mathbb{S}}S'(\tilde{q}_n)\tilde{q}_n\,{\rm d}x$.

\begin{lemm}\label{S(q) eneregy}
	Denote by $S=S(v)$ any of the function $S_l(v_\pm)$ defined as \eqref{S_l}, or $\frac 1 2v^2$, $\frac 1 2v_\pm^2$. Let $\overline{S(q)},\ \overline{S'(q)},\ \overline{S(q)q},\ \overline{S'(q)q},\ \overline{S'(q)q^2}$ and $\overline{S''(q)q^2}$ be the Skorokhod-Jakubowski representations in Proposition \ref{SJ representations}, and let $\tilde{P}$ defined by Lemma \ref{tilde Pn convergence}. Then 
	\begin{equation}\label{Sq equation}
		\begin{aligned}
			\partial_t \overline{S(q)}&+\partial_x(\tilde{u}\,\overline{S(q)})\\
			&+\left[
			\overline{S'(q)}\left(P[\tilde{u},\overline{q^2}]-\tilde{u}^2\right)
			-\left(\overline{S(q)\,q}-\frac 1 2\overline{S'(q)\,q^2}\right)
			-\frac 1 2\overline{S''(q)\,q^2}
			\right]
			-\overline{S'(q)\,q}\,\dot{\tilde{W}}\le 0,
		\end{aligned}
	\end{equation}
	that is for any $0\le \varphi\in C^\infty_c([0,T)\times\mathbb{S})$,
	\begin{equation}\label{Sq inequality}
		\begin{aligned}
			\int_0^T&\int_{\mathbb{S}}\overline{S(q)}\,\partial_t\varphi\,{\rm d}x\,{\rm d}t+\int_{\mathbb{S}}S(\partial_x\tilde{u}_0)\,\varphi(0,x)\,{\rm d}x
			\\
			&+\int_0^T\int_{\mathbb{S}}\left[\overline{S'(q)}\left(\tilde{u}^2-\tilde{P}\right)+\overline{H(q)}+\frac 1 2\overline{S''(q)\,q^2}\right]\varphi\,{\rm d}x\,{\rm d}t\\
			&+\int_0^T\int_{\mathbb{S}}\tilde{u}\,\overline{S(q)}\,\partial_x\varphi\,{\rm d}x\,{\rm d}t+\int_0^T\int_{\mathbb{S}}\overline{S'(q)\, q}\,\varphi\,{\rm d}x\,{\rm d}\tilde{W}\ge 0,\quad\tilde{\mathbb{P}}-{\rm a.s.,}
		\end{aligned}
	\end{equation}
	with $H(v)=S(v)v-\frac 1 2S'(v)v^2$. By the linearity of weak limits, $\overline{H(q)}=\overline{S(q)q}-\frac 1 2\overline{S'(q)q^2}$.
\end{lemm}
\begin{proof}
	Fix a nonnegative test function $\varphi\in C^\infty_c([0,T)\times\mathbb{S})$. Multiply \eqref{viscous S(q)} by $\varphi$, integrate over $(t,x)$, since $S''(v)\ge 0$, dropping the dissipation term, we obtain
	\begin{equation}\label{In Mn inequality}
		\tilde{I}_n(\tilde{\omega})+\tilde{M}_n(\tilde{\omega})\ge 0,\quad {\rm for}\ \tilde{\mathbb{P}}-{\rm a.e.}\ \tilde{\omega}\in\tilde{\Omega}
	\end{equation}
	where
	\begin{equation}\label{tilde Mn def}
		\tilde{M}_n(\omega)=\int_0^T\int_{\mathbb{S}}S'(\tilde{q}_n)\,\tilde{q}_n\,\varphi\,{\rm d}x\,{\rm d}\tilde{W}_n,
	\end{equation}
	and $\tilde{I}_n=\sum_{i=1}^8\tilde{I}_n^{(i)}$ with
	\begin{equation}\label{tilde In def}
		\begin{array}{ll}
			\tilde{I}_n^{(1)}=\int_0^T\int_{\mathbb{S}}S(\tilde{q}_n)\,\partial_t\varphi\,{\rm d}x\,{\rm d}t,\quad 
			&\tilde{I}_n^{(2)}=\int_{\mathbb{S}}S(\partial_x\tilde{u}_{0,n})\,\varphi(0,x)\,{\rm d}x,\\[2ex]
			\tilde{I}_n^{(3)}=\int_0^T\int_{\mathbb{S}}\tilde{u}_n\,S(\tilde{q}_n)\,\partial_x\varphi\,{\rm d}x\,{\rm d}t,\quad &\tilde{I}_n^{(4)}=\varepsilon_n\int_0^T\int_{\mathbb{S}} S(\tilde{q}_n)\,\partial_x^2\varphi\,{\rm d}x\,{\rm d}t,\\[2ex]
			\tilde{I}_n^{(5)}=\int_0^T\int_{\mathbb{S}}H(\tilde{q}_n)\,\varphi \,{\rm d}x\,{\rm d}t,\quad
			&\tilde{I}_n^{(6)}=-\int_0^T\int_{\mathbb{S}}S'(\tilde{q}_n)\,\tilde{P}_n\,\varphi \,{\rm d}x\,{\rm d}t,\\[2ex]
			\tilde{I}_n^{(7)}=\int_0^T\int_{\mathbb{S}}S'(\tilde{q}_n)\,\tilde{u}_n^2\,\varphi\,{\rm d}x\,{\rm d}t,\quad
			&\tilde{I}_n^{(8)}=\frac 1 2\int_0^T\int_{\mathbb{S}}S''(\tilde{q}_n)\,\tilde{q}_n\,\varphi \,{\rm d}x\,{\rm d}t,
		\end{array}
	\end{equation}
	
	We can also write the \eqref{S(q) eneregy} in the form 
	$$\tilde{I}(\tilde{\omega})+\tilde{M}(\tilde{\omega})\ge 0\quad {\rm for}\ \tilde{\mathbb{P}}-{\rm a.e.}\ \tilde{\omega}\in\tilde{\Omega},$$
	where $\tilde{M}(\tilde{{\omega}})$ and $\tilde{I}(\tilde{\omega})\triangleq\sum_{i=1}^8 \tilde{I}^{(i)}$ are as define in \eqref{tilde Mn def} and \eqref{tilde In def} via the corresponding limit terms identified in \eqref{S(q) eneregy}.
	
	\noindent $\circ$ Convergence of deterministic integrals $\tilde{I}^{(i)}_n$.
	
	Here, for the deterministic integrals $\tilde{I}^{(i)}_n$,  we only consider the most challenging choices of $S$, namely $S(v)=\frac 1 2 v^2,\frac 1 2 v_\pm^2$. The argument is same for $S=S_l$ since $S_l(v)\lesssim_l |v|$. According to \eqref{SJ convergence} and \eqref{tilde qn convergence}, it is obvious that
	\begin{equation}\label{Ii convergence 12458}
		\tilde{I}_n^{(i)}\rightarrow\tilde{I}^{(i)},\quad\tilde{\mathbb{P}}-{\rm a.s.}\ \text{and strongly in }L^2(\tilde{\Omega}),\ \forall\ i\notin\{3,6,7\}.
	\end{equation}
	
	For $\tilde{I}_n^{(3)}$, in the view of \eqref{SJ convergence}, $S(\tilde{q}_n)\rightharpoonup\overline{S(q)}$ in $L^r_{t,x}$ a.s., where $r\in[1,3/2)$ is fixed. And by Lemma \ref{priori estimate}, $\tilde{\mathbb{E}}\Vert S(\tilde{q}_n)\Vert_{L^r([0,T]\times\mathbb{S})}^p\lesssim1$. Hence, by a weak compactness argument, we can suppose $\overline{S(q)}\in L^r(\tilde{\Omega}\times[0,T]\times\mathbb{S})$ and 
	\begin{equation}\label{tilde Sn convergence r}
		S(\tilde{q}_n)\rightharpoonup\overline{S(q)}\quad{\rm in}\ L^r(\tilde{\Omega}\times[0,T]\times\mathbb{S}).
	\end{equation}
	Also, $\tilde{u}_n\rightarrow\tilde{u}$ in $C_tL^2_x$ a.s. and $\tilde{E}\Vert\tilde{u}_n\Vert_{L^\infty([0,T];H^1(\mathbb{S}))}^{p_0}\lesssim 1$, by Vitali's convergence theorem we thus obtain $\tilde{u}_n\rightarrow\tilde{u}$ in $L^{r'}(\tilde{\Omega}\times[0,T]\times\mathbb{S})$. Finally, by interpolation for $3<r'<p_0$
	\begin{equation}\label{tilde un convergence r'}
		\tilde{u}_n\rightarrow\tilde{u}\text{ in } L^{r'}(\tilde{\Omega}\times[0,T]\times\mathbb{S}).
	\end{equation}
	Given the strong convergence of $\tilde{u}_n$ \eqref{tilde un convergence r'} and the weak convergence of $S(\tilde{q}_n)$ \eqref{tilde Sn convergence r}, we can deduce that
	$$\tilde{u}_n\,S(\tilde{q}_n)\rightharpoonup\tilde{u}\,\overline{S(q)}\quad{\rm in}\ L^1 (\tilde{\Omega}\times[0,T]\times\mathbb{S}),$$
	therefore,
	\begin{equation}\label{Ii convergence 3}
		\tilde{I}_n^{(3)}\rightharpoonup I^{(3)}\quad{\rm in}\ L^1(\tilde{\Omega}).
	\end{equation}
	
	For $I_n^{(6)}$, by \eqref{SJ convergence} and Lemma \ref{priori estimate}, $S'(\tilde{q}_n)\rightharpoonup \overline{S'(q)}$ in $L^{2r}_{t,x}$ a.s. and $\tilde{\mathbb{E}}\Vert S'(\tilde{q}_n)\Vert_{L^{2r}([0,T]\times\mathbb{S})}^{2r}\lesssim 1$. We can assume $\overline{S'(q)}\in L^{2r}(\tilde{\Omega}\times[0,T]\times\mathbb{S})$ similarly by a weak compactness argument and
	\begin{equation}\label{tilde Sn' convergence 2r}
		S(\tilde{q}_n)\rightharpoonup\overline{S'(q)}\quad{\rm in}\ L^{2r}(\tilde{\Omega}\times[0,T]\times\mathbb{S}).
	\end{equation}
	Then the dual index $3/2<(2r)'<2$, by the Cauchy-Schwarz inequality and \eqref{tilde un convergence r'}, we can deduce that 
	$$\tilde{u}_n\rightarrow\tilde{u}\text{ in } L^{2r'}(\tilde{\Omega}\times[0,T]\times\mathbb{S}).$$
	Combining with \eqref{tilde Sn' convergence 2r}, also
	\begin{equation}\label{Ii convergence 7}
		\tilde{I}_n^{(7)}\rightharpoonup I^{(7)}\quad{\rm in}\ L^1(\tilde{\Omega}).
	\end{equation}
	
	Following the same line, Lemma \ref{tilde Pn convergence} yields 
	\begin{equation}\label{Ii convergence 6}
		\tilde{I}_n^{(6)}\rightharpoonup I^{(6)}\quad{\rm in}\ L^1(\tilde{\Omega}).
	\end{equation}

	\noindent $\circ$ Convergence of stochastic integrals $\tilde{M}_n$
	
	Finally, we deal with the remaining stochastic integral term $\tilde{M}_n$. Denote 
	$$\tilde{J}_n=\int_{\mathbb{S}} S'(\tilde{q}_n)\,\tilde{q}_n\,\varphi\,{\rm d}x.$$
	We divide the argument into two cases: $S=S_l(v_\pm)$ and $S=\frac 1 2 v^2,\frac 1 2v_\pm^2$.
	
	Begin with the case $S(v)=S_l(v_\pm)$, by \eqref{SJ convergence}, $S'(\tilde{q}_n)\,\tilde{q}_n\rightarrow\overline{S'(q)\,q}$ in $L^{2r}(L^{2r}_w)$ a.s. with $2r>2$, since $\varphi\in L^\infty_{t,x}$,
	$$\tilde{J}_n\rightarrow\int_{\mathbb{S}} \overline{S'(q)\,q}\,\varphi\,{\rm d}x=:\tilde{J}\quad{\rm in}\ L^{2r}([0,T]),\text{ a.s.}
	$$
	This implies that $\tilde{J}_n\rightarrow\tilde{J}$ in $L^2([0,T])$ in probability. Besides, $\tilde{W}_n\rightarrow\tilde{W}$ in $C([0,T])$ a.s. in the view of \eqref{SJ convergence} and thus in probability. According to Lemma \ref{convergence of stochastic integrals}, we conclude that 
	$$
		\tilde{\mathcal{M}}_n(\cdot):=\int_0^\cdot\int_{\mathbb{S}} S'(\tilde{q}_n)\,\tilde{q}_n\,\varphi\,{\rm d}x\,{\rm d}\tilde{W}_n\rightarrow\int_0^\cdot\int_{\mathbb{S}} \overline{S'(q)\,q}\,\varphi\,{\rm d}x\,{\rm d}\tilde{W}=:\tilde{\mathcal{M}}(\cdot)
	$$
	in $L^2([0,T])$, in probability. By passing to a subsequence, we may assume this convergence holds $\tilde{\mathbb{P}}$-a.s. 
	
	Next, consider $S(v)=\frac 1 2 v^2$. Thus, $S'(\tilde{q_n})\,\tilde{q}_n=\tilde{q}_n^2\rightarrow\overline{q^2}$ in $L^2_{\tilde{\omega},t}(H^{-1}_x)$ according to \eqref{tilde qn convergence}. This implies $\tilde{J}_n\rightarrow\tilde{J}$ in $L^2([0,T])$ in probability. As above, $\tilde{W}_n$ tends to $\tilde{W}$ in $C([0,T])$ and in probability, by extracting a subsequence, we can also assume the above convergence holds $\tilde{P}$-a.s for $S(v)=\frac 1 2 v^2$.
	
	The case $S(v)=\frac 1 2 v^2_\pm$ can be obtained in the same way and is therefore omitted. In conclusion,
	\begin{equation}\label{tilde mathcal Mn convergence}
		\left\Vert
		\tilde{\mathcal{M}}_n-\tilde{\mathcal{M}}
		\right\Vert_{L^2([0,T])}^2\rightarrow0\quad\tilde{\mathbb{P}}-{\rm a.s.},
	\end{equation}
	and, for any $p\in[2,p_0/2]$, by the BDG inequality and the a priori estimate in Lemma \ref{priori estimate},
	\begin{equation}\label{tilde mathcal Mn control}
		\begin{aligned}
			\tilde{\mathbb{E}}\left\Vert
			\tilde{\mathcal{M}}_n-\tilde{\mathcal{M}}
			\right\Vert_{L^2([0,T])}^p
			&\lesssim_T\tilde{\mathbb{E}}\sup_{t\in[0,T]}\left|\tilde{\mathcal{M}}_n(t)\right|^p+\tilde{\mathbb{E}}\sup_{t\in[0,T]}\left|\tilde{\mathcal{M}}(t)\right|^p\\
			&\lesssim\tilde{\mathbb{E}}\left[\left(\int_0^T\left|\tilde{J_n}(t)\right|^2\,{\rm d}t\right)^{p/2}
			+\left(\int_0^T\left|\tilde{J}(t)\right|^2\,{\rm d}t\right)^{p/2}\right]\\
			&\lesssim_T\tilde{\mathbb{E}}\int_0^T\left|\int_{\mathbb{S}}\tilde{q}_n^2(t)\,{\rm d}x\right|^p\, {\rm d}t
			+\tilde{\mathbb{E}}\int_0^T\left|\int_{\mathbb{S}}\overline{q^2}(t)\,{\rm d}x\right|^p\, {\rm d}t\\
			&\lesssim\tilde{\mathbb{E}}\Vert\tilde{q}_n\Vert_{L^\infty([0,T];L^2(\mathbb{S}))}^{2p}+\tilde{\mathbb{E}}\int_0^T\Vert\overline{q^2}(t)\Vert_{H^{-1}(\mathbb{S})}^p\,{\rm d}t\lesssim 1.
		\end{aligned}		
	\end{equation}
	Given \eqref{tilde mathcal Mn convergence} and \eqref{tilde mathcal Mn control}, Vitali's convergence theorem yields
	$$
	\tilde{\mathbb{E}}\int_0^T\left|\tilde{\mathcal{M}}_n(t)-\tilde{\mathcal{M}}(t)\right|^2\,{\rm d}t\rightarrow0.
	$$
	By Riesz's lemma, passing to a subsequence,
	\begin{equation}\label{Dn convergence}
		D_n(t)=\tilde{\mathbb{E}}\left|\tilde{\mathcal{M}}_n(t)-\tilde{\mathcal{M}}(t)\right|^2\rightarrow 0,\quad\text{for a.e. in } t\in[0,T].
	\end{equation}
	
	Next, we claim that $D_n(t)$ depends continuously on $t\in[0,T]$, uniformly on $n$. Since the remaining cases can be treated analogously, we shall consider only the case 
	$S(v)=\frac 1 2 v^2$ in what follows. By straightforward computation
	$$\left|D_n(t_2)-D_n(t_1)\right|^2\lesssim\tilde{\mathbb{E}}\left|\tilde{\mathcal{M}}_n(t_2)-\tilde{\mathcal{M}}_n(t_1)\right|^2+\tilde{\mathbb{E}}\left|\tilde{\mathcal{M}}(t_2)-\tilde{\mathcal{M}}(t_1)\right|^2,$$
	for any $0\le t_1\le t_2\le T$. The It\'{o} isometry and Lemma \ref{priori estimate} implies
	$$
	\begin{aligned}
		\tilde{\mathbb{E}}\left|\tilde{\mathcal{M}}_n(t_2)-\tilde{\mathcal{M}}_n(t_1)\right|^2
		&=\tilde{\mathbb{E}}\int_{t_1}^{t_2}|\tilde{J}_n(t)|^2\,{\rm d}t\\
		&\lesssim\tilde{\mathbb{E}}\int_{t_1}^{t_2}\left|\int_{\mathbb{S}}\tilde{q}_n^2(t)\,{\rm d}x\right|^2\,{\rm d}t
		\lesssim|t_2-t_1|\,\tilde{\mathbb{E}}\left\Vert\tilde{q}_n\right\Vert_{L^\infty([0,T];L^2(\mathbb{S}))}^4\lesssim|t_2-t_1|,
	\end{aligned}
	$$
	$$
	\begin{aligned}
		\tilde{\mathbb{E}}\left|\tilde{\mathcal{M}}(t_2)-\tilde{\mathcal{M}}(t_1)\right|^2
		&\lesssim\tilde{\mathbb{E}}\int_{t_1}^{t_2}\left\Vert\overline{q^2}(t)\right\Vert_{H^{-1}(\mathbb{S})}^2\,{\rm d}t\\
		&\lesssim|t_2-t_1|^{\frac 1 {p'}}\left(\tilde{\mathbb{E}}\int_0^T\left\Vert\overline{q^2}(t)\right\Vert_{H^{-1}(\mathbb{S})}^{2p}\,{\rm d}t\right)^{1/p}\lesssim|t_2-t_1|^{\frac 1 {p'}}.
	\end{aligned}
	$$
	Above implies that $D_n\in C([0,T])$. Applying the Arzel\`{a}-Ascoli theorem, we can finally deduce that $D_n(t)\rightarrow D(t)\equiv0$ in $C([0,T])$ up to a subsequence. This yields that
	\begin{equation}\label{Mn convergence}
		\tilde{M}_n=\tilde{\mathcal{M}}_n(T)\rightarrow\tilde{\mathcal{M}}(T)=\tilde{M}\quad{\rm in}\ L^2(\tilde{\Omega}).
	\end{equation}
	
	Given \eqref{In Mn inequality}, \eqref{Ii convergence 12458}, \eqref{Ii convergence 3}, \eqref{Ii convergence 6}, \eqref{Ii convergence 7} and \eqref{Mn convergence} imply that for any measurable set $A\in\tilde{\mathcal{F}}$,
	$$\int_{\tilde{\Omega}}\mathds{1}_A(\tilde{\omega})\,\left(I(\tilde{{\omega}})+M(\tilde{{\omega}})\right)\,{\rm d}\tilde{\mathbb{P}}(\tilde{{\omega}})\ge0,$$
	therefore,
	\begin{equation}\label{I M inequality}
		\tilde{I}(\tilde{{\omega}})+\tilde{M}(\tilde{{\omega}})\ge 0,\quad{\rm for}\ \tilde{\mathbb{P}}-{\rm a.e.}\ \tilde{{\omega}}\in\tilde{\Omega},
	\end{equation}
	the inequality \eqref{Sq inequality} is claimed.
\end{proof}

\subsection[Renormalized equation of the weak limit q]{Renormalized equation of the weak limit $\tilde{q}$}

\quad In a view of Proposition \ref{SPDE tilde u}, the Skorokhod-Jakubowski representation $\tilde{u}$ satisfies the below equation weakly in $x$:
\begin{equation}\label{SPDE tilde u without identification}
	\begin{aligned}
		&{\rm d}\tilde{u}+\left[\tilde{u}\,\partial_x\tilde{u}+\partial_x\tilde{P}\right]\,{\rm d}t=\tilde{u}\,{\rm d}\tilde{W},\\
		&\tilde{P}=K\ast\left(\tilde{u}^2+\frac 1 2\overline{q^2}\right).
	\end{aligned}	
\end{equation}
By Lemma \ref{priori estimate}, $\tilde{q}=\partial_x\tilde{u}$ weakly. Differentiating \ref{SPDE tilde u without identification} with respect to $x$, we thus obtain
\begin{equation}\label{SPDE tilde q without identification}
	{\rm d}\tilde{q}+\left(\partial_x(\tilde{u}\,\tilde{q})-\frac 1 2\overline{q^2}+\tilde{P}-\tilde{u}^2\right)\,{\rm d}t=\tilde{q}\,{\rm d}\tilde{W}
\end{equation}

Consider linear growing $S\in W^{2,\infty}(\mathbb{R})$, formally applying It\'{o}'s formula to $S(\tilde{q})$, expressing the temporal differential as a time-derivative in $\mathcal{D}'([0,T))$, we have
\begin{equation}\label{SPDE Sq without identification}
	\partial_t S(\tilde{q})+\partial_x(\tilde{u}\,S(\tilde{q}))+\left[S'(\tilde{q})\,\left(\tilde{P}-\tilde{u}^2\right)-H(\tilde{q})-\frac 1 2S'(\tilde{q})\,\left(\overline{q^2}-\tilde{q}^2\right)-\frac 1 2S''(\tilde{q})\tilde{q}^2\right]=S'(\tilde{q})\,\tilde{q}\,\dot{\tilde{W}},
\end{equation}
in $\mathcal{D}'([0,T)\times\mathbb{S})$, a.s. with $H(\tilde{q})$ defined as Lemma \ref{S(q) eneregy} and initial data $S(\tilde{q})(0)=S(\partial_x\tilde{u}_0)$. Here, relying crucially on the strong right-continuity at $t=0$ be claimed in Lemma \ref{tilde u continuity t=0}. Note that $S$ is assume linearly growing in order to make sense to the product $S'(\tilde{q})(\overline{q^2}-\tilde{q}^2)$, thus, the $S(v)=\frac 1 2 v^2,\frac 1 2 v_\pm^2$ we assume in Lemma \ref{S(q) eneregy} is excludes. Different from \cite{Holden2024JDE}, equation \eqref{SPDE Sq without identification} does not contain a second-order derivative term and is therefore easier to derive. It can be obtained by standard mollification, the real-valued It\'{o} formula, and the DiPerna–Lions estimates \cite{DiPerna1989,Holden2023,Lions1998}, therefore, we omit the detail.

Next, we give the derivation of the SPDE for the renormalization $S(\tilde{q})=S_l(\tilde{q}_\pm)$ of $\frac 1 2 \tilde{q}^2$ and $\frac 1 2 \tilde{q}_\pm^2$. 

\begin{lemm}[Renormalization of limit SPDE]\label{SPDE of Sl}
	Denote by $S(v)$ any one of the $S_l(v_\pm)$ defined in \ref{S_l}. Let $\tilde{u}$, $\tilde{q}=\partial_x\tilde{u}$ and $\overline{q^2}$ be the Skorokhod-Jakubowski representations from Proposition \ref{SJ representations}, $H(v)$ defined in Lemma \ref{S(q) eneregy}. The SPDE \eqref{SPDE Sq without identification} holds weakly in $(t,x)$, almost surely, that is, for all $\varphi\in C^\infty_c([0,T)\times\mathbb{S})$,
	\begin{equation}\label{SPDE Sq integral without identification}
		\begin{aligned}
			&\int_0^T\int_{\mathbb{S}}S(\tilde{q})\,\partial_t\varphi\,{\rm d}x\,{\rm d}t
			+\int_{\mathbb{S}}S(\partial_x\tilde{u}_0)\,\varphi(0,x)\,{\rm d}x
			+\int_0^T\int_{\mathbb{S}}\tilde{u}\,S(\tilde{q})\partial_x\varphi\,{\rm d}x\,{\rm d}t\\
			+&\int_0^T\int_{\mathbb{S}}\left[S'(\tilde{q})\,\left(\tilde{u}^2-\tilde{P}\right)+H(\tilde{q})+\frac 1 2S'(\tilde{q})\,\left(\overline{q^2}-\tilde{q}^2\right)+\frac 1 2S''(\tilde{q})\tilde{q}^2\right]\,\varphi\,{\rm d}x\,{\rm d}t\\
			+&\int_0^T\int_{\mathbb{S}}S'(\tilde{q})\,\tilde{q}\,\varphi\,{\rm d}x\,{\rm d}\tilde{W}=0,\quad\tilde{\mathbb{P}}-{\rm a.s.}
		\end{aligned}
	\end{equation}
\end{lemm}
\begin{proof}
	Let $J_\delta$ be a standard Friedrichs mollifier on $\mathbb{S}$. Denote $f_\delta=f\ast J_\delta$. Mollifying the limit SPDE \eqref{SPDE tilde u}, we have
	\begin{equation}\label{mollify SPDE u}
		{\rm d}\tilde{u}_\delta+\tilde{u}_\delta\,\tilde{q}_\delta\,{\rm d}t+\partial_x K\ast J_\delta\ast\left(\tilde{u}^2+\frac 1 2\overline{q^2}\right)\,{\rm d}t+E_\delta\,{\rm d}t=\tilde{u}_\delta\,{\rm d}\tilde{W},\quad E_\delta=J_\delta\ast(\tilde{u}\,\tilde{q})-\tilde{u}_\delta\,\tilde{q}_\delta.
	\end{equation}
	Differentiating \eqref{mollify SPDE u} with respect to $x$, we deduce that
	\begin{equation}\label{mollify SPDE q}
		{\rm d}\tilde{q}_\delta+\left(\partial_x(\tilde{u}_\delta\,\tilde{q}_\delta)-J_\delta\ast\left(\tilde{u}^2+\frac 1 2\overline{q^2}\right)+K\ast J_\delta\ast\left(\tilde{u}^2+\frac 1 2\overline{q^2}\right)\right)\,{\rm d}t+\partial_xE_\delta\,{\rm d}t=\tilde{q}_\delta\,{\rm d}\tilde{W}.
	\end{equation}
	Consider $S(v)=S_l(v_\pm)$, given \eqref{mollify SPDE q}, by It\'{o}'s formula, we obtain
	$$0=\int_0^T\int_{\mathbb{S}}S(\tilde{q}_\delta)\,\partial_t\varphi\,{\rm d}x\,{\rm d}t+\int_{\mathbb{S}}S(\tilde{q}_\delta(0))\,\varphi(0,x)\,{\rm d}x+\sum_{i=1}^5I_\delta^{(i)},$$
	where
	$$\begin{aligned}
		&I_\delta^{(1)}=\int_0^T\int_{\mathbb{S}}H(\tilde{q}_\delta)\,\varphi+\tilde{u}_\delta\,S(\tilde{q}_\delta)\,\partial_x\varphi\,{\rm d}x\,{\rm d}t,\\
		&I_\delta^{(2)}=-\int_0^T\int_{\mathbb{S}}S'(\tilde{q}_\delta)\, \bar{I}_\delta^{(2)}\,\varphi\,{\rm d}x\,{\rm d}t,\quad I_\delta^{(3)}=\frac 1 2\int_0^T\int_{\mathbb{S}}S''(\tilde{q}_\delta)\,\tilde{q}_\delta^2\,\varphi\,{\rm d}x\,{\rm d}t,\\
		&I_\delta^{(4)}=\int_0^T\int_{\mathbb{S}}S'(\tilde{q}_\delta)\,\tilde{q}_\delta\,\varphi\,{\rm d}x\,{\rm d}\tilde{W},\quad I_\delta^{(5)}\int_0^T\int_{\mathbb{S}}S'(\tilde{q}_\delta)\,\partial_x E_\delta\,\varphi\,{\rm d}x\,{\rm d}t,
	\end{aligned}$$
	with
	$$\bar{I}_\delta^{(2)}=K\ast J_\delta\ast\left(\tilde{u}^2+\frac 1 2\overline{q^2}\right)-J_\delta\ast\left(\tilde{u}^2+\frac 1 2\overline{q^2}\right)+\frac 1 2\tilde{q}_\delta^2.$$
	
	Let $I^{(i)}$ denote the expression obtained by formally taking $\delta$ to zero in$I_\delta^{(i)}$ for $i=1,\dots,5$, and the same for $\bar{I}^{(2)}$. Recall that for $r\in[1,3/2)$ and $p_0>4$, $\tilde{q}\in L^{p_0}_{\tilde{\omega}}L^\infty_tL^2_x\cap L^{2r}_{\tilde{\omega},t,x}$, $\overline{q^2}\in L^r_{\tilde{\omega},t,x}$ and $\tilde{u}\in L^{p_0}_{\tilde{\omega}}L^\infty_{t,x}$. By standard property of mollifiers, $\tilde{q}_\delta\rightarrow\tilde{q}$ a.e. in $\tilde{\Omega}\times[0,T]\times\mathbb{S}$ as $\delta\rightarrow0$. Assume $f(v)$ is nonlinear continuous function of $v$, it is obvious that
	\begin{equation}\label{f continuous convergence a.e.}
		f(\tilde{q}_\delta)\rightarrow f(\tilde{q})\quad\text{a.e. in }\tilde{\Omega}\times[0,T]\times\mathbb{S}.
	\end{equation}
	In particular, according to the definition of $S_l(v_\pm)$ \eqref{S_l}, $S(v)$, $S'(v)$, $S''(v)v^2$ and $H(v)$ are continuous, therefore, \eqref{f continuous convergence a.e.} remains valid for functions taking the aforementioned values.
	
	According to \eqref{S bound}, \eqref{S_l} and Lemma \ref{priori estimate}, we can conclude that
	$$
		\begin{array}{ll}
			\Vert S'(\tilde{q}_\delta)\Vert_{L^\infty_{t,x}}\lesssim_l 1,&\Vert S''(\tilde{q}_\delta) \tilde{q}_\delta^2\Vert_{L^\infty_{t,x}}\lesssim_l 1,\\[2ex]
			\Vert S(\tilde{q}_\delta)\Vert_{L^\infty_t L^2_x}\lesssim_{l,\tilde{{\omega}}} 1,&\Vert S(\tilde{q}_\delta)\Vert_{L^{2r}_{t,x}}\lesssim_{l,\tilde{{\omega}}} 1,\\[2ex]
			\Vert \tilde{q}_\delta^2\Vert_{L^r_{t,x}}\lesssim_{l,\tilde{{\omega}}}1,&\Vert H(\tilde{q}_\delta)\Vert_{L^r_{t,x}}\lesssim_{l,\tilde{{\omega}}}1.
		\end{array}
	$$
	Thus, applying \eqref{f continuous convergence a.e.}, the interpolation inequality and Vitali's convergence theorem, we obtain that
	\begin{equation}\label{f convergence p}
		\begin{aligned}
			&S'(\tilde{q}_\delta)\rightarrow S'(\tilde{q})\quad{\rm in}\ L^p_{t,x}\text{ a.s., for }1\le p<\infty,\\ &S''(\tilde{q}_\delta)\,\tilde{q}_\delta^2\rightarrow S''(\tilde{q})\,\tilde{q}^2\quad{\rm in}\ L^p_{t,x}\text{ a.s., for }1\le p<\infty,\\
			&S(\tilde{q}_\delta)\rightarrow S(\tilde{q})\quad{\rm in}\ L^{p_1}_t L^{p_2}_x\text{ and }L^p_{t,x}\text{ a.s., for }1\le p_1<\infty,1\le p_2<2,1\le p<2r,\\
			&\tilde{q}_\delta^2\rightarrow\tilde{q}^2\quad{\rm in}\ L^p_{t,x}\text{ a.s., for }1\le p<r,\\
			&H(\tilde{q}_\delta)\rightarrow H(\tilde{q})\quad{\rm in}\ L^p_{t,x}\text{ a.s., for }1\le p<r.
		\end{aligned}
	\end{equation}
	
	Besides, $\overline{q^2}\in L^r_{t,x}$ and $\tilde{u}\in L^\infty_{t,x}$, we have
	\begin{equation}\label{tilde u and q mollifier convergenve}
		\begin{aligned}
			&\overline{q^2}\ast J_\delta\rightarrow\overline{q^2}\quad{\rm in}\ L^r_{t,x},\text{ a.s.},\\
			&\tilde{u}_\delta\rightarrow\tilde{u}\quad{\rm in}\ L^p_{t,x}\text{, a.s., for any }p\in[1,\infty),\\
			&\tilde{u}^2\ast J_\delta\rightarrow\tilde{u}^2\quad{\rm in}\ L^p_{t,x}\text{, a.s., for any }p\in[1,\infty).
		\end{aligned}
	\end{equation}
	In view of Young's convolution inequality and \eqref{tilde u and q mollifier convergenve}, it is easily to obtain that for $p\in[1,r)$, 
	\begin{equation}\label{tilde P mollifier convergenve}
		\bar{I}_\delta^{(2)}\longrightarrow K\ast\left(\tilde{u}^2+\frac 1 2\overline{q^2}\right)-\tilde{u}^2-\frac 1 2\left(\overline{q^2}-\tilde{q}^2\right)=:\bar{I}^{(2)}\quad{\rm in}\ L^p_{t,x}\text{, a.s.}
	\end{equation}
	 
	 Consider the term of initial data, according to \eqref{S temporal estimate} and $S(\tilde{q}_\delta(0))\rightarrow S(\tilde{q}(0))$ in $L^1_x$ a.s. similarly, we deduce that
	 $$
	 \begin{aligned}
	 	&\int_{\mathbb{S}} S(\tilde{q}_\delta(0))\,\varphi(0,x)\,{\rm d}x\\
	 	=&\int_{\mathbb{S}}S(\partial_x\tilde{u}_0)\,\varphi(0,x)\,{\rm d}x+\int_{\mathbb{S}}\left(S(\tilde{q}_\delta(0))-S(\tilde{q}(0)) \right)\,\varphi(0,x)\,{\rm d}x+\int_{\mathbb{S}}\left(S(\tilde{q}(0))-S(\partial_x\tilde{u}_0) \right)\,\varphi(0,x)\,{\rm d}x\\
	 	=&\int_{\mathbb{S}}S(\partial_x\tilde{u}_0)\,\varphi(0,x)\,{\rm d}x+\int_{\mathbb{S}}\left(S(\tilde{q}_\delta(0))-S(\tilde{q}(0)) \right)\,\varphi(0,x)\,{\rm d}x\longrightarrow\int_{\mathbb{S}}S(\partial_x\tilde{u}_0)\,\varphi(0,x)\,{\rm d}x,\quad{\rm a.s.}
	 \end{aligned}
	 $$
	 Since $\varphi\in C^\infty_c([0,T)\times\mathbb{S})$, by \eqref{f convergence p}, \eqref{tilde u and q mollifier convergenve} and \eqref{tilde P mollifier convergenve}, the following convergences follows directly:
	 $$I_\delta^{(i)}\rightarrow I^{(i)},\quad i=1,2,3.$$
	 
	 Finally, it remains to handle the error term $I_\delta^{(5)}$ and the stochastic integral term $I_\delta^{(4)}$. By the It\'{o} isometry and the Cauchy-Schwarz inequality, since $\tilde{q}_\delta\rightarrow\tilde{q}$ in $L^{2r}_{t,x}$ a.s.,
	 $$\tilde{\mathbb{E}}\left|I_\delta^{(4)}-I^{(4)}\right|^2\le \tilde{\mathbb{E}}\int_0^t\int_{\mathbb{S}}\left|S'(\tilde{q}_\delta)\,\tilde{q}_\delta-S'(\tilde{q})\,\tilde{q}\right|^2\,|\varphi|^2\,{\rm d}x\,{\rm d}t\xrightarrow{\eqref{f convergence p}}0.$$
	 Therefore, up to a subsequence, $I_\delta^{(4)}\rightarrow I^{(4)}$, a.s. And according to the standard commutator estimate \cite{Debussche2016,Holden2024JDE,Lions1998}, $\tilde{\mathbb{E}}\left\Vert \partial_x E_\delta\right\Vert_{L^1([0,T]\times\mathbb{S})}\rightarrow0$, similarly, we conclude that up to a subsequence, $I_\delta^{(5)}\rightarrow 0$, a.s.
	 
	 This concludes the proof of the lemma.
\end{proof}

According to Lemma \ref{S(q) eneregy}, substituting $S(v)=\frac 1 2 v_+^2$,
\begin{equation}\label{inequality overline q+^2}
	\begin{aligned}
		&\int_0^T\int_{\mathbb{S}}\frac 1 2\,\overline{q_+^2}\,\partial_t\varphi\,{\rm d}x\,{\rm d}t+\int_{\mathbb{S}}\frac 1 2\,(\partial_x\tilde{u}_0)_+^2\,\varphi(0,x)\,{\rm d}x
		\\
		+&\int_0^T\int_{\mathbb{S}}\left[\overline{q_+}\left(\tilde{u}^2-\tilde{P}\right)+\frac 1 2\,\overline{q_+^2}\right]\varphi\,{\rm d}x\,{\rm d}t+\int_0^T\int_{\mathbb{S}}\frac 1 2\,\overline{q_+^2}\,\tilde{u}\,\partial_x\varphi\,{\rm d}x\,{\rm d}t\\
		+&\int_0^T\int_{\mathbb{S}}\overline{q_+^2}\,\varphi\,{\rm d}x\,{\rm d}\tilde{W}\ge 0,\quad\tilde{\mathbb{P}}-{\rm a.s.,}
	\end{aligned}
\end{equation}

Lemma \ref{SPDE of Sl} is applicable to to the linearly growing approximations $S_l(v_\pm)$ of $v_\pm^2$, but not to $v_\pm$ themselves. However, exploiting some structural properties of SPDE \eqref{Sq equation}, we can obtain an inequality of the positive part $\tilde{q}_+^2$. Together with \eqref{inequality overline q+^2}, it is possible to control the positive part $\frac 1 2\left(\overline{q^2_+}-\tilde{q}_+^2\right)$ of defect measure.

\begin{lemm}[Characterization of $\tilde{q}_+^2$]\label{characterization of tilde q+^2}
	Let $\tilde{u}$, $\tilde{q}=\partial_x\tilde{u}$ and $\overline{q^2}$ be the Skorokhod-Jakubowski representations from Proposition \ref{SJ representations}. Then for any nonnegative $\varphi\in C^\infty_c([0,T)\times\mathbb{S})$,
	\begin{equation}\label{inequality tilde q+^2}
		\begin{aligned}
			&\int_0^T\int_\mathbb{S}\frac 1 2\,\tilde{q}_+^2\,\partial_t\varphi\,{\rm d}x\,{\rm d}t+\int_{\mathbb{S}}\frac 1 2\,(\partial_x\tilde{u}_0)_+^2\,\varphi(0,x)\,{\rm d}x\\
			+&\int_0^T\int_{\mathbb{S}}\left[\tilde{q}_+\left(\tilde{u}^2-\tilde{P}\right)+\frac 1 2\,\tilde{q}_+^2\right]\,\varphi\,{\rm d}x
			,{\rm d}t+\int_0^T\int_{\mathbb{S}}\frac 1 2\,\tilde{q}_+^2\,\tilde{u}\,\partial_x\varphi\,{\rm d}x\,{\rm d}t\\
			+&\int_0^T\int_{\mathbb{S}}\tilde{q}_+^2\,\varphi\,{\rm d}x\,{\rm d}\tilde{W}\le0,\quad\tilde{\mathbb{P}}-{\rm a.s.,}
		\end{aligned}
	\end{equation}
	where $\tilde{P}$ is defined in \eqref{SPDE tilde u without identification}.
\end{lemm}
\begin{proof}
	Similar to the notations in Lemma \ref{S(q) eneregy}, denote the left-hand side of \eqref{inequality tilde q+^2} by $I+M$, where $I$ is the sum of the deterministic integral terms and $M$ is the stochastic integral term. Consider $A\in\tilde{\mathcal{F}}$ is an arbitrary measurable set. Given $S(v)=S_l(v_+)$ in \eqref{SPDE Sq integral without identification}, according to the definition of $S_l(v_+)$ in \eqref{S_l}, by direct computation and the weak lower semicontinuity of convex function $v^2$, we observe that
	\begin{equation}\label{H tilde q>0}
			H(\tilde{q})=S(\tilde{q})\,\tilde{q}-\frac 1 2\,S'(\tilde{q})\,\tilde{q}^2\ge 0,\quad S'(\tilde{q})\,\left(\overline{q^2}-\tilde{q}^2\right)\ge0.
	\end{equation}
	
	In particular, 
	\begin{equation}\label{q+^2 error term}
		\begin{array}{ll}
			S(\tilde{q})=\frac 1 2\,\tilde{q}_+^2+e_l^{(1)}(t,x), &S'(\tilde{q})=\tilde{q}_+ +e_l^{(2)}(t,x),\\[2ex]
			\frac 1 2\,S''(\tilde{q})\,\tilde{q}^2=\frac 1 2\,\tilde{q}_+^2+e_l^{(3)}(t,x),
			&S'(\tilde{q})\,\tilde{q}=\tilde{q}_+^2+e_l^{(4)}(t,x),
		\end{array}
	\end{equation}
	where
	$$
	\begin{aligned}
		&e_l^{(1)}=-\frac 1 {6l}\,(\tilde{q}-l)^3\,\mathds{1}_{\{l<\tilde{q}<2l\}}-\frac 1 6\left(3\tilde{q}^2-9l\tilde{q}+7l^2\right)\,\mathds{1}_{\{\tilde{q}\ge 2l\}},\\
		&e_l^{(2)}=-\frac 1 {2l}\,\left(\tilde{q}-l\right)^2\,\mathds{1}_{\{l<\tilde{q}<2l\}}+\frac 1 2(3l-2\tilde{q})\,\mathds{1}_{\{\tilde{q}\ge 2l\}},\\
		&e_l^{(3)}=-\frac 1 {2l}\,\tilde{q}^2\,(\tilde{q}-l)\,\mathds{1}_{\{l<\tilde{q}<2l\}}-\frac 1 2\,\tilde{q}^2\,\mathds{1}_{\{\tilde{q}>2l\}},\\
		&e_l^{(4)}=-\frac 1 {2l}\,\tilde{q}\,\left(\tilde{q}-l\right)^2\,\mathds{1}_{\{l<\tilde{q}<2l\}}+\frac 1 2\tilde{q}\,(3l-2\tilde{q})\,\mathds{1}_{\{\tilde{q}\ge 2l\}}.
	\end{aligned}
	$$
	By Lemma \ref{priori estimate}, $\tilde{q}\in L^{p_0}\left(\tilde{\Omega};L^\infty([0,T];L^2(\mathbb{S}))\right)$, with $p_0>4$ fixed in Theorem \ref{main theorem}. Due to the integrability of $\tilde{q}$,
	$$\mathds{1}_{\{\tilde{q}>l\}}\,\tilde{q}\rightarrow 0\quad\text{and}\quad \mathds{1}_{\{\tilde{q}>l\}}\,\tilde{q}^2\rightarrow 0\quad\text{a.e. in }(\tilde{{\omega}},t,x).$$
	As a result, for $1\le i\le4$,
	$$\left|e_l^{(i)}\right|\lesssim \mathds{1}_{\{\tilde{q}>l\}}\,\tilde{q}+\mathds{1}_{\{\tilde{q}>l\}}\,\tilde{q}^2\rightarrow0\quad\text{a.e. in }(\tilde{\omega},t,x).$$
	
	Inserting \eqref{H tilde q>0} and \eqref{q+^2 error term} into \eqref{SPDE Sq integral without identification}, we arrive at
	$$
	\begin{aligned}
		&\int_{\tilde{\Omega}}\mathds{1}_A(\tilde{{\omega}})\left(I(\tilde{{\omega}})+M(\tilde{{\omega}})\right)\,{\rm d}\tilde{\mathbb{P}}(\tilde{\omega})\\
		\lesssim\ \,& \tilde{\mathbb{E}}\int_{\mathbb{S}}\left|e_l^{(1)}(0)\right|\,{\rm d}x +\tilde{\mathbb{E}}\int_0^T\int_{\mathbb{S}}\left|e_l^{(1)}\right|+\left|e_l^{(3)}\right|\,{\rm d}x\,{\rm d}t +\tilde{\mathbb{E}}\int_0^T\int_{\mathbb{S}}|\tilde{u}| \,|\left|e_l^{(1)}\right|\,{\rm d}x\,{\rm d}t\\ +&\tilde{\mathbb{E}}\int_0^T\int_{\mathbb{S}}\left|\tilde{P}-\tilde{u}^2\right|\,\left|e_l^{(2)}\right|\,{\rm d}x\,{\rm d}t +\tilde{\mathbb{E}}\left|\int_0^T\int_{\mathbb{S}}e_l^{(4)}\,\varphi\,{\rm d}x\,{\rm d}\tilde{W}\right|=:\sum_{i=1}^5 E_{i,l},
	\end{aligned}
	$$
	for all $l\in\mathbb{N}$.
	
	Similarly, $e_l^{(1)}(0)$ refers to
	$$-\frac 1 {6l}\,(\partial_x\tilde{u}_0-l)^3\,\mathds{1}_{\{l<\partial_x\tilde{u}_0<2l\}}-\frac 1 6\left(3(\partial_x\tilde{u}_0)^2-9l\partial_x\tilde{u}_0+7l^2\right)\,\mathds{1}_{\{\partial_x\tilde{u}_0\ge 2l\}},$$
	according to Lemma \ref{priori estimate}, it is easy to obtain that
	$$
	\begin{aligned}
		&\left|e_l^{(1)}(0)\right|\lesssim (\partial_x \tilde{u}_0)^2\in L^1_{\tilde{{\omega}},t,x},\quad\quad \left|e_l^{(1)}\right|\lesssim\tilde{q}^2\in L^1_{\tilde{{\omega}},t,x},\quad
		\left|e_l^{(3)}\right|\lesssim\tilde{q}^2\in L^1_{\tilde{{\omega}},t,x},\\ &\left|e_l^{(1)}\right|\,|\tilde{u}|\lesssim |\tilde{u}|\,\tilde{q}^2\in L^{\frac {p_0} 2}_{\tilde{{\omega}}}L^\infty_tL^1_x,\quad
		\left|\tilde{P}-\tilde{u}^2\right|\,\left|e_l^{(2)}\right|\lesssim |\tilde{P}|\,\tilde{q}^2+\tilde{u}^2\tilde{q}^2\in L^{\frac {p_0} 4}_{\tilde{{\omega}}}L^\infty_tL^1_x.
	\end{aligned}
	$$
	Therefore, following the Lebesgue dominated theorem, $\sum_{i=1}^4E_{i,l}\rightarrow0$ as $l$ tends to $\infty$.
	
	Finally, let us consider the stochastic term. Utilizing the Cauchy-Schwarz inequality, the It\'{o} isometry and the Lebesgue dominated convergence theorem,
	$$|E_{5,l}|^2\le\tilde{\mathbb{E}}\int_0^T \left|\int_{\mathbb{S}}e_l^{(4)}\,\varphi\,{\rm d}x \right|^2\,{\rm d}t\lesssim_{\varphi}\tilde{\mathbb{E}}\int_0^T\left|\int_\mathbb{S}\mathds{1}_{\{\tilde{q}_+>l\}}\,\tilde{q}_+^2\,{\rm d}x\right|^2\,{\rm d}t\longrightarrow0.$$
	
	Therefore, we conclude that
	$$\int_{\tilde{\Omega}}\mathds{1}_A(\tilde{{\omega}})\left(I(\tilde{{\omega}})+M(\tilde{{\omega}})\right)\,{\rm d}\tilde{\mathbb{P}}(\tilde{\omega})\\
	\le0,$$
	the arbitrariness of $A\in\tilde{\Omega}$ completes the proof of \eqref{inequality tilde q+^2}.
\end{proof}

\subsection{Control of the defect measure}
\quad Denote the positive part of the defect measure \eqref{defect measure}
\begin{equation}\label{defect measure positive}
	\mathbb{D}^+=\frac 1 2\left(\overline{q_+^2}-\tilde{q}_+^2\right)\ge 0.
\end{equation}
by directly subtracting \eqref{inequality overline q+^2} and \eqref{inequality tilde q+^2},
\begin{equation}\label{inequality D+}
	\partial_t\mathbb{D}^+ 
	+\left[\partial_x\left(\tilde{u}\,\mathbb{D}^+\right)+\left(\overline{q_+}-\tilde{q}_+\right)\left(\tilde{P}-\tilde{u}\right)-\mathbb{D}^+
	\right] -\mathbb{D}^+\dot{\tilde{W}}\le 0\quad \text{in }\mathcal{D}'([0,T)\times\mathbb{S}),\ \tilde{\mathbb{P}}-{\rm a.s.},
\end{equation}
with zero initial data in the sense of distributions. This formulation is weak in $(t,x)$. Similar to the method used in the proof of Lemma \ref{energy and one-sided temporal}, it can be transform into pointwise in $(\tilde{\omega},t)$ and integrated in $x$.

\begin{lemm}[Positive part of defect measure]\label{positive part of defect measure}
	Let $\mathbb{D}^+$ be defined as \eqref{defect measure positive}, then for a.e. $(\tilde{{\omega}},t)\in\tilde{\Omega}\times[0,T]$,
	\begin{equation}\label{inequality D+ integral}
		\begin{aligned}
			\int_{\mathbb{S}}\mathbb{D}^+(t)\,{\rm d}x+\int_0^t\int_{\mathbb{S}}(\overline{q_+}-\tilde{q}_+)\left(\tilde{P}-\tilde{u}^2\right)-\mathbb{D}^+\,{\rm d}x\,{\rm d}s-\int_0^t\int_{\mathbb{S}}\mathbb{D}^+\,{\rm d}x\,{\rm d}\tilde{W}\le 0.
		\end{aligned}
	\end{equation}
	Moreover, the stochastic integral is a square-integrable martingale.
\end{lemm}
\begin{proof}
	Consider the test function $\varphi(t,x)=\psi(t)$ with arbitrary nonnegative $\psi\in C^\infty_c([0,T))$, then
	$$
	\frac {\rm d}{{\rm d}t}\int_{\mathbb{S}}\mathbb{D}^+\,{\rm d}x+\int_{\mathbb{S}}(\overline{q_+}-\tilde{q}_+)\left(\tilde{P}-\tilde{u}^2\right)-\mathbb{D}^+\,{\rm d}x-\int_{\mathbb{S}}\mathbb{D}^+\,{\rm d}x\dot{\tilde{W}}\le 0\quad\text{in }\mathcal{D}'([0,T)),\text{ a.s.}
	$$
	
	Following the proof of Lemma \ref{energy and one-sided temporal}, consider the piecewise continuous linear function $\beta_\delta$ equals $1$ on $[0,t-\delta]$, $0$ on $[t,T]$ and is linear on $[t-\delta,t]$, we thus derive that
	$$
	\frac 1 \delta\int_{t-\delta}^t\int_{\mathbb{S}}\mathbb{D}^+(s)\,{\rm d}x\,{\rm d}t+\int_0^T\int_{\mathbb{S}}\left[(\overline{q_+}-\tilde{q}_+)\left(\tilde{P}-\tilde{u}^2\right)-\mathbb{D}^+\right]\,\beta_\delta(s)\,{\rm d}x\,{\rm d}s-\int_0^T\int_{\mathbb{S}}\mathbb{D}^+\,\beta_\delta(s)\,{\rm d}x\,{\rm d}\tilde{W}(s)\le 0.
	$$
	According to Lemme \ref{priori estimate}, since $|\mathbb{D}^+|\lesssim |\overline{q^2}|+\tilde{q}^2$,
	$$\tilde{\mathbb{E}}\int_0^T\left|\int_{\mathbb{S}}\mathbb{D}^+\,{\rm d}x\right|\,{\rm d}t\lesssim \tilde{\mathbb{E}}\int_0^T\left\Vert\overline{q^2}\right\Vert_{H^{-1}(\mathbb{S})}^2\,{\rm d}t+\tilde{\mathbb{E}}\int_0^T\left\Vert\tilde{q}^2\right\Vert_{H^{-1}(\mathbb{S})}^2\,{\rm d}t<\infty,$$
	hence, the stochastic integral is a square-integrable martingale.
	
	Taking $\delta\rightarrow0$ and following the same line in the proof of Lemma \ref{energy and one-sided temporal}, we can similarly deduce \eqref{inequality D+ integral}, the details are omitted here.
\end{proof}

For the negative part of defect measure 
\begin{equation}\label{defect measure negative}
	\mathbb{D}^-=\frac 1 2\left(\overline{q_-^2}-\tilde{q}_-^2\right),
\end{equation} 
we consider the approximates at first:
\begin{equation}\label{defect measure negative approximate}
	\mathbb{D}_l^-=\overline{S_l(q_-)}-S_l(\tilde{q}_-)\ge 0,\quad l\in\mathbb{N}.
\end{equation}
At the beginning, we provide an estimate for the approximation error:
\begin{lemm}\label{convergence Dl-}
	Let $r\in(1,3/2)$, then
	$$\left|\tilde{\mathbb{E}}\int_0^T\int_{\mathbb{S}}\mathbb{D}_l^- -\mathbb{D}^-\,{\rm d}x\,{\rm d}t\right|\lesssim l^{-2(r-1)}.$$
\end{lemm}
\begin{proof}
	According to \eqref{SJ convergence} and weak compactness argument, $\tilde{q}_n^2\rightharpoonup\overline{q^2}$ in $L^r_{\tilde{\omega},t,x}$ and $S_l(\tilde{q}_n)\rightharpoonup\overline{S_l(q)}$ in $L^{2r}_{\tilde{{\omega}},t,x}$, since $1\in L^{r'}\cap L^{(2r)'}(\tilde{\Omega}\times[0,T]\times\mathbb{S})$,
	$$
	\tilde{\mathbb{E}}\int_0^T\int_{\mathbb{S}}\mathbb{D}_l^- -\mathbb{D}^-\,{\rm d}x\,{\rm d}t=\lim_{n\rightarrow\infty}\tilde{\mathbb{E}}\int_0^T\int_{\mathbb{S}}\left(S_l\left((\tilde{q}_n)_-\right)-(\tilde{q}_n)_-^2\right)+\left(S_l\left(\tilde{q}_-\right)-\tilde{q}_-^2\right)\,{\rm d}x\,{\rm d}t.
	$$
	By direct calculation, for $1<r<3/2$, 
	$$\left|S_l(v_-)-v_-^2\right|\lesssim\frac 1 l\,|v+l|^3\,\mathds{1}_{\{-2l\le v\le -l\}}+v^2\,\mathds{1}_{\{v\le-2l\}}\lesssim v^2\,\mathds{1}_{\{|v|\ge l\}}\lesssim l^{2(1-r)}\,v^{2r}\,\mathds{1}_{\{|v|\ge l\}},
	$$
	therefore, combining with Lemma \ref{priori estimate},
	$$
	\begin{aligned}
		&\left|\tilde{\mathbb{E}}\int_0^T\int_{\mathbb{S}}\left(S_l\left((\tilde{q}_n)_-\right)-(\tilde{q}_n)_-^2\right)+\left(S_l\left(\tilde{q}_-\right)-\tilde{q}_-^2\right)\,{\rm d}x\,{\rm d}t\right|\\
		\le& \ l^{2(1-r)}\,\tilde{\mathbb{E}}\int_0^T\int_{\mathbb{S}}|\tilde{q}_n|^{2r}\,\mathds{1}_{\{|\tilde{q}_n|\ge l\}}+|\tilde{q}|^{2r}\,\mathds{1}_{\{|\tilde{q}|\ge l\}}\,{\rm d}x\,{\rm d}t\lesssim l^{2(1-r)}.
	\end{aligned}
	$$
	Sending $n\rightarrow\infty$ in the above expression establishes the lemma.
\end{proof}

Substituting $S(v)=S_l(v_-)$ into \eqref{Sq equation} and \eqref{SPDE Sq without identification} yields 
\begin{equation}\label{inequality overline Sl-}
	\begin{aligned}
		&\partial_t\overline{S_l(q_-)}+\partial_x(\tilde{u}\,\overline{S_l(q_-)})\\
		+&
		\left[\overline{S_l(q_-)'}\,\left(\tilde{P}-\tilde{u}^2\right)-\overline{H_{l,-}(q)}-\frac 1 2\overline{S_l(q_-)''\, q^2}\right]-\overline{S_l(q_-)'\,q}\,\dot{\tilde{W}}\le 0\quad\tilde{\mathbb{P}}-{\rm a.s.},
	\end{aligned}
\end{equation}
and
\begin{equation}\label{inequality tilde Sl-}
	\begin{aligned}
		&\partial_t S_l(\tilde{q}_-)+\partial_x(\tilde{u}\,S_l(\tilde{q}_-))\\
		+&
		\left[S_l(\tilde{q}_-)\,\left(\tilde{P}-\tilde{u}^2\right)-H_{l,-}(\tilde{q})-\frac 1 2 S_l(\tilde{q}_-)'\left(\overline{q^2}-\tilde{q}^2\right)-\frac 1 2S_l(\tilde{q}_-)''\,\tilde{q}^2\right]\\
		-&
		S_l(\tilde{q}_-)'\,\tilde{q}\,\dot{\tilde{W}}=0\quad\tilde{\mathbb{P}}-{\rm a.s.,}
	\end{aligned}
\end{equation}
in the sense of  $\mathcal{D}'([0,T)\times\mathbb{S})$, where
\begin{equation}\label{def Hl-}
	H_{l,-}(v)=S_l(v_-)\,v-\frac 1 2S_l(v_-)'\,v^2,\quad\overline{H_{l,-}(v)}=\overline{S_l(v_-)\,v}-\frac 1 2 \overline{S_l(v_-)'\,v^2}.
\end{equation}

Subtracting \eqref{inequality overline Sl-} from \eqref{inequality tilde Sl-}, we obtain that 
\begin{equation}\label{inequality Dl-}
	\begin{aligned}
		&\partial_t\mathbb{D}_l^- +\partial_x(\tilde{u}\,\mathbb{D}_l^-)\\
		+&\left[\left(\overline{S_l(q_-)'}-S_l(\tilde{q}_-)'\right)\,\left(\tilde{P}-\tilde{u}^2\right)-\left(\overline{H_{l,-}(q)}-H_{l,-}(\tilde{q})\right)+\frac 1 2 S_l(\tilde{q}_-)'\left(\overline{q^2}-\tilde{q}^2\right)\right]\\
		-&\frac 1 2\left(\overline{S_l(q_-)''\, q^2}-S_l(\tilde{q}_-)''\,\tilde{q}^2\right)-\left( \overline{S_l(q_-)'\,q}-S_l(\tilde{q}_-)'\,\tilde{q}\right)\,\dot{\tilde{W}}\le 0\quad\text{in }\mathcal{D}'([0,T)\times\mathbb{S}),\ \tilde{\mathbb{P}}-{\rm a.s.,}
	\end{aligned}
\end{equation}
with zero initial data $\mathbb{D}_l^-(0)=0$.

\begin{lemm}[Negative part of defect measure]\label{negative part of defect measure}
	Let $\mathbb{D}_l^-$ and $\mathbb{D}^-$ be defined by \eqref{defect measure negative approximate} and \eqref{defect measure negative} respectively. Let $\{\tilde{u}_n\}$ and $\{\tilde{P}_n\}$ be the Skorokhod-Jakubowski representations in Proposition \ref{SJ representations}, then for any $n_0\in\mathbb{N}$ and $L>0$, define the measurable set
	\begin{equation}\label{A_L^n0}
		A_L^{n_0}=\left\{
		\tilde{{\omega}}\in\tilde{\Omega}:\left\Vert\tilde{P}_{n_0}-\tilde{u}_{n_0}^2\right\Vert_{L^\infty([0,T]\times\mathbb{S})}\le L\right\},
	\end{equation} 
	which satisfies $\tilde{\mathbb{P}}(A_L^{n_0})\rightarrow1$ as $L\rightarrow\infty$, uniformly in $n_0$.
	
	For a.e. $t\in[0,T]$ and sufficiently large $l$ depending on $L$, such that
	\begin{equation}\label{inequality integral Dl-}
		\begin{aligned}
			&\int_{\mathbb{S}}\mathbb{D}_l^-(t)\,{\rm d}x+\int_0^t\int_{\mathbb{S}}\left(\overline{q_-}-\tilde{q}_-\right)\,\left(\tilde{P}-\tilde{u}^2\right)\,{\rm d}x\,{\rm d}s\\
			+
			&\int_0^t\int_{\mathbb{S}}\left( \overline{S_l(q_-)'-q_-} - \left(S_l(\tilde{q}_-)' - \tilde{q}_-\right) \right)\,\left(\tilde{P}-\tilde{u}^2-\tilde{P}_{n_0}+\tilde{u}_{n_0}^2\right)\,{\rm d}x\,{\rm d}s\\
			-
			&\int_0^t\int_{\mathbb{S}} \frac {3l} 2\,(\mathbb{D}_l^- +\mathbb{D}^+)+\mathbb{D}_l^-\,{\rm d}x\,{\rm d}s
			+\int_0^t\int_{\mathbb{S}}S_l(\tilde{q}_-)'\,\tilde{q}-\overline{S_l(q_-)'\,q} \,{\rm d}x\,{\rm d}\tilde{W}\le 0,\quad\text{a.s. on }A_L^{n_0}.
		\end{aligned}
	\end{equation}
	In particular, the stochastic integral $\mathcal{M}_l^-(t):=\int_0^t\int_{\mathbb{S}}S_l(\tilde{q}_-)'\,\tilde{q} -\overline{S_l(q_-)'\,q}\,{\rm d}x\,{\rm d}\tilde{W}$ is a square-integrable martingale with $\tilde{\mathbb{E}}|\mathcal{M}_l^-(T)|^2\lesssim_l 1.$
\end{lemm}
\begin{proof}
	Taking the test function $\varphi(t,x)=\psi(t)$ with $0\le\psi\in C^\infty_c([0,T))$ in \eqref{inequality Dl-}, we obtain that
	\begin{equation}\label{inequality Dl- dt}
		\frac {\rm d} {{\rm d}t}\int_{\mathbb{S}}\mathbb{D}_l^-\,{\rm d}x+\sum_{i=1}^3 I_l^{(i)}\,{\rm d}x+\int_{\mathbb{S}} I_l^{(4)}\,{\rm d}x\,\dot{\tilde{W}}\le 0,
	\end{equation}
	which holds in $\mathcal{D}'([0,T))$, $\tilde{\mathbb{P}}$-a.s., with zero initial data, where
	$$
	\begin{aligned}
		&I_l^{(1)}=\left(\overline{S_l(q_-)'}-S_l(\tilde{q}_-)'\right)\,\left(\tilde{P}-\tilde{u}^2\right),\\
		&I_l^{(2)}=\frac 1 2 S_l(\tilde{q}_-)'\left(\overline{q^2}-\tilde{q}^2\right)
		-\left(\overline{H_{l,-}(q)}-H_{l,-}(\tilde{q})\right),\\
		&I_l^{(3)}=-\frac 1 2\left(\overline{S_l(q_-)''\, q^2}-S_l(\tilde{q}_-)''\,\tilde{q}^2\right),\\
		&I_l^{(4)}=-\left( \overline{S_l(q_-)'\,q}-S_l(\tilde{q}_-)'\,\tilde{q}\right).
	\end{aligned}
	$$
	\noindent$\circ$ The term $I_l^{(1)}$.
	
	By the definition of $S_l(v)$, we have
	$$I_l^{(1)}=(\overline{q_-}-\tilde{q}_-)\left(\tilde{P}-\tilde{u}^2\right)+E_l^{(1)}\left(\tilde{P}_{n_0}-\tilde{u}_{n_0}^2\right)+E_l^{(1)}\left(\tilde{P}-\tilde{u}^2-\tilde{P}_{n_0}+\tilde{u}_{n_0}^2\right),$$
	where
	$$\begin{aligned}
		E_l^{(1)}=&\ \overline{S_l(q_-)'-q_-}-\left(S_l(\tilde{q}_-)'-\tilde{q}_-\right)\\
		=&\ \overline{\frac 1 {2l}(q+l)^2\,\mathds{1}_{\{-2l<q<-l\}}-\frac 1 2(3l+2q)\,\mathds{1}_{\{q\le -2l\}}}\\
		&-\left( \frac 1 {2l}(\tilde{q}+l)^2\,\mathds{1}_{\{-2l<\tilde{q}<-l\}}-\frac 1 2(3l+2\tilde{q})\,\mathds{1}_{\{\tilde{q}\le -2l\}}\right).
	\end{aligned}$$
	according to the convexity of function $v\mapsto \frac 1 {2l}(v+l)^2\,\mathds{1}_{\{-2l<v<-l\}}-\frac 1 2(3l+2v)\,\mathds{1}_{\{v\le -2l\}}$, by a standard result on the weak convergence of convex function sequences (cf. Corollary 3.33 in \cite{Novotny2004}), we have 
	$$
	I_l^{(1)}\ge (\overline{q_-}-\tilde{q}_-)\left(\tilde{P}-\tilde{u}^2\right)-E_l^{(1)}\left\Vert\tilde{P}_{n_0}-\tilde{u}_{n_0}^2\right\Vert_{L^\infty([0,T]\times\mathbb{S})}+E_l^{(1)}\left(\tilde{P}-\tilde{u}^2-\tilde{P}_{n_0}+\tilde{u}_{n_0}^2\right).$$
	
	\noindent$\circ$ The term $I_l^{(2)}$.
	
	Notice that $\overline{q^2}-\tilde{q}^2=2\mathbb{D}^- + 2\,\mathbb{D}^+$, by direct computation,
	$$\frac 1 2 S_l(\overline{q^2}-\tilde{q}^2)\,\left(\overline{q^2}-\tilde{q}^2\right)=S_l(\tilde{q}_-)'\,\mathbb{D}^+ +S_l(\tilde{q}_-)'\left(\overline{S_l(q_-)}-S_l(\tilde{q}_-)\right)+S_l(\tilde{q}_-)'\,\tilde{E}_l^{(2)}$$
	with
	$$
	\begin{aligned}
		\tilde{E}_l^{(2)}=&\mathbb{D}^- - \left(\overline{S_l(q_-)}-S_l(\tilde{q}_-)\right)\\
		=&\,\overline{-\frac 1 {6l}\,(q+l)^3\,\mathds{1}_{\{-2l<q<-l\}}+\frac 1 6\left(3q^2+9lq+7l^2\right)\,\mathds{1}_{\{q\le -2l\}}}\\
		&-\left(-\frac 1 {6l}\,(\tilde{q}+l)^3\,\mathds{1}_{\{-2l<\tilde{q}<-l\}}+\frac 1 6\left(3\tilde{q}^2+9l\tilde{q}+7l^2\right)\,\mathds{1}_{\{\tilde{q}\le -2l\}}\right).
	\end{aligned}
	$$
	
	It is Obvious that $v\mapsto -\frac 1 {6l}\,(v+l)^3\,\mathds{1}_{\{-2l<v<-l\}}+\frac 1 6\left(3v^2+9lv+7l^2\right)\,\mathds{1}_{\{v\le -2l\}}$ is nonnegative and convex. According to \eqref{S_l'} and \eqref{S_l of q exact}, $S_l(\tilde{q}_-)'\ge -\frac 3 2\,l$, then
	$$
	\frac 1 2 S_l(\overline{q^2}-\tilde{q}^2)\,\left(\overline{q^2}-\tilde{q}^2\right)\ge -\frac {3l} 2\,\left(\mathbb{D}^+ + \mathbb{D}_l^-\right)-\frac {3l} 2\,\tilde{E}_l^{(2)}.$$
	
	Combining above inequality with
	$$
	\overline{H_{l,-}(q)}-H_{l,-}(\tilde{q})=-\frac 1 {12l}(v^4-3l^2\,v^2-2l^3\,v)\mathds{1}_{\{-2l<v<-l\}}-\frac 1{12}(9l\,v^2+14l^2 v)\mathds{1}_{\{v\le -2l\}},$$
	we obtain that
	$$I_l^{(2)}\ge -\frac{3l} 2\left(\mathbb{D}^+ + \mathbb{D}_l^-\right)+E_l^{(2)},$$
	where
	$$
	\begin{aligned}
		E_l^{(2)}=&\ \overline{\frac 1 {12l}\left(q^4+3lq^3+6l^2q^2+7l^3q+3l^4\right)\mathds{1}_{\{-2l<q<-l\}}-\frac 1 {12}\left(13l^2q+21l^3\right)\mathds{1}_{\{q<-2l\}}}\\
		&-\left(\frac 1 {12l}\left(\tilde{q}^4+3l\tilde{q}^3+6l^2\tilde{q}^2+7l^3\tilde{q}+3l^4\right)\mathds{1}_{\{-2l<\tilde{q}<-l\}}-\frac 1 {12}\left(13l^2\tilde{q}+21l^3\right)\mathds{1}_{\{\tilde{q}<-2l\}}\right).
	\end{aligned}
	$$
	
	\noindent$\circ$ The term $I_l^{(3)}$.
	
	Similarly, we can compute that
	$$\frac 1 2S_l(v_-)''v^2=S_l(v_-)+\left(\frac 1 {3l}\,v^3-\frac 1 2\,l\,v-\frac 1 6 l^2\right)\,\mathds{1}_{\{-2l<v<-l\}}+\left(\frac 3 2 l\,v+\frac 7 6l^2\right)\,\mathds{1}_{\{v\le -2l\}},$$
	thus,
	$$I_l^{(3)}=-\mathbb{D}_l^-+ E_l^{(3)}$$
	with 
	$$\begin{aligned}
		E_l^{(3)}=-&\ \overline{\left(\frac 1 {3l}\,q^3-\frac 1 2\,l\,q-\frac 1 6 l^2\right)\,\mathds{1}_{\{-2l<q<-l\}}+\left(\frac 3 2 l\,q+\frac 7 6l^2\right)\,\mathds{1}_{\{q\le -2l\}}}\\
		+&\left[\left(\frac 1 {3l}\,\tilde{q}^3-\frac 1 2\,l\,\tilde{q}-\frac 1 6 l^2\right)\,\mathds{1}_{\{-2l<\tilde{q}<-l\}}+\left(\frac 3 2 l\,\tilde{q}+\frac 7 6l^2\right)\,\mathds{1}_{\{\tilde{q}\le -2l\}}\right].
	\end{aligned}$$
	
	\noindent$\circ$ The term $I_l^{(4)}$.
	
	By \eqref{S_l'}, $|S_l(v_-)'\,v|\lesssim_l |v|$, according to \eqref{SJ convergence}, for each fixed $l$, $I_L^{(4)}\in L^{2r}_{\tilde{\omega},t,x}$ with $2r>2$. Therefore, by the It\'{o} isometry, the stochastic integral
	$$\mathcal{M}_l^-(t)=\int_0^t\int_{\mathbb{S}} I_l^{(4)}\,{\rm d}x\,{\rm d}\tilde{W}$$
	is a square-integrable martingale on $[0,T]$.
	
	Combining the above inequalities together, we deduce that
	$$
			\sum_{i=1}^3 I_l^{(i)}\ge (\overline{q_-}-\tilde{q}_-)\left(\tilde{P}-\tilde{u}^2\right)-\frac{3l} 2\left(\mathbb{D}^+ + \mathbb{D}_l^-\right)-\mathbb{D}_l^-+E_l^{(1)}\left(\tilde{P}-\tilde{u}^2-\tilde{P}_{n_0}+\tilde{u}_{n_0}^2\right)\\
			+(\overline{G_l(q)}-G_l(\tilde{q})),
	$$
	where
	$$
		\overline{G_l(q)}-G_l(\tilde{q})
		=
		-E_l^{(1)}\left\Vert\tilde{P}_{n_0}-\tilde{u}_{n_0}^2\right\Vert_{L^\infty([0,T]\times\mathbb{S})}+E_l^{(2)}+ E_l^{(3)}.
	$$
	
	By the explicit expression of $E_l^{(i)}$ for $i=1,2,3$, $v\mapsto G_l(v)\in C^1(\mathbb{R})$, then we can directly calculate that
	$$G_l''(v)=\frac 1 l\left(-\left\Vert\tilde{P}_{n_0}-\tilde{u}_{n_0}^2\right\Vert_{L^\infty([0,T]\times\mathbb{S})}+v^2+\frac 3 2 lv-2v+l^2\right)\,\mathds{1}_{\{-2l<v<-l\}},$$ 
	and so on $A_L^{n_0}$, let $l$ sufficiently large, then for any $v\in\mathbb{R}$ and a.e. $x\in\mathbb{S}$,
	$$G_l''(v)\ge\frac 1 l\left[v^2+\left(\frac 3 2 l-2\right)v+l^2-L\right]\,\mathds{1}_{\{-2l<v<-l\}}\ge 0,$$
	hence, $G_l(v)$ is convex for sufficiently large $l$. The standard result on the weak convergence of convex function sequences implies that on $A_L^{n_0}$,
	$$\overline{G_l(q)}-G_l(\tilde{q})\le 0\quad\text{a.e. on } [0,T]\times\mathbb{S}.$$
	
	In conclusion, for a sufficiently large $l=l(L)$,
	\begin{equation}\label{I123 inequality}
		\sum_{i=1}^3 I_l^{(i)}\ge (\overline{q_-}-\tilde{q}_-)\left(\tilde{P}-\tilde{u}^2\right)-\frac{3l} 2\left(\mathbb{D}^+ + \mathbb{D}_l^-\right)-\mathbb{D}_l^-+E_l^{(1)}\left(\tilde{P}-\tilde{u}^2-\tilde{P}_{n_0}+\tilde{u}_{n_0}^2\right).
	\end{equation}
	
	Inserting \eqref{I123 inequality} into \eqref{inequality Dl- dt} multiplied by $\mathds{1}_{A_L^{n_0}}$, we deduce that
	\begin{equation}\label{inequality Dl- dt on AL}
		\begin{aligned}
		&\frac {\rm d}{{\rm d}t}\int_{\mathbb{S}}\mathbb{D}_l^-\,{\rm d}x+\int_{\mathbb{S}}\left(\overline{q_-}-\tilde{q}_-\right)\,\left(\tilde{P}-\tilde{u}^2\right)\,{\rm d}x\\
		-&\int_{\mathbb{S}} \frac {3l} 2\,(\mathbb{D}_l^- +\mathbb{D}^+)+\mathbb{D}_l^-\,{\rm d}x
		+\int_{\mathbb{S}}E_l^{(1)}\left(\tilde{P}-\tilde{u}^2-\tilde{P}_{n_0}+\tilde{u}_{n_0}^2\right)\,{\rm d}x\\
		+&\int_{\mathbb{S}}I_l^{(5)}\,{\rm d}x\dot{\tilde{W}},\quad\text{in }\mathcal{D}'([0,T)),\text{ a.s. on }A_l^{n_0},
		\end{aligned}
	\end{equation}
	with zero initial data. Following the same argument used in the proofs of Lemma \ref{continuity t=0} and Lemma \ref{positive part of defect measure}, we can obtain that \eqref{inequality Dl-} holds a.e. in $A_L^{n_0}\times[0,T]$.
	
	Finally, according to the Chebyshev inequality and Lemma \ref{priori estimate}, we can easily deduce that
	\begin{equation}\label{probability AL^c}
		\tilde{P}\left(\left(A_L^{n_0}\right)^c\right)\le\frac 1 L\,\tilde{\mathbb{E}}\left\Vert\tilde{P}_{n_0}-\tilde{u}_{n_0}^2 \right\Vert_{L^\infty([0,T]\times\mathbb{S})}\rightarrow 0,
	\end{equation}
	therefore, the above inequality yields that $\tilde{\mathbb{P}}\left(A_L^{n_0}\right)\rightarrow 1$ as $L\rightarrow\infty$.
\end{proof}

According to above estimates of the positive part of defect measure $\mathbb{D}^+$ and the approximation of negative part $\mathbb{D}_l^-$, we can identify the weak limit $\overline{q^2}$ with $\tilde{q}^2$.

\begin{proof}[Proof of Theorem \ref{main theorem}]
	Notice that $\tilde{q}=\tilde{q}_+ +\tilde{q}_-=\overline{q_+}+\overline{q_-}=\overline{q}$.
	According to Lemma \ref{negative part of defect measure}, fixed $n_0$ and consider sufficiently large $l$ depending on fixed $L$, combining with \eqref{inequality D+ integral} and take an expectation, we obtain that
	\begin{equation}\label{inequality D+ + Dl-}
		\begin{aligned}
			&\tilde{\mathbb{E}}\int_{\mathbb{S}}\mathds{1}_{A_L^{n_0}}\left(\mathbb{D}^+ + \mathbb{D}_l^-\right)(t)\,{\rm d}x
			-\tilde{\mathbb{E}}\int_0^t\int_{\mathbb{S}}\mathds{1}_{A_L^{n_0}}\,\left(\frac {3l} 2+1\right)\,\left(\mathbb{D}^+ +\mathbb{D}_l^-\right)\,{\rm d}x\,{\rm d}s\\
			\le& -\tilde{\mathbb{E}}\int_0^t\int_{\mathbb{S}}\mathds{1}_{A_L^{n_0}}\,\left( \overline{S_l(q_-)'-q_-} - \left(S_l(\tilde{q}_-)' - \tilde{q}_-\right) \right)\,\left(\tilde{P}-\tilde{u}^2-\tilde{P}_{n_0}+\tilde{u}_{n_0}^2\right)\,{\rm d}x\,{\rm d}s.
		\end{aligned}
	\end{equation}
	Applying the Gronwall inequality to \eqref{inequality D+ + Dl-}, for a.e. $t\in[0,T]$,
	$$
	\begin{aligned}
		&\tilde{\mathbb{E}}\left(\mathds{1}_{A_L^{n_0}}\,\left\Vert\mathbb{D}^+(t)+\mathbb{D}_l^-(t)\right\Vert_{L^1(\mathbb{S})}\right)\\
		\le& e^{(3l/2+1)t}\,\left|\,\tilde{\mathbb{E}}\int_0^t\int_{\mathbb{S}}\mathds{1}_{A_L^{n_0}}\,\left( \overline{S_l(q_-)'-q_-} - \left(S_l(\tilde{q}_-)' - \tilde{q}_-\right) \right)\,\left(\tilde{P}-\tilde{u}^2-\tilde{P}_{n_0}+\tilde{u}_{n_0}^2\right)\,{\rm d}x\,{\rm d}s\right|,
	\end{aligned}
	$$
	integrating over $[0,T]$ and adding $\tilde{\mathbb{E}}\int_0^T\mathds{1}_{\left(A_L^{n_0}\right)^c}\,\left\Vert\mathbb{D}^+(t)+\mathbb{D}_l^-(t)\right\Vert_{L^1(\mathbb{S})}\,{\rm d}t$ to both sides,
	\begin{equation}\label{inequality D+ + Dl- 2}
		\begin{aligned}
		&\tilde{\mathbb{E}}\int_0^T\left\Vert\mathbb{D}^+(t)+\mathbb{D}_l^-(t)\right\Vert_{L^1(\mathbb{S})}\,{\rm d}t\\
		\le\ & \tilde{\mathbb{E}}\int_0^T\mathds{1}_{\left(A_L^{n_0}\right)^c}\,\left\Vert\mathbb{D}^+(t)+\mathbb{D}_l^-(t)\right\Vert_{L^1(\mathbb{S})}\,{\rm d}t\\
		&+T\,e^{(3l/2+1)T}\,\tilde{\mathbb{E}}\int_0^T\int_{\mathbb{S}}\left| \overline{S_l(q_-)'-q_-} - \left(S_l(\tilde{q}_-)' - \tilde{q}_-\right) \right|\,\left|\tilde{P}-\tilde{u}^2-\tilde{P}_{n_0}+\tilde{u}_{n_0}^2\right|\,{\rm d}x\,{\rm d}s.
		\end{aligned}
	\end{equation}
	
	Lemma \ref{priori estimate} yields
	$$
	\Vert\mathbb{D}^+ +\mathbb{D}_l^-\Vert_{L^r_{\tilde{\omega},t,x}},\ \left\Vert \tilde{P}-\tilde{u}^2-\tilde{P}_{n_0}+\tilde{u}_{n_0}^2\right\Vert_{L^p_{\tilde{{\omega}},t,x}},\ \left\Vert\overline{S_l(q_-)'-q_-} - \left(S_l(\tilde{q}_-)' - \tilde{q}_-\right)\right\Vert_{L^{2r}_{\tilde{{\omega}},t,x}}\lesssim1
	$$
	for fixed $r\in(1,3/2)$ and $p\in[1,p_0/2]$. By the H\"{o}lder inequality, \eqref{probability AL^c}, and \eqref{SJ convergence}, we can deduce that
	$$\tilde{\mathbb{E}}\int_0^T\mathds{1}_{\left(A_L^{n_0}\right)^c}\,\left\Vert\mathbb{D}^+(t)+\mathbb{D}_l^-(t)\right\Vert_{L^1(\mathbb{S})}\,{\rm d}t\lesssim \Vert\mathbb{D}^+ +\mathbb{D}_l^-\Vert_{L^r_{\tilde{\omega},t,x}}\,\left(\tilde{\mathbb{P}}((A_L^{n_0})^c)\right)^{1/r'}\lesssim L^{-1/r'},$$
	and
	$$
	\begin{aligned}
	&\tilde{\mathbb{E}}\int_0^T\int_{\mathbb{S}}\left| \overline{S_l(q_-)'-q_-} - \left(S_l(\tilde{q}_-)' - \tilde{q}_-\right) \right|\,\left|\tilde{P}-\tilde{u}^2-\tilde{P}_{n_0}+\tilde{u}_{n_0}^2\right|\,{\rm d}x\,{\rm d}s\\[2ex]
	\le\ &  \left\Vert\overline{S_l(q_-)'-q_-} - \left(S_l(\tilde{q}_-)' - \tilde{q}_-\right)\right\Vert_{L^{2r}_{\tilde{{\omega}},t,x}}\,\left\Vert\tilde{P}-\tilde{u}^2-\tilde{P}_{n_0}+\tilde{u}_{n_0}^2\right\Vert_{L^p_{\tilde{{\omega}},t,x}} \xrightarrow{n_0\rightarrow\infty} 0.
	\end{aligned}
	$$
	Inserting above inequality into \eqref{inequality D+ + Dl- 2} and sending $n_0$ to $\infty$,
	\begin{equation}\label{inequality D+ + Dl- 3}
		\tilde{\mathbb{E}}\int_0^T\int_{\mathbb{S}}\mathbb{D}^+ +\mathbb{D}_l^-\,{\rm d}x\,{\rm d}t\lesssim\ L^{-1/r'},
	\end{equation}
	utilizing Lemma \ref{convergence Dl-} and sending $l$ and $L$ to $\infty$ successively in \eqref{inequality D+ + Dl- 3}, we can conclude that
	$$\overline{q_+^2}=\tilde{q}_+^2,\quad\overline{q_-^2}=\tilde{q}_-^2\quad\text{a.e. in }\tilde{\Omega}\times[0,T]\times\mathbb{S},$$
	with Proposition \ref{SPDE tilde u}, the existence of $H^1$ weak martingale solution to \eqref{CH u} is proved.
	
	According to Proposition \ref{one side estimate} and the equality of laws $\partial_x\tilde{u}_n$ and $\partial_xu_\varepsilon$, we can deduce that
	$$\tilde{\mathbb{P}}\left\{\lim_{t\rightarrow\infty}\partial_x \tilde{u}_n\le0,\ \forall x\in\mathbb{S}\right\}=1.$$
	The convergence \eqref{SJ convergence} yields that $\tilde{q}_n\rightarrow\tilde{q}$ a.e. in $\tilde{\Omega}\times[0,T]\times\mathbb{S}$, then we can establish \eqref{main property one side}.
	
	The estimate \eqref{main property space time higher} is a direct consequence of Lemma \ref{priori estimate}.
	
	Finally, according to Proposition \ref{viscous estimate} and the equality of laws similarly, we have for a.e. $\tilde{{\omega}}\in\tilde{{\Omega}}$,
	\begin{equation}\label{large time un}
		\Vert \tilde{u}_n(\tilde{{\omega}},t)\Vert_{L^\infty(\mathbb{S})}\lesssim\Vert \tilde{u}_n(\tilde{{\omega}},t)\Vert_{H^1(\mathbb{S})}\lesssim \tilde{\eta}_n(\tilde{{\omega}},t)\Vert \tilde{u}_0\Vert_{H^1 (\mathbb{S})},
	\end{equation}
	where $\tilde{\eta}_n(\tilde{{\omega}},t)=e^{\tilde{W}_n(t)-\frac t 2}$. According to \eqref{SJ convergence}, $\tilde{W}_n\rightarrow\tilde{W}$ in $C_t$, and it is easy to observe that for a.e. $\tilde{\omega}\in\tilde{\Omega}$, $\lim_{t\rightarrow\infty}\tilde{\eta}_n(\tilde{{\omega}},t)=0$ and $\lim_{t\rightarrow\infty}\tilde{\eta}(\tilde{{\omega}},t)=0$. Therefore, for any $t\ge 0$, a.e. $\tilde{{\omega}}\in\tilde{\Omega}$, $|\tilde{\eta}_n(\tilde{{\omega}},)t|+|\tilde{\eta}(\tilde{{\omega}},)t|\le C(\tilde{\omega})$ and the Constant $C$ independent of $n$.
	
	Due to Proposition \ref{viscous estimate Holder} and the Aubin-Lions lemma, similar to the proof of Lemma \ref{tightness}, up to a subsequence, we can assume $\tilde{u}_n\rightarrow\tilde{u}$ in $C([0,T];L^\infty(S))$.

	Thus, when $n$ tends to $\infty$ in \eqref{large time un}, we deduce that for a.e. $\tilde{{\omega}}\in\tilde{{\Omega}}$,
	$$\Vert \tilde{u}(\tilde{{\omega}},t)\Vert_{L^\infty(\mathbb{S})}\lesssim \tilde{\eta}(\tilde{{\omega}},t)\Vert \tilde{u}_0\Vert_{H^1 (\mathbb{S})},$$
	hence, sending $t$ to $\infty$, we can conclude that
	$$\lim_{t\rightarrow\infty}|\tilde{u}(t,x)|=0,\quad\forall\ x\in\mathbb{S},\ \tilde{\mathbb{P}}-{\rm a.s.}$$
	Hence, Theorem \ref{main theorem} is proved.
\end{proof}

\appendix
\section{Appendix}\label{appendix}
\begin{defi}[Quasi-Polish space]\label{quasi Polish}
	A topological space $(Z,\tau)$ is quasi-Polish if there exists a sequence $\{f_n\}_{n\in\mathbb{N}}$ of continuous functions $f_n:Z\rightarrow[-1,1]$ separating points of $Z$.
\end{defi}

Consider the classical Bochner space of equivalence classes of measurable functions $z:[0,T]\rightarrow L^{p_2}(\mathbb{S})$ for which $\|z(\cdot)\|_{L^{p_2}(\mathbb{S})}\in L^{p_1}([0,T])$. Equipping this space with the locally convex topology generated by the seminorms
\begin{equation}\label{weak topology}
	L^{p_1}(L^{p_2})\ni z\mapsto\left(\int_0^T\left|\int_{\mathbb{S}}\phi(x)z(t,x)\,{\rm d}x\right|^{p_1}\right)^{1/p_1},\quad \phi\in L^{{p_2}'}(\mathbb{S}),
\end{equation}
where $\frac 1 {p_2}+\frac 1{{p_2}'}=1$, we denote the resulting topology space by $L^{p_1}(L^{p_2}_w)$. Moreover, $L^{p_1}(L^{p_2}_w)$ is a quasi-Polish space \cite{Holden2024JDE}.

\begin{lemm}\label{quasi Polish product}
	Let $\{Z_i\}_{i\in\mathbb{N}}$ be countable collection of quasi-Polish spaces, then $Z=\prod_{i\in\mathbb{N}} Z_i$ endowed with the product topology, is a quasi-Polish space.
\end{lemm}

\begin{lemm}[The quasi-Polish version of the KLS theorem]\label{KLS}
	Let $Z$ be a quasi-Polish space and let $Y$ be a Polish space for which exists a continuous injection $b:X\rightarrow Z$. For any Borel set $B\subset Y$, the set $b[B]$ is Borel in $Z$.
\end{lemm}

\begin{lemm}[Vitali's convergence theorem]\label{Vitali}\cite{Kuo2006}
	Let $p\in[1,\infty)$, $X_n\in L^p$ and $X_n$ converge to $X$ in probability. Then the following are equivalent:
	\begin{enumerate}
		\item $\lim_{n\rightarrow\infty}X_n=X$ in $L^p$;
		\item $|X_n|^p$ is uniformly integrable;
		\item  $\lim_{n\rightarrow\infty}\mathbb{E}(|X_n|^p)=\mathbb{E}(|X|^p)$.
	\end{enumerate}
	In particular, if $\sup_n\mathbb{E}(|X_n|^p)<\infty$ for some $p<q<\infty$, or if there exists a $Y\in L^p$ such that $|X_n|^p\le Y$ for all $n$, then the above properties hold true.
\end{lemm}

\begin{theo}[Prokhorov theorem]\label{Prokhorov theorem}\cite{Prato2014}
	Let $X$ be a complete, separable metric space. A sequence of probabilty measures $\{\mu_n\}\subset\mathcal{P}_\gamma(X)$ is tight if and only if it is relatively compact, i.e., there is a subsequence $\{\mu_{n_k}\}$ which converges weakly to a probability measure $\mu$ on $X$.
\end{theo}

\begin{theo}[Skorokhod-Jakubowski a.s. representation]\label{SJ reprensentation theorem}\cite{Jakubowski1997}
	Let $(Z, \tau, \mathcal{B}_\tau)$ be a quasi-Polish space, and denote by $\Sigma_f \subset \mathcal{B}_\tau$ the $\sigma$-algebra generated by the sequence $\{f_n\}$ of continuous functions that separate points. Then  
	\begin{enumerate}
		\item every $\tau$-compact subset of $Z$ is metrisable;
		\item every Borel subset of a sigma compact set in $Z$ belongs to $\Sigma_f$;
		\item every probability measure supported by a sigma compact set in $Z$ has a unique Radon extension to the Borel $\sigma$-algebra $\mathcal{B}_\tau = \mathcal{B}(Z)$.
	\end{enumerate}
	Moreover, if $\{\mu_l\}$ is a tight sequence of probability measures on $(Z, \Sigma_f)$, then there exist a subsequence $\{l_k\}_k$, a probability space $(\tilde{\Omega}, \tilde{\mathcal{F}}, \tilde{\mathbb{P}})$, and Borel measurable $Z$-valued random variables $\tilde{v}_k$, $\tilde{v}$, such that $\mu_{n_k}$ is the law of $\tilde{v}_k$ and $\tilde{v}_k \to \tilde{v}$ $\tilde{\mathbb{P}}$-a.s. in $Z$. Besides, the law $\mu$ of $\tilde{v}$ is a Radon measure on $\mathcal{B}_\tau$.
\end{theo}

In order to claim that Skorokhod-Jakubowski representations satisfy corresponding SPDEs, it is necessary to obtain convergence result of stochastic integrals, we use the following lemma:
\begin{lemm}[Convergence of stochastic integrals]\label{convergence of stochastic integrals}\cite{Temam2011}
	Fix a probability space $(\Omega, \mathcal{F}, \mathbb{P})$. For each $n \in \mathbb{N}$, consider a stochastic basis 
	$\mathcal{S}^n = (\Omega, \mathcal{F}, \{\mathcal{F}_t^n\}_{t\in[0,T]}, \mathbb{P})$, a Wiener process $W^n$ on $\mathcal{S}^n$, 
	and a predictable $L^2(\mathbb{S})$-valued process $G^n$ on $\mathcal{S}^n$ satisfying 
	$G^n \in L^2([0, T]; L^2(\mathbb{S}))$, $\mathbb{P}$-{\rm a.s.}. Suppose there is a stochastic basis 
	$\mathcal{S} = (\Omega, \mathcal{F}, \{\mathcal{F}_t\}_{t\in[0,T]}, \mathbb{P})$, a Wiener process $W$ on $\mathcal{S}$, 
	and a predictable $L^2(\mathbb{S})$-valued process $G$ on $\mathcal{S}$ with 
	$G \in L^2((0, T); L^2(\mathbb{S}))$ $\mathbb{P}$-almost surely, such that 
	$$
	W^n \xrightarrow[]{n \uparrow \infty} W \ \text{in } C([0, T]), \quad 
	G^n \xrightarrow[]{n \uparrow \infty} G \quad \text{in } L^2([0, T]; L^2(\mathbb{S})), \text{ in probability}.
	$$
	Then
	$$
	\int_0^t G^n \, dW^n \xrightarrow[]{n \uparrow \infty} \int_0^t G \, {\rm d}W \quad \text{in } L^2([0, T]; L^2(\mathbb{S})), \text{ in probability}.
	$$
\end{lemm}

By above convergence result, we can construct martingale solutions, in order to upgrade martingale solutions to pathwise solution, we will need the uniqueness of martingale solutions and the following Gy\"{o}ngy-Krylov characterization of convergence in probability:
\begin{theo}[Gy\"{o}ngy-Krylov]\label{Gyongy-Krylov}\cite{Gyongy1996}
	Let $\mathcal{X}$ be a Polish space. For a sequence $\{v_n\}$ of $\mathcal{X}$-valued random variables, define the joint probability laws $\{\mu_{m,n}\}_{m,n}$ by setting, for all $A \in \mathcal{B}(\mathcal{X} \times \mathcal{X})$,
	$$
	\mu_{m,n}(A) := \mathbb{P}(\{(v_m, v_n) \in A\}).
	$$
	Then the sequence $\{v_n\}$ converges in probability if and only if for every subsequence $\{\mu_{m_k, n_k}\}_k$, there exists a further subsequence that converges weakly to a probability measure $\mu$ supported on the diagonal:
	$$
	\mu(\{(v, w) \in \mathcal{X} \times \mathcal{X} : v = w\}) = 1.
	$$
\end{theo}

However, due to the lack of regularity estimates for continuity in time, we can only obtain the tightness of a family of probability laws on a quasi-Polish space, in which case the Gy\"{o}ngy-Krylov theorem \cite{Prato2014} cannot be applied directly. This problem can be overcome by the projection mapping that separates points in the Definition \ref{quasi Polish} of a quasi-Polish space.
\begin{lemm}\label{map of quasi Polish}\cite{Endelking1975,Jakubowski1997}
	Assume $Z$ is a quasi-Polish space, and $\{f_n\}_{n\in\mathbb{N}}$ is the sequence of continuous functions defined in Definition \ref{quasi Polish}. Define map $f:Z\rightarrow[-1,1]^{\mathbb{N}}$ by
	$$f:u\mapsto\{f_n(u)\}_{n\in\mathbb{N}}.$$
	The map $f$ is a measurable bijective function when restricted to a $\sigma$-compact subspace of $Z$.
\end{lemm}

\begin{lemm}[Nonlinear compositions' a.s. representations]\label{representation nonlinear}\cite{Holden2024JDE}
	Let $Z, W$ be quasi-Polish spaces, and suppose $F : Z \to W$ is a Borel function. Consider a sequence $\{Y_j\}_{j=1}^{\infty}$ of $Z$-valued random variables on $(\Omega, \mathcal{F}, P)$. Then there exist a subsequence $\{Y_{j_k}\}_{k\in\mathbb{N}}$ and $Z$-valued random variables $V_0, V_1, V_2, \dots$, defined on $([0,1], \mathcal{B}([0,1]), \operatorname{Leb})$, where $\operatorname{Leb}$ is the Lebesgue measure, such that
	$$
	Y_{j_k} \sim V_k,\quad k\in\mathbb{N},\qquad 
	V_k(\xi) \xrightarrow[]{k\uparrow\infty} V_0(\xi) \quad \text{for a.e. } \xi\in[0,1].$$
	Denote by $(V_k, \tilde{F}_k)$ the a.s.\ representations of $(Y_j, F(Y_j))$. Then $\tilde{F}_k = F(V_{n_k})$, a.s., $k \in \mathbb{N}$.
\end{lemm}

	\smallskip
	\noindent\textbf{Acknowledgments}
	
This work was partially supported by the National Key R\&D Program of China(No. 2021YFA10021
00) and the National Natural Science Foundation of China (No.12171493 and No. 12471233).
	
	\smallskip
	\noindent\textbf{Conflict of interest}
	
	The authors have no conflicts to disclose.
	
	\smallskip
	\noindent\textbf{Data Availability}
	
	The data that support the findings of this study are available from the corresponding author upon reasonable request.

	\phantomsection
	\addcontentsline{toc}{section}{\refname}
	\bibliographystyle{abbrv} 
	\bibliography{SCHref}
\end{document}